\documentclass{amsart}

\usepackage[utf8]{inputenc}
\usepackage[dvips]{graphicx}
\usepackage{amssymb}
\usepackage{latexsym}
\usepackage{xypic}
\usepackage{multirow}
\usepackage{color}

\newtheorem{theorem}{Theorem}[section]
\newtheorem{lemma}[theorem]{Lemma}
\newtheorem{proposition}[theorem]{Proposition}

\newtheorem{corollary}[theorem]{Corollary}

\theoremstyle{definition}
\newtheorem{definition}[theorem]{Definition}
\newtheorem{example}[theorem]{Example}

\theoremstyle{remark}
\newtheorem{remark}[theorem]{Remark}

\definecolor{lav}{rgb}{0.80,0.0,0.95}

\numberwithin{equation}{section}

\def\N{\bf{N}}

\def\C{{\mathbb C}}
\def\R{{\mathbb R}}
\def\Z{{\mathbb Z}}
\def\N{{\mathbb N}}
\def\H{{\mathbb H}}

\def\P{\mathbb P_{\mathbb C}}
\def\dP{{\check {\mathbb P}_\C}}

\def\PSL{\text{{PSL}}}
\def\SL{\text{{SL}}}
\def\PU{\text{{PU}}}

\def\SP{\text{{SP}}}

\def\GL{\text{{GL}}\,}
\def\GL{\text{{GL}}\,}
\def\Kul{\text{{Kul}}\,}
\def\Heis{\text{{Heis}}\,}

\title[Purely parabolic  groups]{\sc{ Discrete parabolic groups in $\PSL(3, \C)$ } }

\author{Waldemar Barrera}
\address{Facultad de Matem\'aticas, Universidad Aut\'onoma de Yucat\'an, M\'exico}
\email{bvargas@correo.uady.mx }

\author{Angel Cano}
\address{Instituto de Matem\'aticas, Unidad Cuernavaca, Universidad Nacional Aut\'onoma de M\'exico.}
\email{angelcano@im.unam.mx}

\author{Juan Pablo Navarrete}
\address{Facultad de Matem\'aticas, Universidad Aut\'onoma de Yucat\'an, M\'exico}
\email{jp.navarrete@correo.uady.mx}

\author{Jos\'e Seade}
\address{ Instituto de Matemáticas, Universidad Nacional Autónoma de México, and  CNRS (France) IRL Laboratorio Solomon Lefschetz, México. }
\email{jseade@im.unam.mx}

\subjclass{Primary 37F99,  Secondary 30F40, 20H10, 57M60, 32H50, 32Q30}

\begin{document}
\maketitle

			We study and classify the purely parabolic discrete subgroups of $\PSL(3,\C)$. This includes all discrete subgroups of the Heisenberg group ${\rm Heis}(3,\C)$.
			While  for  $\PSL(2,\C)$   every purely parabolic subgroup is Abelian and  acts on $\P^1$ with limit set a single point, the case of $\PSL(3,\C)$ is far more subtle and intriguing.   We show that there are  five  families of purely parabolic discrete groups in $\PSL(3,\C)$, and     some of these actually split into subfamilies. We classify all these by means of their limit set and the control group. We use  first the  Lie-Kolchin Theorem and Borel's fixed point theorem to show that all purely parabolic discrete groups in $\PSL(3,\C)$ are virtually triangularizable. Then we prove that purely parabolic  groups  in $\PSL(3,\C)$ are virtually solvable and polycyclic, hence finitely presented. We then prove a slight generalization  of the Lie-Kolchin Theorem for these groups: they are either virtually unipotent or else Abelian of rank 2 and of a very special type. All the virtually unipotent ones turn out to be conjugate to subgroups of the Heisenberg group ${\rm Heis}(3,\C)$. We classify these using the obstructor dimension introduced by Bestvina, Kapovich and Kleiner. 
			We find that their Kulkarni limit set  is either a projective line, a cone of lines  with base a  circle or else the whole $\P^2$. We determine the relation with the Conze-Guivarc'h limit set of the action on the dual projective space $\dP^2$ and we  show that in all cases the Kulkarni region of discontinuity is the largest open set where the group acts properly discontinuously. 

	\maketitle 
	

	\section*{Introduction}
%
	
	Henri Poincar\'e introduced in \cite{Poincare}  the concept of Kleinian groups, {\it i.e. discrete subgroups of the M\"obius group }M\"ob$(2,\C)$, which is isomorphic to $\PSL(2,\C)$, and he classified its elements  into three types.  Poincar\'e's  classification can be stated by saying that        $g \in \PSL(2,\C)$  is elliptic if it has a lift $\tilde g$  
	 to $\SL(2,\C)$ which is   diagonalizable and its eigenvalues are all unitary; $g$  is
 parabolic if   $\tilde g$
	 is non-diagonalizable and its eigenvalues are all unitary, and $g$ is  loxodromic  otherwise.  That classification, with these same definitions,  extends  to  $\PSL(n+1,\C)$ for all $n \ge 2$,  see  \cite{CLU, CNS, Nav2}.
	 
 	In this work we look at  $\PSL(3, \C)$, the group of automorphisms of the projective plane $\P^2$, and we classify its discrete subgroups that (besides the identity)  have only parabolic elements. These are called {\it purely parabolic groups}.  	In order to describe our results we remark that every 	purely parabolic discrete group in $\PSL(3,\C)$ has a global fixed point $p \in \P^2$ and therefore   $\P^2 \setminus \{p\}$, one has a canonical  holomorphic projection map $\pi$ from $\P^2 \setminus \{p\}$ into $\ell \cong \P^1$. This defines a group homomorphism:
		\begin{displaymath}
		\begin{matrix}
		\Pi = \Pi_{p,\ell,G} : \PSL(3,\C) \rightarrow Bihol(\ell)\cong \PSL(2,\C) \;, \\
		\Pi(g)(x) = \pi(g(x))
	\end{matrix}
	\end{displaymath}
	which is independent of $\ell$ up to conjugation. Its restriction to $G$ is the control morphism of $G$ and  its image $\Pi(G)$ is the control group (see  Definition \ref{d:proyeccion} and \cite{CNS}).

 The various types of  purely parabolic groups in $\PSL(3, \C)$ are fully described in Section \ref{s:examples}, where we also describe their algebraic and dynamical properties. There are five such main families, these are:	
\begin{itemize}
\item Elliptic groups. These are the only ones that are not conjugate to subgroups of the Heisenberg group ${\rm Heis} (3,\C)$ and they are   subgroups of  fundamental groups of elliptic surfaces (see \cite{BCNS3}).
\item Torus groups. These are subgroups of fundamental groups of complex tori:
 \[
			\mathcal{T}(\mathfrak L)=
			\left 
			\{
			\left[
			\begin{array}{lll}
				1 & 0 & a\\
				0 & 1 & b\\
				0 & 0 & 1\\ 
			\end{array}
			\right]:(a,b)\in \mathfrak L
			\right \} \;, \] 
where $ \mathfrak L$ is an additive discrete subgroup of $\C^2$.
\item Dual torus groups, 
 \[
			\mathcal{T}^*(\mathfrak L)=
			\left 
			\{
			\left[
			\begin{array}{lll}
				1 & a & b\\
				0 & 1 & 0\\
				0 & 0 & 1\\ 
			\end{array}
			\right]:(a,b)\in \mathfrak L
			\right \} \;. \] 
These split into three types: the first   have  Kulkarni limit set (see Definition  \ref{def Kulkarni limit set})  a complex projective line; the second have  Kulkarni limit set a cone of projective lines over a circle, and the third type have all  $\P^2$ as Kulkarni limit set.  From now on,  for simplicity, we will say only  limit set instead Kulkarni limit set,   unless we specify otherwise. 

\item Inoue groups of three types. The first have limit set a cone of lines over a circle,  the others have limit set all of $\P^2$.  The three types can be distinguished by the limit set and the control morphism  (see Proposition \ref{Prop Types Inoue groups}).


\item Kodaira groups ${\mathcal{K}_0}$, which are Abelian,  and their extensions. There are five types of extensions ${\mathcal{K}_i}$, $i = 1, \cdots , 5$, 
which are 
purely parabolic and discrete. 
The first type ${\mathcal{K}_1}$ have limit set a projective line; the second type ${\mathcal{K}_2}$ have limit set a cone of projective lines over the circle. The remaining three types have limit set all of $\P^2$ and they are distinguished by their control morphism (see details in Section \ref{s:examples}).

\end{itemize}

	In this work we prove:


	

	\begin{theorem} \label{t:main1}
		Let $G\subset {\rm PSL} (3,\C)$ be a  purely parabolic discrete subgroup. Then:
\begin{enumerate} 		
		\item $G$  is  either virtually elliptic  or virtually   conjugate to a subgroup of the Heisenberg group ${\rm Heis} (3,\C)$, which is itself purely parabolic.
		
		\item Its Kulkarni limit set $\Lambda_{\Kul}$ is either a line, a cone of lines over a circle, or the whole of $\P^2$, and up to conjugation:
		\begin{enumerate} 
			\item  $\Lambda_{\Kul}$ is a line	if, and only if,  $G$ is an elliptic group, a torus group,  a dual torus group of type I,   Abelian Kodaira group or  a $\mathcal{K}_1$ group.
			
			\item  If $\Lambda_{\Kul}$ is a cone of lines over a circle, then $G$ is either a dual torus group of type II, a Kleinian  Inoue group,  or a extended  Kodaira group   ${\mathcal{K}_2}$.

			\item If $\Lambda_{\Kul} = \P^2$, then $G$ is  a dual Torus group of type III, a discrete non-Kleinian Inoue group,  an extended Inoue group,  or  an extended Kodaira group $ {\mathcal{K}_i}$ for $i= 3, 4, 5$.
			\end{enumerate}
			
		\end{enumerate}
		
	\end{theorem}

	
	Concerning the dynamics we have:
	
	\begin{theorem} \label{t:main2}
		Let $G\subset {\rm PSL} (3,\C)$ be as in the theorem above and let $\Lambda_{CoG}^*$ be the Conze-Guivarc'h  limit set of the action on the dual $\dP^2$. Then:

		\begin{enumerate}
			\item If  $\Lambda_{\Kul}$ is a line and $G$ is not a dual torus group, then  $\Lambda_{CoG}^*$  is  the projective dual of $\Lambda_{\Kul}$ and it  is the unique minimal set.

		\item If $\Lambda_{\Kul}$ is a cone of lines over a circle, then 
$\Lambda_{\rm CG}^*$ contains a projective real line and it is not a minimal set because  there is always a  global fixed point.

		\item If $\Lambda_{\Kul} = \P^2$, then:
$\Lambda_{CoG}^*$ contains a complex projective line and it is never a minimal set (as before, because  there is a  global fixed point).

		\end{enumerate}

\end{theorem}

The theorem below is  essential for our proof of 
Theorem \ref{t:main1}.

\begin{theorem}    \label {semidirec} \label{t:desc}  
	Let $G\subset {\rm Heis} (3,\Bbb{C})$ be a discrete group, then $G$ is purely parabolic and:
	\begin{enumerate} 
		\item  \label{p:1r}	There are $B_1,\ldots, B_n $ subgroups of $G$  such that $G=Ker(\Pi\vert_G)  \rtimes B_1\rtimes \cdots \rtimes B_n$ and each $B_i$ is isomorphic to $\Bbb{Z}^{k_i}$, for some $k_i\in \Bbb{N}\cup\{0\}$.
		\item \label{p:2r}  $rank(Ker(\Pi\vert_G) )+\sum_{i=1}^n k_i\leq 6$. 
		\item \label{p:3r} If $G$ is complex Kleinian, then  $rank(Ker(\Pi\vert_G) )+\sum_{i=1}^n k_i\leq 4$ .
	\end{enumerate}
\end{theorem}

We recall (see Definition \ref{def complex Kleinian}) that $G$ is complex Kleinian if it is discrete and there is a non-empty open invariant set where $G$ acts properly discontinuously.

The proof of  Theorem \ref {semidirec} uses in an essential way the obstructor dimension of a group $G$,  introduced  by  Bestvina, Kapovich and Kleiner in \cite {BKK}.  This  is a lower
bound for the ``action dimension" of $G$, and it is based on the classical van Kampen obstruction for embedding a
simplicial complex  into an Euclidean space  \cite{VK}. 
Theorem \ref{semidirec}  strengthens, for discrete groups in  $\Heis(3,\Bbb{C})$,  a Theorem of Bieri and Strebel \cite{BBDZ},  ensuring  that every infinite, finitely presented solvable group  is virtually an ascending  HNN-extension of a finitely generated solvable group.

A difficulty one meets when working with the projective groups $\PSL(n+1,\C)$, which are non-compact, is that one does not have the convergence property (cf.  \cite{Bow, kapovich}). We overcome this problem by using repeatedly  the space of  pseudo-projective maps,
\[\SP(3,\C)=(M(3,\C)-\{{\bf 0}\})/\mathbb{C}^*,\]
where $M(3,\C)$ is the set of all $3\times 3$ matrices with complex coefficients and $\mathbb{C}^*$ acts by the
usual scalar
multiplication. This  was introduced in \cite{CNS, CS} and it provides a natural compactification of the projective group $\PSL(3,\C)$.

We remark that every parabolic element in $\PSL(3,\C)$ is conjugate to a parabolic element in $\PU(2,1)$, the group of holomorphic isometries of the complex hyperbolic space (see \cite{CNS, Nav2}). Yet, the results in this paper show that there are plenty of 
purely parabolic groups  in $\PSL(3,\C)$ which are not conjugate to subgroups in $\PU(2,1)$. Some examples are:
\begin{enumerate}
	\item All purely parabolic groups whose limit set $\Lambda_{\Kul}$ is not a single line.
	\item Unipotent Abelian complex Kleinian groups whose limit set is a single line and the rank is at least three. 
\end{enumerate}
It is possible to provide a full characterization of the purely parabolic groups in $\PSL(3,\C)$ that are conjugate to   groups in $\PU(2,1)$. This shall be done elsewhere.

A corollary of the results in this article is that the Kulkarni limit set of purely parabolic complex Kleinian groups in $\PSL(3,\C)$  either consists of one line or it has infinitely many lines but only two  in general position  (a set of lines is said to be in general position if no three lines in the set are concurrent). This is 
 essential for the classification of the complex Kleinian groups that are   elementary, {\it i.e.},  they have either a point or a line with finite orbit. 
The complete classification of the elementary groups  in $\PSL(3,\C)$ is given in our forthcoming paper \cite{BCNS1}, in relation with a  dictionary in complex dimension two, inspired by  Sullivan's dictionary \cite{BCNS2}.

 We thank the referee for plenty of  valuable suggestions that highly improved the original version of this article.

\section{Preliminaries}\label{s:nb}
Let 
$\mathbb{P}^2_{\mathbb{C}}:= (\mathbb{C}^{3}\setminus
\{{\bf 0}\})/\mathbb{C}^*
$ be the complex projective plane and $[\mbox{
}]:\mathbb{C}^{3}\setminus\{{\bf 0}\}\rightarrow
\mathbb{P}^{2}_{\mathbb{C}}$  the quotient map. A line in $\P^2$ means the image under this projection 
of a complex linear
subspace of dimension $2$. Given  $p,q\in
\mathbb{P}^2_{\mathbb{C}}$ distinct points, there exists  a unique 
complex line passing through  $p$ and $q$; such line is 
denoted by $\overleftrightarrow{p,q}$. The projective dual  $\dP^{2} $ of $\P^2$ is $Gr(\P^2)$ the  Grassmannian of all  
complex lines   in $\P^2$ equipped
with the topology of the Hausdorff convergence. 

The following notion is used along the paper. 

\begin{definition}
	A pencil of lines in $\P^2$ is a collection of lines passing trough a common point.
	\end{definition}

Consider the usual action of $\mathbb{Z}_{3}$  on  $\SL(3,\mathbb{C})$. Then
$\PSL(3,\mathbb{C})=\SL(3,\mathbb{C})/\mathbb{Z}_{3}\,$ is a
Lie group
whose elements are called projective transformations.  We denote also by 
$[\mbox{ }]:\SL(3,\mathbb{C})\rightarrow \PSL(3,\mathbb{C})$  
the
quotient map. We denote by 
${\bf g}= (g_{ij})$ the  elements in $\SL(3,\Bbb{C})$. Given $g\in \PSL(3,\mathbb{C})$,  we say that ${\bf g}\in
\SL(3,\mathbb{C})$ is a lift of
$g$ if $[{\bf g}]=g$.  Then 
$ \PSL(3,\mathbb{C})$  acts  transitively,
effectively and by biholomorphisms on
$\mathbb{P}^2_{\mathbb{C}}$
by $[{\bf g }]([w])=[{\bf g }(w)]$, where $w\in
\mathbb{C}^3\setminus\{{\bf 0}\}$ and ${\bf g }\in \SL(3,\mathbb{C})$.
Recall  (cf. \cite[Chapter 4]{CNS}): 

\begin{definition}\label{def parabolic}
	Let $g \in \PSL(3,\C)$ and ${\bf g}$ a lift to  $\SL(3,\C)$. Then $g$ is:
	\begin{itemize}  
		\item elliptic if  ${\bf g}$ is diagonalizable with unitary eigenvalues; 
		\item parabolic if 
		${\bf g}$  is non-diagonalizable with unitary eigenvalues;     
		\item   loxodromic if   ${\bf g}$  has  some non-unitary eigenvalue.
	\end{itemize}
\end{definition}


Now let $M(3,\C) $ be the set of all $3 \times
3$ matrices with complex coefficients.
Define the space of  pseudo-projective maps by:
$\SP(3,\C)=(M(3,\C)-\{{\bf 0}\})/\mathbb{C}^*,$
where $\mathbb{C}^*$ acts on $M(3,\C)-\{{\bf 0}\}$ by the
usual scalar
multiplication. We have the  quotient map $[\mbox{ }]:M(3,\C)\setminus
\{{\bf 0}\}\rightarrow
\SP(3,\C)$. Given
$P\in \SP(3,\C)$ we define its kernel by:
$$Ker(P)=[Ker({\bf P})\setminus\{{\bf 0}\}],$$
where ${\bf P }\in M(3,\C)$ is a lift of $P$.
Clearly $\PSL(3,\Bbb{C})\subset \SP(3,\Bbb{C})$ and an element $P$ in $\SP(3,\C)$ is in 
$\PSL(3,\C)$ if and only if $Ker(P)=\emptyset$. Notice that $\SP(3,\C)$ is a  manifold naturally diffeomorphic to $\P^8$, so it is compact. 

Recall that  given a discrete group $G$ in $\PSL(3,\C)$, 
its equicontinuity set ${\rm Eq}(G)$   is
the largest open set on which  the $G$-action forms a normal family. 

\begin{theorem}[See Proposition  2.5 in    \cite{CLU}]	Let $G\subset {\rm PSL(3,\C)}$ be a discrete group. Then $G$ acts properly discontinuously on    ${\rm Eq}(G)$  and one has:
	$$ {\rm Eq}(G) = \P^2 \setminus \overline{\bigcup Ker(P)} \;,$$
	where the union runs over the kernels of all $P\in {\rm SP}(3,\C) \setminus {\rm PSL}(3,\C)$ satisfying that there exists a sequence 
	$(g_n)\subset G$  that converges to $P$.
\end{theorem}

Now let $G\subset \PSL(3,\C)$ be a discrete group and let 
$\Omega$ be a non-empty G-invariant set, {\it i.e.},  $G \Omega=\Omega$. 
We say that
$G$ {\it acts properly discontinuously} on  $\Omega$  if
for each compact set $K\subset \Omega$ the set $\{g\in
G\,\vert  \, g(K)\cap K\}$ is finite.

\vskip.2cm
The following notion was introduced in \cite{SV1}:
 
\begin{definition} \label{def complex Kleinian}
$G$ is {\it complex Kleinian} if there exists a non-empty open $G$-invariant set in $\P^2$ where $G$ acts properly discontinuously.
\end{definition}

\begin{definition}\label{def Kulkarni limit set} Let  $G \subset \PSL(2,\C)$ be a discrete group. 
	Its  {\it Kulkarni limit set} is \cite {Ku}: $$\Lambda_{Kul}(G) = L_0(G) \cup
	L_1(G) \cup L_2(G)\, ,$$  where
	$L_0(G)$  is  the closure  of  the points in
	$\mathbb {P}^2_{\mathbb{C}}$ with infinite isotropy group, 
	$L_1(G)$ is the closure of the set of accumulation points of the orbits 
	$G z$ where $z$ runs over
	$\mathbb{P}^2_{\mathbb{C}}\setminus
	L_0(G)$,  and  
	$L_2(G)$ is the closure of the set of
	accumulation points of orbits $G
	K$  where $K$ runs  over all   compact sets in
	$\mathbb{P}^2_{\mathbb{C}}-(L_0(G) \cup
	L_1(G))$. 
	The  {\it Kulkarni region of discontinuity (or the ordinary
		set)} of $G$ is:
	$$\Omega_{Kul}(G) = \mathbb{P}^2_{\mathbb{C}}\setminus
	\Lambda_{Kul}(G).$$
	
\end{definition}


\begin{proposition} \label{p:pkg}
	
	Let    $G$  be a complex  Kleinian group. Then:
	
	\begin{enumerate}
		
		\item  {\rm [See \cite{Ku}]} The sets\label{i:pk2}
		$\Lambda_{Kul}(G),\,L_0(G),\,L_1(G),\,L_2(G)$
		are  $G$-invariant closed sets. 
		
		\item {\rm (See \cite{Ku}) }  \label{i:pk3} The group $G$ acts properly 
		discontinuously on  $\Omega_{Kul}(G)$. 
		
		\item  {\rm (See  \cite{Nav2} or   \cite[Proposition 3.3.6]{CNS})}\label{i:pk4} Let $\mathcal{C}\subset\mathbb{P}^2_{\mathbb{C}}$ be
		a closed $G$-invariant set such that  for every compact set $K\subset
		\mathbb{P}^2_{\mathbb{C}}-\mathcal{C}$, the set of cluster points
		of  $G K$ is contained in $(L_0(G)\cup L_1(G))\cap \mathcal{C}$, then
		$\Lambda_{Kul}(G)\subset \mathcal{C}$.

		\item   {\rm (See  \cite[Corollary  2.6]{CLU})} The equicontinuity set of $G$ is contained in $\Omega_{Kul}(G)$.
		\item  {\rm( See   \cite[Proposition  3.6]{BCN4})}	If $G_0\subset G$ is a subgroup with finite index, then 
		\[
		\Lambda_{Kul}(G)=\Lambda_{Kul}(G_0).
		\]
		\item  {\rm (See  \cite[Proposition  3.3.4]{CNS})} The set 			$\Lambda_{Kul}(G)$ contains at  least one  complex  line.
	\end{enumerate}
\end{proposition}

There is also the Conze-Guivarc'h limit set, see \cite{CG}. For this  we need the following generalization introduced in \cite{BNU}.  

\begin{definition}
Let  $G\subset  \PSL(3, \C)$ be a discrete group acting on $ \check{ \Bbb{P}}^2_\C$. We define the Conze-Guivarc’h limit set of $G$, denoted $\Lambda_{CoG}(G)$, as the closure  of the set of  points  $q \in \check{ \Bbb{P}}^2_\C$  for which  there exist an open subset $U\subset  \check{ \Bbb{P}}^2_\C$ and a
sequence $(g_n)  \subset  G$,  $g_n=g_m$ if $n=m$, such that for every $p \in  U$
\[
\lim_{n \rightarrow  \infty } g_n  (p) = q .
\]
\end{definition}

A couple of examples will picture the previous concept. We need:

\begin{definition} 
	A matrix $g\in  \GL(3, \Bbb{C})$ is {\it proximal} if it has an
	eigenvalue $\lambda_0\in\Bbb{C}$ such that $\vert \lambda_0\vert  > \vert \lambda \vert $ for all other eigenvalues $\lambda$ of $g$. For such a $g$, an eigenvector $v_0 \in \Bbb{C}^{3}$ corresponding to the eigenvalue $\lambda_0$ is a {\it dominant eigenvector} of $g$. 
	We say that $g\in \PSL(3,\Bbb{C})$ is {\it proximal} if it has a lift $\widetilde g\in \SL(3,\Bbb{C})$ which is proximal; and $v\in \P^2$ is {\it dominant} for $g$ if there is a lift $\tilde v\in \C^{3}$ of $v$ which is dominant  for $\widetilde \gamma$. 
\end{definition}

 We remark that by \cite{ Nav2}, every strongly loxodromic element in $ \PSL(3, \C)$  is proximal, and all loxodromic elements in $\PU(2,1)$ are strongly loxodromic.

\begin{example} \label{e:1}	[Complex hyperbolic groups]
	If   $G\subset \PU(2,1)$  is a non-elementary discrete subgroup, then $\Lambda_{CoG}(\Gamma)$ coincides with the Chen-Greenberg limit set \cite{ChG} of $G$, $\Lambda_{CG}(\Gamma)$, which is the closure of the orbits of points in the complex hyperbolic space $\partial \H^2_\C$. This follows because  loxodromic elements in $\PU(2,1)$ are proximal and  their  attracting fixed points in $\partial \H^2_\C$ correspond to dominant vectors. 
\end{example}

\begin{example} \label{e:2}[Veronese groups]
	Given a non-elementary discrete subgroup $G\subset \PSL(2,\C)$, let $\iota: \PSL(2,\Bbb{C})\rightarrow \PSL(3,\Bbb{C})$ be the canonical irreducible representation  and $\psi:\Bbb{P}^1_{\Bbb{C}}\rightarrow \Bbb{P}^2_{\Bbb{C}}$ the Veronese embedding. A simple computation shows that  $\iota$
	carries  loxodromic elements in $\PSL(2,\C)$ into strongly  loxodromic elements in $\PSL(3,\C)$. This implies    $\Lambda_{CoG}(\iota(G))=\psi(\Lambda_{ChG}(G))$,  see  \cite[Theorem 2.10]{CL}.  \end{example}

Recall from  \cite{CG} that a 
{\it strongly irreducible group} in $\PSL(3,\C)$ is a group whose action on $\P^2$ does not have points or lines with finite orbit.

\begin{theorem} [See \cite{CG} and Corollary 3 in  \cite{BNU}]
	Let $G\subset {\rm PSL}(3,\C)$ be a strongly irreducible group, then:
	\begin{enumerate}
		\item The limit set $\Lambda_{CoG}(G)$ is non-empty and  is the unique minimal set for the action of $G$ on $\P^2$.
		
		\item  The closure  of the dominant points of proximal elements in  $G$ coincides with  $\Lambda_{CoG}(G)$.
	\end{enumerate}
\end{theorem}

Now recall \cite{CNS}:

\begin{definition}\label{d:proyeccion}
	Let 	$G$    be a discrete group in  ${\rm PSL}(3, \C)$. We say that $G$ is {\it weakly-controllable} if it acts with a fixed point $p$ in $\P^2$. In this case a choice of a line $\mathcal L$ in $\P^2 \setminus \{p\}$ determines a projection map $\P^2 \setminus \{p\} \to \mathcal L$ and a group morphism $\Pi : G \rightarrow   \PSL(2,\C)$  called the {\it control morphism} of $G$;  its image 
	$\Pi(G)\subset  \PSL(2, \C)$ is the {\it control group}. These are well defined and independent of $\mathcal L$ up to an automorphism of $\PSL(2, \C)$. 
\end{definition}

\begin{theorem} [See Theorem 5.8.2 in  \cite{CNS}] \label{t:semi}  Let $G\subset {\rm PSL}(3,\Bbb{C})$ be discrete and weakly-controllable, with   $p\in \Bbb{P}^2_\Bbb{C}$ a $G$-invariant point  and $\ell\subset\Bbb{P}^2_\Bbb{C}$ a complex line not containing $p$. Let $\Pi_{p,\ell}=\Pi$ be a projection map defined as  above. If
	$Ker(\Pi\vert_G) $ is finite  and $\Pi(G) \subset {\rm Aut} (\ell) \cong {\rm PSL}(2,\C)$  is discrete, then $G$ acts properly discontinuously on
	\[
	\Omega=
	\left ( 
	\bigcup_{z\in \Omega(\Pi(G))}
	\overleftrightarrow{p,z}
	\right )
	-\{p\} \;,
	\]
	where the union runs over all points in $\ell$ where the action of $\Pi(G)$ is discontinuous.
\end{theorem}

	The following is an improvement of the $\lambda$-Lemma in \cite{Nav2} that we use in the sequel. This is inspired by the classical $\lambda$-Lemma of Palis and De Melo \cite {Pa-Me}.
		
	\begin{lemma}[See Section 2 in {\rm Heis} {\rm Heis}  \cite{CLU}]\label{l:lambda} Let $G$ be a discrete group and let $(g_n)\subset G$ be a sequence of distinct elements, then there exist a subsequence $(h_n)\subset (g_n)$ and pseudo projective maps $P,Q \in SP(3,\C)$ satisfying:
		\begin{enumerate}
			\item \label{i:1}	 $h_n \xymatrix{ \ar[r]_{m \rightarrow  \infty}&} P$ and $ h^{-1}_n \xymatrix{ \ar[r]_{m \rightarrow  \infty}&}Q$.
			\item \label{i:2} \[
			\begin{array}{l}
				Im(P)\subset Ker(Q) \,, \\ 
				Im(Q)\subset Ker(P) \,, \\
				dim (Im(P))+dim(Ker(P))=1 \,, \\
				dim (Im(Q))+dim(Ker(Q))=1 \,. \\
			\end{array}
			\]
			\item \label{i:3} For every point $x\in Ker(P)$ we get
			\[
			Ker(Q)=\bigcup_{x_n\rightarrow x}\{\textrm{accumulation points of }(h_n(x_n))\}.
			\]
			\item \label{i:4} If $\Omega\subset \Bbb{P}^2_{\Bbb{C}} $ is an open set on which $G$ acts properly discontinuously, then either $Ker(P)\subset \Bbb{P}^2_{\Bbb{C}}-\Omega$ or $Ker(Q)\subset \Bbb{P}^2_{\Bbb{C}}-\Omega$.
		\end{enumerate}
	\end{lemma}

We use in the sequel the  obstructor dimension of a group $G$,  introduced  by  Bestvina, Kapovich and Kleiner in \cite {BKK}.   We refer to  \cite [Definition 4]{BKK}. 
\begin{definition}
	Fix a non-negative integer $m$. A finite simplicial complex $K$
	of dimension $\leq m$ is an $m$-obstructor complex if the following holds:
	\begin{enumerate}
		\item  There is a collection 
		$$\Sigma=
		\{
		(\sigma_i,\tau_i)_{i=1}^k \}$$
		of unordered pairs of disjoint simplices of $K$ with $dim\, \sigma_i + dim\, \tau_i = m$
		that determine an $m$-cycle (over $\Z_2$) in 
		$$
		\left \{ \sigma  \times  \tau  \subset  K \times K\vert \sigma \cap \tau  = \emptyset \right \}/\Z_2 
		$$
		where $\Z_2$ acts by $(x, y)\mapsto (y, x).$
		\item For some (any) general position map $f : K \rightarrow  \R^m$ the (finite) number
		$\Sigma_{i=1}^k 	
		|f(\sigma_i)\cap  f(\tau_i)|$
		is odd.
		\item  For every $m$-simplex $ \sigma \in K$ the number of vertices $v$ such that the
		unordered pair $\left \{\sigma , \tau  \right \}$ is in $\Sigma$ is even.
	\end{enumerate}	
\end{definition}

\begin{definition}
	The obstructor dimension $obdim(G)$ is defined to be 0 for finite groups, 1 for 2-ended groups  (see  \cite[Section 9.1]{Dru-Kapo} for a  definition of the ends of a group).
	Otherwise $obdim(G)$ is $m + 2$ where $m$ is the largest integer such that for some $m$-obstructor complex $K$ and some triangulation of the open cone $cone(K)$  there exists a proper map $f : cone(K)^{(0)}\rightarrow  G$ satisfying:
	\begin{enumerate}
		\item  for disjoint simplices $\sigma,\tau  $ in $K$ and every $D > 0$ there are compact sets $ C_1 \subset  cone(\sigma )$,    $ C_2 \subset  cone(\tau  )$ such that $f(cone(\sigma ) - C_1)$ and $f(cone(\sigma ) - C_2)$ are $ > D$ apart.
		\item there is a uniform upper bound on the distance between the images of adjacent vertices in $cone(K)^{(0)}$
	\end{enumerate} 
\end{definition}

We use  the following  theorems;  see \cite{BF,BKK} for the corresponding  proofs :

\begin{theorem}[See Theorem 1  in \cite{BKK}] \label{t:obdim}
	If $obdim (G)\geq m$, then $G$ can not act properly discontinuously on a contractible manifold of dimension $<m$. 
\end{theorem}

\begin{theorem}[See Corollary 2.7 in \cite{BKK}] \label{c:obdim}
	If $G=H\rtimes Q$  with $H$ and $Q$  finitely generated  and $H$ weakly convex, then $obdim(G)\geq obdim(H)+obdim(Q)$.
\end{theorem}

\begin{theorem} [See Corollary  2.2 in  \cite{BF}]
	Let $G$ be a lattice 	in  a simply connected nilpotent Lie group N, then $obdim(G)=dim (N)$, In particular $obdim(\Z^n)=n$. 
\end{theorem}

\section{The families of purely  parabolic  discrete groups} \label{s:examples}

There is a partition of the purely parabolic discrete groups in $\PSL(3,\C)$ into
 five  families: elliptic, torus, dual torus, Inoue and Kodaira groups. 
We now define these families and discuss their algebraic structure and dynamical properties. We find that some of these families naturally split into subfamilies according to the topology of their limit set and the structure of the control group. 

Note that the simplest purely parabolic groups are  cyclic, generated by a parabolic element; there are three types of such elements in $\PSL( 3, \mathbb C)$,  described by the Jordan normal form
of their lifts to $\SL(3, \mathbb C)$. These are:
\begin{displaymath}\label{parabolic elements}
	\left(
	\begin{array}{ccc}
		1 & 1 & 0 \\
		0 & 1 & 0 \\
		0 & 0 & 1
	\end{array}
	\right) , \, \, \, \left(
	\begin{array}{ccc}
		1 & 1 & 0 \\
		0 & 1 & 1 \\
		0 & 0 & 1
	\end{array}
	\right) , \, \, \, \left( \begin{array}{ccc}
		\lambda & 1 & 0 \\
		0 & \lambda & 0 \\
		0 & 0 & \lambda^{-2}
	\end{array}
	\right) \,, |\lambda| =1\, ,  \;\lambda \neq 1.
\end{displaymath}
The first two of these are unipotent; the third 
type is called ellipto-parabolic: it is rational if $\lambda$ is a root of unity or irrational otherwise (see \cite[Chapter 4]{CNS} for details). \label{page}
Each of these belongs to a different type of the families we describe below. The first type  generates torus groups, the second generates Abelian Kodaira groups and the  ellipto-parabolic elements generate elliptic groups.

We remark that all the groups we present in this section are upper triangular. Hence they fix the point $e_1 \in \P^2$, so they are weakly-controllable. Given one of these groups $G$, we let $\Pi : G \rightarrow   \PSL(2,\C)$  be its control morphism (Definition \ref{d:proyeccion}). 
Notice that except for the elliptic groups, all others are subgroups of the Heisenberg group.

Throughout this section we denote by $W$ an additive subgroup of $\C$, by $\mathfrak L$ and additive subgroup of $\C^2$ and by $\mathcal M$ an additive subgroup of $\R$.

\subsection{Elliptic groups}\label{e:fd2}  
Let $W\subset\Bbb{C} $ be an additive discrete subgroup  and consider a group morphism
$ \mu:W\rightarrow \Bbb{S}^1$. Define:
\[
{{\rm Ell}}(W,\mu)=
\left \{
\left[
\begin{array}{lll}
	\mu(w) & \mu(w)w&0\\
	0& \mu(w)&0\\ 
	0&0& \mu(w)^{-2}\\ 
\end{array}
\right ]
:w\in W
\right \}.
\] 	
\begin{lemma}  \label{2.1}
	Elliptic  groups   act with two fixed points, $ \{e_1\}$ and $\{e_3\}$. The control group with respect to $\{e_3\}$ is   $W$ and the kernel of $\Pi$ is trivial, while the control group with respect to $\{e_1\}$ is the image of $W$ under $\mu$. 
	Also:
	\begin{itemize}
		\item The Kulkarni limit set is a  line.
		\item The  equicontinuity set coincides with  Kulkarni's discontinuity  region and is the largest
		open set on which the group acts properly discontinuously.
		\item  The Conze-Guivarc'h limit set of the action on the dual $\dP^2$  is the dual point of the Kulkarni limit set, but it is not the only minimal set.
	\end{itemize}
\end{lemma}

The kernel may or may not be trivial. If the kernel is not trivial, then it is a proper subgroup of $W$ isomorphic to $\Z$.

\begin{proof}	
	The proof of the statements about the control groups are straightforward from the definition.
	From Theorem \ref{t:semi} we have that 
	the Kulkarni limit set is the line $\overleftrightarrow{e_1,e_2}$  and  the Kulkarni's discontinuity  region is the largest
	open set on which $ {\rm Ell}(W,\mu)$ acts properly discontinuously.   Let us prove that  $\Omega_{Kul}$ coincides with the equicontinuity set.	Let 
	$(g_n)\subset {{\rm Ell}}(W,\mu)$ be a sequence of distinct elements. Then  $(g_n)$ can be written as:
	
	$$
	g_n=
	\begin{bmatrix}
		1 & a_n &0\\
		0 & 1 & 0\\
		0 & 0 & \mu^{-3}(a_n) 
	\end{bmatrix} \;,
	$$
	for some 		  sequence  $(a_n)$ in   $W$ with $|a_n|$ converging to $\infty$. 
	Hence there exists  $a\in \C^*$ such that:
	$$
	g_n
	\xymatrix{		\ar[r]_{n \rightarrow \infty}&}
	\begin{bmatrix}
		0 & a &0\\
		0 & 0 & 0\\
		0 & 0 & 0 
	\end{bmatrix}\;;  \;\hbox {and} \;
	g_n^*
	\xymatrix{		\ar[r]_{n \rightarrow \infty}&}
	\begin{bmatrix}
		0 & 0 &0\\
		a & 0 & 0\\
		0 & 0 & 0 
	\end{bmatrix} \;.
	$$
We get $Eq({{\rm Ell}}(W,\mu))=\C^2=\Omega_{Kul}({\rm Ell}(W,\mu))$ and  $\Lambda_{CoG}({\rm Ell}(W,\mu))=\{e_2\}$ is a minimal set. This also proves the statements about  the Conze-Guivarc'h limit set. 
Notice that $\{e_3\}$ also is a minimal set for the dual action, so there is more than one minimal set.
\end{proof}

\subsection{Torus groups}\label{sub:torus groups} 
These are of the form:
$$
{{\mathcal T}}(\mathfrak L)=
\left \{
\left[
\begin{array}{lll}
	1 & 0  &a\\
	0 & 1 & b\\
	0 & 0& 1\\ 
\end{array}
\right]:(a,b)\in \mathfrak L
\right \} \;,
$$	
where $\mathfrak L$ is an additive discrete subgroup of  $\C^2$.

\begin{lemma}\label{torus groups}
	The control group $\Pi(G)$ is an additive subgroup of $\C$ that may or may not be discrete, and:
	\begin{itemize}
		\item   The Kulkarni limit set $\Lambda_{\Kul}$  is a line. 
		\item The   equicontinuity set coincides with the Kulkarni discontinuity region and it is the largest open set where the group acts properly discontinuously.
		\item The Conze-Guivarc'h limit set 
		$\Lambda_{CoG}^*(\mathcal{T}(W))$ of the action on the dual projective plane $\dP^2$ is a single  point, the projective dual of the unique line in $\Lambda_{Kul}$, and it is the only minimal set for the action on $\dP^2$. 
	\end{itemize}
\end{lemma}

The kernel of the control morphism $\Pi$ may or may not be trivial.

\begin{proof}
	The first and second  statements follow from \cite[3.4.2]{CNS}; in fact $\Lambda_{\Kul}$ is the line $\overleftrightarrow{e_1,e_2}$.
	It remains to prove the statements about $\Lambda_{CoG}^*(\mathcal{T}(\mathfrak L))$.
	Let $(g_n)\subset \mathcal{T}(\mathfrak L)$ be a sequence of distinct elements.   Choose a sequence  $(a_n,b_n) \in \mathfrak L$ such that  $\vert a_n\vert+ \vert b_n\vert  \xymatrix{		\ar[r]_{n \rightarrow \infty}&} \infty  $ and set:
	$$
	g_n=
	\begin{bmatrix}
		1 & 0 &a_n\\
		0 & 1 & b_n\\
		0 & 0 & 1 
	\end{bmatrix} \;.
	$$
	We can assume  that (taking a subsequence if necessary) there exist  $a,b\in \C$ so that   $\vert a\vert+ \vert b\vert\neq 0$ and 
	$$
	g_n
	\xymatrix{		\ar[r]_{n \rightarrow \infty}&}
	\begin{bmatrix}
		0 & 0 &a\\
		0 & 0 & b\\
		0 & 0 & 0 
	\end{bmatrix} \;.
	$$
	Hence   
	$$
	g_n^*=		
	\begin{bmatrix}
		1 & 0 &0\\
		0 & 1 & 0\\
		-a_n & -b_n & 1 
	\end{bmatrix} \,
	$$
	converges to $\{e_3\}$. 
	Thus $\Lambda_{CoG}^*(\mathcal{T}(L))=\{e_3\}$ and this set is minimal.  
\end{proof}

\subsection{Dual Torus groups}
\label{e:if}\label{nk-Torus}
This family actually splits in three classes, depending on the limit set. To explain this trichotomy, let us consider first the following lemma. We recall that the rank of a group is the smallest number of elements that generate it.

\begin{lemma} \label{ladd}
	Let $\mathcal{L}$ be an additive discrete subgroup of  $\C^2$ and consider the natural projection $[ \,]: \C^2 \setminus \{0\} \to \P^1$. Then the closure  of  $[\mathcal{L}  \setminus \{0\}]$ is  either a point, a real projective line or the whole of $\P^1$, and one has:
	\begin{itemize}
		\item The closure  of  $[\mathcal{L}\setminus \{0\}]$ is a point if and only if $\mathcal{L}$ has rank 1 or it has rank 2 and it is generated by two $\C$-linearly dependent elements.
		\item The closure  of $[\mathcal{L} \setminus \{0\}]$ is a real projective line if and only if $\mathcal{L}$ has rank 2 and it is generated by two $\C$-linearly independent elements. 
		\item The closure  of  $[\mathcal{L}  \setminus \{0\}] = \P^1$  if and only if $\mathcal{L}$ has rank at least 3.
	\end{itemize}
\end{lemma} 
\begin{proof}
	If $\mathcal{L}$ has rank 1, then trivially $[\mathcal{L} \setminus \{0\}]$ is a single point. If $\mathcal{L}$ has  rank two  and is generated  by two    $\C$-linearly dependent vectors,   then $[\mathcal{L}  \setminus \{0\}]$ is trivially a single point. If $\mathcal{L} $ has rank two and is generated by two  $\C$-linearly independent vectors,     then we can assume that $\mathcal{L}=\Bbb{Z}\oplus \Bbb{Z}$, therefore   $[\mathcal{L} \setminus \{0\}]=\{[1,m/n]:m,n\in \Z\}$ which is a dense set  in a real projective   line in $\P^1$. Finally, if $\mathcal{L}$ has rank at least 3, then we can pick up  two elements in $\mathcal{L}$ which are $\C$-linearly independent. Moreover we can assume that $(1,0)$ and $(0,1)$ are such elements, let $p=(w_1,w_2)$  be the other point in $\mathcal{L}$, then
	$$
	\overline{[\Pi  \setminus \{0\}]}
	=
	\overline{
		\{
		[k+nw_1:l+nw_2]:k,l,n\in \Bbb{Z}
		\}
	}
	=
	\overline{
		\{
		[r+sw_1:t+sw_2]:r,s,t\in \Bbb{R}
		\} 
	} \;.
	$$
	Now, if  $z\in \C$ satisfies  $Im(z)\neq 0$ and  we consider
	\[
	s_0=1,\,
	r_0=\frac{Im(w_2)-Re(z)Im(w_1)}{Im(z)}-Re(w_1),\,
	t_0=z(r+w_1)-w_2\,,
	\]
	then a straightforward computation  shows  that $[1:z]=[r_0+s_0w_1:t_0+s_0w_2]$, which concludes the proof.
\end{proof}

\begin{definition} \label{2.4}
	A {\it dual torus group}  in $\PSL(3,\C)$ is a group of the form
	\[
	\mathcal{T}^*(\mathfrak L)=
	\left 
	\{
	g_{(a,b)}=
	\left[
	\begin{array}{lll}
		1 & a & b\\
		0 & 1 & 0\\
		0 & 0 & 1\\ 
	\end{array}
	\right]:(a,b)\in \mathfrak L
	\right \} \;,
	\] 
	where $\mathfrak L \subset \Bbb{C}^2 $ is an additive discrete group. {\it The torus group is of type I}  if its Kulkarni limit set $\Lambda_{\Kul}$ is a line; {\it it is 
		of type II if}  $\Lambda_{\Kul}$ is a cone of lines over a circle, and {\it it is of type III}  if $\Lambda_{\Kul} = \P^2$.
\end{definition}


\begin{lemma} \label{l:dualtd}
	All dual torus groups are discrete, with trivial control group, and:
	\begin{enumerate}
		\item The group is:
		\begin{enumerate}
			\item Type I  if and only if $\mathfrak L$ either has rank $1$ or it is generated by two $\C$-linearly dependent elements;
			\item Type II  if and only if $\mathfrak L$  has rank $2$ and   it is generated by two $\C$-linearly independent elements;
			\item  Type III  if and only if $\mathfrak L$  has rank at least $3$. 
		\end{enumerate}
		\item If it is of type I or II, then:
		\begin{enumerate}
			\item Its Kulkarni discontinuity  set   is the largest
			open set where the action is  properly discontinuously and it coincides with the equicontinuity set.
			\item The Conze-Guivarc'h limit set of the dual action is the projective dual  of $\Lambda_{Kul}$; it is either a point if the group is of type I or a real projective line if it is of type II, and it is a minimal set whenever it is a single point.
		\end{enumerate}
		
		\item If the group is of type III, then:
		\begin{enumerate}
			\item The sets $\Omega_{\Kul}$ and Eq are both empty and there is no non-empty open invariant set of $\P^2$ where the group acts properly discontinuously.
			\item The Conze-Guivarc'h is the projective dual of $\Lambda_{Kul}$, but this is not a    minimal set since there is a global fixed point.
			
		\end{enumerate}
	\end{enumerate}
\end{lemma}

\begin{proof} It is not hard to check  that these groups are discrete and they have  trivial control group. The rest of the proof follows  from Lemma \ref{ladd}.
\end{proof} 

\subsection{The  Inoue groups} \label{e:id}	There are three classes of   groups in this family.

a)\emph{Inoue Kleinian} groups, or just Inoue groups.  These are  proper subgroups of fundamental groups of Inoue surfaces. The
limit set is a cone of lines over a circle.

b) \emph{Inoue non-Kleinian} groups.  These are Inoue groups in the sense of Definition \ref{d:inoue} whose 
limit set is all $\P^2$.

c) \emph{Extended Inoue} groups. These are finite extensions of Inoue groups whose  limit set is all of $\P^2$, hence they are not Kleinian. 


\subsubsection {\bf a) Inoue groups}\label{def group I}
Let $\mathfrak L \subset \Bbb{C}^2 $ be an additive discrete subgroup,
let $ x ,y, z  \in \C$ and set
$$ \gamma_1=	\gamma_1(x,y,z) :=	
	\begin{bmatrix}
		1 & x+z & y\\
		0 & 1& z\\
		0 & 0& 1
	\end{bmatrix}\; \, , \; \, {\mathcal I} = {\mathcal I}(u,v) :=  \left  \langle \begin{bmatrix}
	1 &u  &v\\
	 0 & 1 & 0 \\
	0 & 0 & 1 \\
\end{bmatrix}  \; , \; (u,v)\in \mathfrak L \right \rangle \;. 
$$
Notice ${\mathcal I} $ is a dual torus group.

\begin{definition}  \label{d:inoue}
 An Inoue group is a discrete subgroup of $\PSL(3,\C)$ which is an extension $G = \langle {\mathcal I}, \gamma_1\rangle$ where the dual torus group ${\mathcal I} $ is of type II. This kind of groups splits into two classes:  Kleinian, {\it i.e.} subgroups with non-empty discontinuity region, and non-Kleinian.
	\end{definition}

%


\begin{theorem} \label{2.7}	
	A group $\Gamma$ is Inoue if and only if there exists a dual torus group $\widetilde I$ such that:
	$$ \Gamma = \large\{ \widetilde I \, \gamma_1^k \, | k \in \, \mathbb Z\large \}  \;.$$
	These groups are 
	non-Abelian   semi-direct  products    $\Bbb{Z}^{2} \rtimes \Bbb{Z}$.  They are weakly-controllable with  control group  $\Z$ and kernel  (of the control morphism)  $\Bbb{Z}\oplus \Z$.
	Moreover:
	\begin{enumerate}
	\item  The group is Kleinian if and only if it is of the form:
	\begin{small}
	\[
	{\rm Ino}(x,y,p,q,r)= 
	\left \{
	\left[
	\begin{array}{ccc}
		1 & k+ \frac{l p}{q}+m x & \frac{l }{r} +m \left(k+\frac{l p}{q}\right)+ \begin{pmatrix} m\\ 2\end{pmatrix} x+m y\\
		0 & 1 & m \\
		0 & 0 & 1 \\
	\end{array}
	\right]:k,l,m\in \Bbb{Z}
	\right \}
	\]
\end{small}
\hskip -4pt where   $x,y\in \C$ and  $p,q,r\in \Bbb{Z}$  are such that $p,q$ are co-primes  and $q^2$ divides $r$.   

\item If the group is Kleinian, then:
	
	\begin{itemize}
		\item The Kulkarni limit set is a cone of lines over a real projective space: 
		$$\Lambda_{Kul} =  \overleftrightarrow{e_1,e_2}
		\cup \bigcup_{s\in \Bbb{R}} \overleftrightarrow{e_1,[0:1: s]\,}.
		$$ 
		
		\item   The Kulkarni discontinuity  set   coincides with the equicontinuity set and it  is the largest
		open set on which the group  acts properly discontinuously. These sets are 
		are biholomorphic to $\C \times (\mathbb H^+ \cup \mathbb H^-)$ where $\mathbb{H}^{\pm}$ are the open half planes in $\C$.
		\item The Conze-Guivarc'h limit set for  ${\rm Ino}(x,y,p,q,r)$ is a real projective line, and it is not minimal.
	\end{itemize}
	\end{enumerate} 
\end{theorem}

\begin{proof}   Set $\widetilde {\mathcal{I}}=\{g\in \Gamma: g\in Ker(\Pi) \}$,  then $\Gamma=\{h\gamma^k_1:k\in \Bbb{Z}, h\in \widetilde {\mathcal{I}}\}$ and $\widetilde {\mathcal{I}}$ is a dual torus groups, proving the first statement.  On the other hand, it is clear that 	${\rm Ino}(x,y,p,q,r)$ is a discrete group.   Set: 
	\[
	g{(k,l,m)}
	=
	\left[
	\begin{array}{ccc}
		1 & k+l c+m x & ld +m \left(k+lc\right)+ \begin{pmatrix} m\\ 2\end{pmatrix} x+m y\\
		0 & 1 & m \\
		0 & 0 & 1 \\
	\end{array}
	\right] \,.
	\]
	Let $k,l\in \Z$, then  a straightforward computation shows  that the fixed point set is: $$Fix(g{(k,l,0)})= \overleftrightarrow {e_1,[0: -ld:k+lc]}\,.$$ Hence, letting $L_0$ be as in Definition \ref{def Kulkarni limit set} we get:
	$$
	\overleftrightarrow{e_1,e_3}
	\cup \bigcup_{s\in \Bbb{R}} \overleftrightarrow{e_1,[0: 1:s]\,}\subset L_0({\rm Ino}(x,y,p,q,r)) \,.
	$$

	Finally, let us show that  
	$$
	\Bbb{P}^2_\Bbb{C}-Eq({\rm Ino}(x,y,p,q,r))=\overleftrightarrow{e_1,e_3}
	\cup \bigcup_{s\in \Bbb{R}} \overleftrightarrow{e_1,[0:1:s]\,} \,.
	$$
	Let   $(g_m)_{m\in \Bbb{N}}\subset {\rm Ino}(x,y,p,q,r)$ be a sequence of distinct elements, then there exists a sequence $u_m=(k_m,l_m,n_m)\in \Bbb{Z}^3$  of distinct elements  such that:
	\[
	g_m=
	\left[
	\begin{array}{ccc}
		1 & k_m+l_m c+n_m x & l_md +n_m \left(k_m+l_m c\right)+ \begin{pmatrix} n_m\\ 2\end{pmatrix} x+n_m y\\
		0 & 1 & n_m \\
		0 & 0 & 1 \\
	\end{array}
	\right].
	\]
	Since $G_w$ is discrete we get 
	$r_m=max\{\vert k_m \vert, \vert l_m \vert, \vert n_m \vert \} \xymatrix{
		\ar[r]_{m \rightarrow  \infty}&} \infty $. Now we can  assume that there exists   $u=(x,y,z)\in\Bbb{R}^3-\{{\bf 0}\}$  such that 
	$r_m^{-1}u_m \xymatrix{\ar[r]_{m \rightarrow  \infty}&} u$,  thus
	\[
	g_m
	\xymatrix{
		\ar[r]_{m \rightarrow  \infty}&}
	P=
	\left[
	\begin{array}{ccc}
		0 & k_0+l_0 c+n_0 x & l_0d +n_0 \left(k_0+l_0 c\right)+ \begin{pmatrix} n_0\\ 2\end{pmatrix} x+n_0 y\\
		0 & 0 & n_0 \\
		0 & 0 & 0 \\
	\end{array}
	\right],
	\]
	\[
	Ker(P)=
	\left \{
	\begin{array}{ll}
		\overleftrightarrow{e_1,e_2} &\textrm { if } k_0+l_0 c+n_0 x =0\\
		\overleftrightarrow{e_1,[0: -l_0d : k_0+l_0 c]} &\textrm { if }  k_0+l_0 c  \neq 0, n_0=0\\
		e_1 &\textrm { in other case }\\
	\end{array}
	\right .
	\] 
	This last convergence implies that  $\Lambda_{CoG}^*$ is a real projective line. This set is not minimal because it has a global fixed point.
\end{proof}

\subsubsection{\bf b) Extended Inoue groups}\label{Ext Inoue} 
These  are discrete extensions  of Inoue groups. 
We use the following normal forms. 

\begin{equation}\label{generators} 		
	\mathfrak g =	
	\begin{bmatrix}
		1& 1& s\\
		0 & 1& 1\\
		0& 0 & 1\\
	\end{bmatrix}\;  ;\, s \in \C \,.
	\end{equation}

\begin{definition}	 \label{2.8}
	An {\it extended Inoue group} is a discrete group $  {\widetilde{\rm Ino}}(\mathfrak L,  x,y,z)$ generated  by matrices with normal forms $\mathfrak g,  \gamma_1$ and the group 
	 $ {\mathcal I}$.
\end{definition}	

We have:
\begin{lemma} \label{l:exti} Up to conjugation, every extended Inoue group is of type:
	\[
	{\widetilde{\rm Ino}}(\mathfrak L,  x,y,z) \;=	\;
	\left \{
	h \cdot 
	\mathfrak g^k \cdot 
	\gamma_1^m
 \;,  k, m \in \mathbb Z, h\in  {\mathcal I} \right \}, 
	\]
	with   $a,b,c,x \in \C$, $k, m\in \Z$ 
	and $(u,v)\in \mathfrak L$ satisfying that if we let $\pi_1, \pi_2$ be the coordinate functions in $\C^2$, then: 
	
	i) $(0,x-z) , (0,\pi_1(\mathfrak L)) , (0,z \cdot \pi_2(\mathfrak L))$ are in $\mathfrak L$. 
	
	ii) $\mathfrak L$ has rank at least 3;  and 
	
	iii) either $z\notin \R$ or $x=y=z=0$; hence the control group has rank 2.
\end{lemma}

The proof follows from Proposition \ref{p:pifa1d}.

\begin{corollary} \label{2.10} The extended Inoue   groups have
		an infinite discrete control group. There is not an open invariant set of $\P^2$ where the group acts properly discontinuously;   the sets $\Omega_{\Kul}$ and Eq are both empty, the Kulkarni limit set is all of $\P^2$  and the  Conze-Guivarc'h limit set contains at least a complex projective line. 		
	\end{corollary}

		\begin{proof}  
		  By Lemma \ref{l:exti} the group $\mathcal{I}$ in \ref{def group I}
		  is a  dual torus group with rank at least three. By Lemma \ref{l:dualtd}  we deduce $L_0(\mathcal{I})=\P^2$, where $L_0$ is the first set in Kulkarni's limit set, so there is no open set on which an extended Inoue group acts properly discontinuously.  Finally, we remark  that by  the last statement in the lemma above, the control group  is the additive  group spanned by $z$ and $1$, so  by  \ref{l:exti} we conclude that the control group has rank two and it is discrete. Note that the kernel of the control morphism is a dual torus group of type III, so 
		  Eq and $ \Omega_{\Kul}$ are both empty, the Kulkarni limit set is all of $\P^2$  and the  Conze-Guivarc'h limit set contains at least one  complex projective line. 
	\end{proof}

	The following is an immediate consequence of Theorem \ref{2.7} and Lemma \ref{l:exti}:
	\begin{proposition}\label{Prop Types Inoue groups}
	\begin{enumerate}
	\item The Inoue Kleinian groups  have limit set a cone of lines over a circle, the kernel of the control morphism is $\Z \oplus \Z$ and the control group is $\Z$.
	\item  The Inoue non-Kleinian groups  have limit set all of $\P^2$,  the kernel of the control morphism is $\Z \oplus \Z$ and the control group is $\Z$.
	\item The extended Inoue  groups  are finite extensions of Inoue groups (Kleinian or not). They	
	have limit set all of $\P^2$,  the kernel of the control morphism is $\Z^k$ for some $k \ge 3$, and the control group is $\Z \oplus \Z$.

	\end{enumerate}
	
	\end{proposition}

\subsection {Kodaira groups}\label{Kodaira groups} 
There are:

a) Abelian Kodaira groups, ${\mathcal{K}_0}$.

b) Extended Kodaira groups ${\mathcal{K}_i}$, $i= 1,\ldots, 5$.  These are all non-Abelian; they are finite extensions of Abelian Kodaira groups and they
split into six types according to their limit set and the control group. These are:

\begin{enumerate}
\item  Groups ${\mathcal{K}_i}$, $i = 1, 2, 3$.
These   three classes are constructed in a similar way  (see Lemma \ref{l. non-Abelian Kodaira} below). 

$\bullet$  The groups ${\mathcal{K}_1}$ have limit set a complex projective line and discrete control group. 

$\bullet$  The groups ${\mathcal{K}_2}$ have limit set a cone of lines over a circle and non-discrete control group.

$\bullet$  The groups ${\mathcal{K}_3}$ have limit set the whole $\P^2$.

\item  The groups ${\mathcal{K}_4}$, ${\mathcal{K}_5}$ and ${\mathcal{K}_6}$ are obtained by a different type of extensions. These also have limit set  the whole of $\P^2$ and  the three classes  are
distinguished by the rank of their control groups, which are always non-discrete.
\end{enumerate}

\subsubsection{\bf Abelian Kodaira groups}\label{Kodaira-0}

These are 	 Abelian subgroups  of  fundamental groups of Kodaira surfaces; they are finite extensions of dual torus groups of type I.

\begin{definition} A  Kodaira group is a discrete group in $\PSL(3,\C)$ such that each  element  in the group can be written in the form:	
$$
\left[
\begin{array}{lll}
	1 & a&b\\
	0 & 1 & a\\
	0 & 0& 1\\ 
\end{array} 
\right] \,.
$$
\end{definition}
We have: 

\begin{lemma} \label{2.12}
	Let $G$ be a  Kodaira group, then 	$G$ is , weakly-controllable and  isomorphic to $Ker (G)\oplus  C(G)$ where   $C(G)$ is  the control group  and  $Ker(G)$ is the kernel of the control morphism. Also:
	\begin{itemize}	
		\item  We have $Rank \,G\leq  4$
		\item	The Kulkarni limit set is a line. 
		\item Its complement $\Omega_{\Kul}$ coincides with the equicontinuity set and 
		is the largest open set on which the group acts properly discontinuously. 
		\item The 
		Conze-Guivarc'h limit set of the action on the dual $\dP^2$ is a point, the dual of $\Lambda_{Kul}$, and it is the unique minimal set. 
		
	\end{itemize} 
\end{lemma}		
\begin{proof}
	That  the group is  a direct sum as stated is immediate. The claim about the rank follows from \cite{Suwa}.
	Now let $(g_n)\subset G$ be a sequence of distinct elements, then there exist sequences   $(a_n),(b_n)\subset \C$  such that:
	$$
	g_n=
	\begin{bmatrix}
		1 & a_n &b_n\\
		0 & 1 & a_n\\
		0 & 0 & 1 
	\end{bmatrix} \;.
	$$
	This implies that there exist  $a,b,c,d\in \C$ satisfying: $\vert a\vert +\vert b \vert \neq 0$, $\vert c\vert +\vert d \vert \neq 0$ and
	$$
	g_n
	\xymatrix{		\ar[r]_{n \rightarrow \infty}&} g=
	\begin{bmatrix}
		0 & a &b\\
		0 & 0 & a\\
		0 & 0 & 0 
	\end{bmatrix} \; ; \; \hbox{and } \; 
	g_n^*
	\xymatrix{		\ar[r]_{n \rightarrow \infty}&}
	h=
	\begin{bmatrix}
		0 & 0 &0\\
		c & 0 & 0\\
		d & c & 0 
	\end{bmatrix} \;.
	$$
	Thus  $Ker(g)\subset \overleftrightarrow{e_1,e_2}$.   The rest of the proof is as in the elliptic case. \end{proof}

The next result enables us to provide a normal form for  the Kodaira groups. 
\begin{lemma}  A group  $G$ is   an Abelian Kodaira group if and only if there is $W\subset \C$, an additive discrete subgroup, $R\subset \Bbb{C} $ is an  additive subgroup and  
	$L:R\rightarrow \Bbb{C}$  a group  morphism such that $Rank(W)+Rank(R)\leq 4$, 
	\[
	G=
{\mathcal{K}_0}(W,R,L)=
	\left \{
	\left[
	\begin{array}{lll}
		1 & a&L(a)+a^2/2 +w\\
		0 & 1 & a\\
		0 & 0& 1\\ 
	\end{array}
	\right]:a\in R,w \in W
	\right \} \,,
	\]
	and 
	\[
	\lim_{n \rightarrow \infty} L(x_n)+w_n=\infty
	\]
	\hskip-15pt for every sequence $(w_n)\in W$ and every sequence $(x_n )\subset R$ converging to $0$. 
	
\end{lemma}

As an example, 
consider  $w_1=1,w_2= \sqrt{2},w_3= e^{\pi i /4},w_4=\sqrt{2}e^{\pi i /4}$, and let $W$ be $Span_\Bbb{Z}\{w_1,w_2,w_3,w_4\}$. Define
$ L: W \rightarrow \C$   by  setting $ L(1)=2^{-1}$, $ L(\sqrt{2})=\sqrt{2}-1$,  $ L(e^{\pi i /4})=i+2^{-1}$, $L(\sqrt{2}e^{\pi i /4})=\sqrt{2}i+1$, and then extend by linearity:
$$L\left (\sum_{j=1}^{4} n_jw_j\right )=\sum_{j=1}^4 n_j L(w_j)\,.$$ 
Then	\[
{\mathcal{K}_0}(W,L)=
\left \{
\left [
\begin{array}{lll}
	1 & a & L(a)+a^2 2^{-1}\\
	0 & 1 & a\\
	0 & 0 &1\\
\end{array}
\right ]
:a\in W
\right \}
\] 
is  weakly-controllable  and  isomorphic to $\Bbb{Z}\oplus \Bbb{Z}\oplus\Bbb{Z}\oplus\Bbb{Z}$. If we let $\Pi$ be its control morphism, then 
a straightforward computation shows that its kernel 
$Ker(\Pi\vert_{{\mathcal{K}_0}(W,L)})$  is trivial and the control group     is a  dense  subgroup of $\C$.

Similarly, let $W$ be now $Span_\Z(\{1, \sqrt{2}, e^{\pi i /4}\})$ and define 
$ L: W \rightarrow \C$   as in the previous example. Then,    
\[
\left \{
\begin{bmatrix}
	1 & x & L(x)+x^2/2 +ik\\
	0 & 1& x\\
	0 & 0 &1\\
\end{bmatrix}:
x\in W, k\in \Z
\right \}
\]
is  a  weakly-controllable discrete group      isomorphic to  $\Bbb{Z}^3\oplus\Bbb{Z}$.   We find
that the kernel  of the projection to the control group  is  isomorphic to $\Z$ and  
the  control group is  non-discrete and   isomorphic to $\Z^2$.






\subsubsection {\bf Extended Kodaira groups ${\mathcal{K}_1}$,  ${\mathcal{K}_2}$, ${\mathcal{K}_3}$} \label{n-ab Kodaira}\label{n-ab Kodairacono} 

$\,$
\vskip.1cm

\noindent
Let    $W$ be a non-trivial additive subgroup of $\C$ and define a normal form:
		$$h_w=
		\begin{bmatrix}
			1 & 0 & w\\ 
			0& 1& 0\\
			0 & 0 & 1
		\end{bmatrix}\, , \; w \in W -\{{\bf 0}\}\; .
		$$	

	\begin{definition} \label{2.1.5}
		  An {\it extended Kodaira group}  ${\mathcal{K}_i}$, $i=1 , 2, 3$, 
		  is a discrete group generated by a normal form $h_w$  and the normal forms $\gamma_1$ and $\mathfrak g$ 
		$$\mathfrak g =	
		\begin{bmatrix}
			1& 1& 0\\
			0 & 1& 1\\
			0& 0 & 1\\
		\end{bmatrix}\;  ;\,
		\quad  \gamma_1=		
		\begin{bmatrix}
			1 & x+z & y\\
			0 & 1& z\\
			0 & 0& 1
		\end{bmatrix}\; ; \;  \, x ,y, z  \in \C\;.$$
			That is: 
		${\mathcal{K}_i}=  \langle  h_w, \mathfrak g, \gamma_1  \rangle  \;.
		$
	\end{definition}
	

	\begin{lemma}\label{l. non-Abelian Kodaira}
We have:		\begin{enumerate}
			\item  The group is ${\mathcal{K}_1}$    $\Leftrightarrow$  its control group is discrete   $\Leftrightarrow$  $z \in \C \setminus \R$;
		\item The group is ${\mathcal{K}_2}$   $\Leftrightarrow$ its control group is non-discrete and  $W$ has rank 1 $\Leftrightarrow$   $W$ has rank 1 and
			$z \in  \R \setminus \Bbb{Q}$;
			\item The group  is ${\mathcal{K}_3}$ $\Leftrightarrow$ $W$ has rank $> 1$ and $z \in  \R \setminus \mathbb Q$. 
In this case the control group is automatically non-discrete.	
		\end{enumerate}
	\end{lemma}

	As an example of type III groups  take  $W=\{(m+ni,k+li)\in \C^2:k,l,m,n\in \Z\}$, $a=b=0$  and $c=i$, we get $\Heis(3,\Z[i])$, the Heisenberg group with coefficients in $\Z[i]$.

	\begin{lemma}	\label{217} 
				The non- Kodaira groups ${\mathcal{K}_i}$, $ i=1, 2, 3$, 
				can be written as:
		\[ {\mathcal{K}_i} \, = \,
		\left \{
		\begin{bmatrix}
			1 &0  &w\\
			0 & 1 & 0 \\
			0 & 0 & 1 \\
		\end{bmatrix}
		\begin{bmatrix}
			1 &1 &0\\
			0 & 1 & 1 \\
			0 & 0 & 1 \\
		\end{bmatrix}^n
		\begin{bmatrix}
			1 &x+z &y\\
			0 & 1 & z \\
			0 & 0 & 1 \\
		\end{bmatrix}^m
		:m,n\in \Z, w\in W
		\right \}
		\]
		Hence these  are
		semi direct  products  of the form  $(\Bbb{Z}^{Rank(W)} \rtimes \Bbb{Z})\rtimes \Z$.     The kernel of the control morphism   is  isomorphic to $\Z^{Rank(W)}$. The  control group is     isomorphic to $\Z\oplus \Z$ and it is discrete if and only if the group  is of type  ${\mathcal{K}_1}$.
		
	\end{lemma}

	\begin{proof}
		The proof is a direct consequence of  Theorem \ref{t:pifa1}.
		\end{proof}

	\begin{lemma} \label{2.18}
		For the non- Kodaira groups ${\mathcal{K}_1}$ and ${\mathcal{K}_2}$ 
		one has:
		\begin{itemize}
			\item The  Kulkarni discontinuity region   coincides with the equicontinuity set and is the largest open set on which the group acts properly discontinuously. 
			\item The Conze-Guivarc'h limit set is the dual  of $\Lambda_{\Kul}$. It is a point if the group is ${\mathcal{K}_1}$ and   a real projective line if the group is  ${\mathcal{K}_2}$.
		\end{itemize}
	\end{lemma}

	\begin{proof}	The proof is similar to the other cases.
		Let  $(g_n)\subset {\mathcal{K}_1}$ be a sequence of distinct elements. Then there exist sequences  $(k_n),(m_n)\subset \Z$ and     $(w_n)\subset W$  such that: 
		$g_n$ is:
		\begin{Small}
			$$
			\begin{bmatrix}
				1 & k_n+m_n (x+z) & \frac{1}{2} \Big((k_n+m_nz)(k_n+m_nz-1)+m_n(2y+z(-x+1-z))+2w_n+m_n^2zx\Big) \\
				0 & 1 & k_n+m_n z \\
				0 & 0 & 1 \\
			\end{bmatrix}, 
			$$
		\end{Small}
		and   the  inverse transpose matrix $g_n^*$ is:
		\begin{small}
			$$
			\begin{bmatrix}
				1 & 0 & 0 \\
				-k_n-m_n x-m_n z & 1 & 0 \\
				1/2 \Big(k_n + k_n^2 - 2 w_n - 2 m_n y + 2 m_n k_n (x + z) + m_n (1 + m_n) z (x + z)\Big) & -k_n-m_n z & 1 \\
			\end{bmatrix} \,.
			$$
		\end{small}
		\vskip.2cm
		
		\hskip-4pt 
				Case 1.  $(k_n), (m_n)$ are eventually constant. In  this case $w_n \xymatrix{		\ar[r]_{n \rightarrow \infty}&} \infty $ and therefore 		$$
		g_n \xymatrix{		\ar[r]_{n \rightarrow \infty}&}
		\begin{bmatrix}
			0 & 0 & 1 \\
			0 & 0 & 0 \\
			0 & 0 & 0 \\
		\end{bmatrix} \,.
		$$
		
In the sequel  we will assume that either $(k_n)$ or $(m_n)$ is   a sequence of  distinct elements.
		
		\vskip.2cm
		\noindent
		Case 2.  $k_n+m_n z \xymatrix{		\ar[r]_{n \rightarrow \infty}&} u\in \C$. So  $z\in \R-\Bbb{Q}$ and therefore  $W$ has rank 1 and  $k_n,m_n \xymatrix{		\ar[r]_{n \rightarrow \infty}&} \infty $. Let  $w\in \C^*$ be the  generator of $W$ and define    
		$\rho_n=\max\{\vert m_n\vert, 2^{-1} \vert m_n(2y+z(-x+1-z))+2w_n+m_n^2zx\vert\}$,  so we can assume  that there are $a_1,b_1\in \C$ such that 
		$$\rho^{-1}_n(m_n, 2^{-1}(m_n(2y+z(-x+1-z))+2w_n+m_n^2zx))  \xymatrix{		\ar[r]_{n \rightarrow \infty}&} (a_1,b_1). $$
		Thus 
		$$g_n \xymatrix{		\ar[r]_{n \rightarrow \infty}&}
		\begin{bmatrix}
			0 & a_1 x & b_1 \\
			0 & 0 & 0 \\
			0 & 0 & 0 \\
		\end{bmatrix}=g \,.
		$$
		For simplicity assume  $a_1\neq 0$, so there is $l\in \Z$ and $(l_n)\subset \Z$ such that 
		$$
		2b_1a_1^{-1}= 2y+z(-lw+1-z)+w \lim_{n \rightarrow \infty} (2l_nm_n^{-1}+m_nzl)
		$$
		where $r=\lim_{n \rightarrow \infty} (2l_nm_n^{-1}+m_nzl)\in \R$. Now a straightforward computation shows:
		\[
		Ker(g)=\overleftrightarrow{e_1,(2lw)^{-1}( 2y+z(1-z)+w(r-lz))e_2-e_3 } \;.
		\]

	\vskip.2cm	
		
\noindent
		Case 3.  $h_n=k_n+m_n z \xymatrix{		\ar[r]_{n \rightarrow \infty}&} \infty$. Let us define:
		\[
		\rho_n =  
		\max\Big\{\vert  h_n+m_n x\vert , \vert \frac{1}{2} (h_n(h_n-1)+m_n(2y+z(1-x-z))+2w_n+m_n^2zx) \vert, \vert 
		h_n\vert
		\Big \} 
		\]   
		\hskip-4pt
		Then we can assume that there are $a_1,b_1,c_1\in \C$ such that $\vert a_1\vert+\vert b_1\vert+\vert c_1\vert \neq 0$ and 
		\begin{equation}\label{e:conve}
			\begin{array}{l}
				\rho_n^{-1} ( k_n+m_n (x+z)) \xymatrix{		\ar[r]_{n \rightarrow \infty}&} a_1\\
				\rho_n^{-1} {\begin{Small}( \frac{1}{2} (h_n(h_n-1)+m_n(2y+z(1-x-z))+2w_n+m_n^2zx) ) \xymatrix{		\ar[r]_{n \rightarrow \infty}&} b_1 \end{Small}}\\
				\rho_n^{-1}(k_n+m_n z ) \xymatrix{		\ar[r]_{n \rightarrow \infty}&} c_1 \;.
			\end{array}
		\end{equation}
		Hence
		$$g_n \xymatrix{		\ar[r]_{n \rightarrow \infty}&}
		\begin{bmatrix}
			0 & a_1  & b_1 \\
			0 & 0 & c_1 \\
			0 & 0 & 0 \\
		\end{bmatrix}=g \,.
		$$
		If $c_1\neq 0$ we get that $Ker(g)$ is either a point or $\overleftrightarrow{e_1,e_2}$, so let us assume $c_1=0$.  Under this assumption by equation \ref{e:conve} we deduce:
		\[
		\begin{array}{l}
			\rho_n^{-1} m_n   \xymatrix{		\ar[r]_{n \rightarrow \infty}&} a_1x^{-1} \;, \\ 
			\rho_n^{-1} ( \frac{1}{2} ((k_n+m_nz)^2+2w_n+m_n^2zx) ) \xymatrix{		\ar[r]_{n \rightarrow \infty}&} b_1-a_12^{-1}x^{-1}(2y+z(-x+1-z)).\\
		\end{array}
		\]
		At this point observe that in this case $a_1=0$ implies  $Ker (g)=\overleftrightarrow {e_1,e_2}$, so we will assume $a_1\neq 0$. 
		
	\vskip.2cm
		\noindent	
		Claim 1. We have  $z\in \R$. Observe that
		 
		\[
		\lim_{n \rightarrow \infty}  m_n^{-1} (k_n+m_n z ) \,= \,
		\lim_{n \rightarrow \infty}  (\rho_n)^{-1}(k_n+m_n z )\cdot \lim_{n \rightarrow \infty}  \frac{\rho _n}{m_n}\\
		=a_1^{-1}\cdot 0=0 \;.
		\] 
		Thus $ \lim_{n \rightarrow \infty}  m_n^{-1} k_n= -z \in \Bbb{R}$.\\

		As a consequence of the previous claim and Lemma \ref{217} we deduce that $W$ has rank 1 and $z\in \Bbb{R}-\Bbb{Q}$. Moreover, by our previous analysis the only interesting case is  $W\nsubseteq \R$. As before let $w$ be the generator of $W$, thus there are  $(l_n)_{n\geq 0}\subset \Z$ such that
		$$g=
		\begin{bmatrix}
			0 & x  & 2^{-1}(2y+z(-l_0w+1-z)+s)\\
			0 & 0 & 0\\
			0 & 0 & 0 \\
		\end{bmatrix} \;,
		$$
		where  
		$s=\lim_{n \rightarrow \infty} m_n^{-1} ( \frac{1}{2} ((k_n+m_nz)^2+2l_n w+m_n^2zl_0w) )$. Now observe that the following limits exist and they are finite.
		
		$$\lim_{n \rightarrow \infty} \frac{1}{m_n}(k_n+m_nz)^2;\,\,\, 2s_2=\lim_{n \rightarrow \infty} m_n^{-1}   (2l_n +m_n^2zl_0) \;.$$
		Thus  $s=s_1+s_2w$ and 
		
		$$g=
		\begin{bmatrix}
			0 & l_0w  & 2^{-1}(2y+z(1-z)+s_1+(s_2-l_0)w)\\
			0 & 0 & 0\\
			0 & 0 & 0 \\
		\end{bmatrix} \;,
		$$
		which concludes the proof. 
	\end{proof}



	
	\subsubsection {\bf Extended Kodaira groups ${\mathcal{K}_4}$ , ${\mathcal{K}_5}$} \label{H group} \label{K group}

	 These are similar to the previous groups, but with one and two more generators, respectively.
	We introduce the following normal forms:
	
	$$
	\gamma_2=
	\begin{bmatrix}
		1 & a + c& b\\
		0 & 1& c\\
		0 & 0& 1\\
	\end{bmatrix} \; , \;
	\gamma_3 = \begin{bmatrix}
		1 &d+f&e\\
		0 & 1 & f \\
		0 & 0 & 1 \\
	\end{bmatrix} \; ,
	$$
	with $a, b, c, d, e, f \in \C$.
	Notice that $\gamma_1, \gamma_2, \gamma_3$ are  matrices of the same type evaluated on different parameters. In each case we will specify the conditions on all these parameters.

	Let $k,l,m\in \Z$, $w\in W$,   with $W$ as above,	and   $x,a,b,c,d,e,f \in \C$   satisfy:  
	$a\neq 0$, $\{a,d, af-dc\}\subset W$, 
	and let ${\mathcal{K}_4}$ be the group depending on all these parameters, defined by 
	\[ {\mathcal{K}_4} \; = \; 
	\left \{
	h_w \cdot g^k \cdot \gamma_1^m \cdot \gamma_2^n \; : \; k, m,n \in \Z, w\in W
	\right \} \,.
	\]
	\hskip -1pt 	
	We assume further that for every  real line  $\ell\subset \C$ passing through the origin  we have $rank(\ell\cap Span_\Z\{1,c,f\})\leq 2$. This  condition springs from \cite{wal} where the author  considers the density
	properties of finitely generated subgroups of rational points on
	a commutative algebraic group over a number field. Additionally the 
	following  restrictions should be imposed over the coefficients  in order to get discrete groups:
	\begin{enumerate}
		\item if $d=0$ then $f\notin \R$;
		\item if $ad^{-1}\notin \R$ then there are 
		$r_1,r_2\in\Bbb{Q}$ such that 
		\[
		c=\frac{a(f-r_1)}{d}-r_2 \,.
		\]
	\end{enumerate}

	We now take the previous ${\mathcal{K}_4}$-groups and add one more generator. 
	Define
	$${\mathcal{K}_5}=\{h_w \cdot g^k \cdot \gamma_1^m \cdot \gamma_2^m  \cdot \gamma_3^l \; : \; k, m,n, l \in \Z, w\in W\}\,,$$
	where   
		$x,a,b,c,d,e,f,g,h,j \in \C$    are subject to the conditions: $a( \vert d\vert +\vert g\vert)\neq 0 $, $\{ a,d, g, dj- gf, af-cd, aj - cg \}\subset W$.
	Furthermore:
	
	\begin{enumerate}
		
		\item   If $g=0$, then there are $r_0,r_1,r_2,r_3\in\Bbb{Q}$ such that $r_1\neq 0$ and
		$$
		(r_2-r_0)^2+4r_1r_3<0;
		$$
		such that: 
		$$
		j=\frac{r_2+r_0\pm \sqrt{(r_2-r_0)^2+4r_1r_3}}{2};
		$$
		\[
		a=d \; \frac{r_2-r_0\pm \sqrt{(r_2-r_0)^2+4r_1r_3}}{2r_1};\
		\]
		$$
		c= (f-r_4) \, \frac{r_2-r_0\pm \sqrt{(r_2-r_0)^2+4r_1r_3}}{2r_1}-r_5
		$$

		\item If $ad^{-1}\notin \R$,  then there are $r_1,r_2,s_1,t_1,s_2, t_2,s_3,t_3 \in \R$ such that: 
		\[
		g=r_1 a+r_2 d\; , \,\quad 
		r_2 t_2\neq t_3 \;,
		\]
		\[
		f=\frac{A_2\pm\left(c+ t_2\right)\sqrt{A_1} }{2 \left(r_2 t_2-t_3\right)}\; \; , \,\quad
		j=\frac{A_3\pm\left(c r_2+t_3\right) \sqrt{A_1} }{2 \left(r_2 t_2-t_3\right)} \;,
		\]
		where:
		\begin{small}
			$$
			\begin{array}{l}
				A_1=\left(-r_2 s_2+r_1 t_2-s_3-t_1\right){}^2-4 \left(r_2 s_1 t_2-r_1 s_2 t_3+r_2 s_2 s_3-r_1 t_1 t_2+s_3 t_1-s_1 t_3\right) \;, \\
				A_2=-c r_2 s_2-c r_1 t_2+c s_3-c t_1+r_2 s_2 t_2-r_1 t_2^2+s_3 t_2-2 s_2 t_3-t_1 t_2 \;, \\
				A_3=r_2 \left(c r_1 t_2+s_3 \left(c+2 t_2\right)-c t_1-s_2 t_3\right)+t_3 \left(-r_1 \left(2 c+t_2\right)-s_3-t_1\right)-c r_2^2 s_2 \;.\\
			\end{array}
			$$
		\end{small}	
	\end{enumerate}

	These groups have the following properties:



\begin{lemma}
	The ${\mathcal{K}_4}$ and ${\mathcal{K}_5}$ groups   are all weakly-controllable, discrete,  with kernel  isomorphic to  $\Bbb{Z}\oplus  \Z$  and  their control group is non-discrete and isomorphic to $\Z\oplus  \Z\oplus \Z$ and $\Z\oplus  \Z\oplus \Z\oplus \Z$  respectively. Moreover  the  Kulkarni limit set is $\P^2$ and the equicontinuity regions is empty.
\end{lemma}
\begin{proof}
	For the proof of this lemma see Lemma \ref{l:m4}  and  Propositions \ref{p:iird3c},   \ref{p:iirnd3c}.
	\end{proof}

	
	\begin{remark}
	
	We notice that in these families  we can have examples where the control group $\Pi({\mathcal{K}_i})$  is non-discrete but is not dense in $\C$, as well as examples where $\Pi({\mathcal{K}_i})$ is dense in $\C$. The control group by definition is a subgroup of $\PSL(2,\C)$, yet, in the cases we consider here, each element in the control group is a translation, so we can think of the control group as being an additive subgroup of $\C$.
	
	For instance, taking 
	$W=\Z[i]$, $x=b=e=d=0$,   $f=i$ and  $c$ an irrational number, we generate a discrete group with a non-discrete dense subgroup of $\C$ as control group. However, taking 
	$W=\Z[i]$,  $x=a=b=e=0$, $c=i$, $d=1$ and  $f=r+is$, where   $r,s\in \R$  satisfy that   $\{1,r,s\}$ is a   $\Bbb{Q}$-linearly independent set, then the corresponding discrete group has a dense subgroup of $\C$ as control group.  This shows that unlike the 1-dimensional case where purely parabolic groups have very simple dynamics, in dimension 2 the
	two different  dynamics described above, both fairly rich,  exist for control groups  of purely parabolic groups.  
	This type of behavior  is important when studying non-discrete subgroups of Lie groups, see for instance \cite{roy, wal}.
	\end{remark}

	\section{The classification theorems}\label{s: classification} 
	
	We now  provide a complete classification of the purely parabolic discrete subgroups of $\PSL(3,\C)$.

	\begin{theorem} \label{t:mainck}
		Let $G\subset {\rm PSL} (3,\C)$ be a complex Kleinian discrete subgroup. Then 
		$G$ does not contain  loxodromic elements if, and only if,  there exists a normal  subgroup $G_0\subset G$ of finite index  such that
		$G_0$ is purely parabolic and it is conjugate to one (and only one) of the following groups:	
		\begin{enumerate}
			\item \label{k-2} An Elliptic group 
			  as in Subsection  \ref {e:fd2}. 		
			\item \label{k-1} A Torus group,   as in Subsection  \ref{sub:torus groups}.
			
			\item \label{k-5} A dual Torus group   of type I and II,   as in Definition  \ref{2.4}.
			
			\item \label{k-6}   A Kleinian Inoue group,  as in Theorem  \ref{2.7}.

			\item \label{k-3} An  Kodaira group,  as in Example \ref {Kodaira-0}.
			
			\item \label{k-4} An extended  Kodaira  group ${K}_1$ or  ${K}_2$,  as in Definition \ref{2.1.5}.

		\end{enumerate}
		
	\end{theorem}


	\begin{theorem} \label{t:maind} Let $G\subset {\rm PSL} (3,\C)$ be a  discrete subgroup which is not Kleinian. Then 
		$G$ does not contain  loxodromic elements if, and only if,  there exists a normal  subgroup $G_0\subset G$ of finite index  such that
		$G_0$ is purely parabolic and it is conjugate to one (and only one) of the following groups:		
		\begin{enumerate}
			
			\item  A dual Torus group   of type III,   as in Definition  \ref{2.4}.

			\item A  discrete non Kleinian Inoue group, as in Theorem  \ref{2.7}.

			\item  An extended Inoue group, as in Definition   \ref{2.8}.

			\item  An extended Kodaira group  $  {\mathcal{K}}_3$,  as in Definition \ref{2.1.5}.

			\item  An extended  Kodaira  group ${\mathcal{K}}_4$ or   ${\mathcal{K}}_5$,   as in Subsection  \ref{K group}.
			
		\end{enumerate}		
		
	\end{theorem}
	\vskip.2cm

	The rest of this article is devoted to proving theorems \ref{t:mainck} and \ref {t:maind}.
	
	Concerning the dynamics we have:
	
	\begin{corollary} F\label{t. geo-dyn-K}
		Let $G\subset {\rm PSL} (3,\C)$ be  a complex Kleinian group.
		Then:
		
		\begin{enumerate}
			
			\item \label{p:1dynkp} The Kulkarni limit set $\Lambda_{\Kul}$ is:
			\begin{itemize}
				\item A pencil of lines over a circle  if it is a dual torus group of type II,  an Inoue group  or extended  Kodaira groups      ${\mathcal{K}}_2$.
				\item One line otherwise. 

					\end{itemize}
			
			\item \label{p2dynkp} Concerning the Conze-Guivarc'h  set $\Lambda_{CoG}^*$ of the action on the dual $\dP^2$ :
			
			\begin{itemize}
				\item  Is either a point or a real projective line.  
				\item   It is a minimal whenever   is a single point.
				
				
			\end{itemize}
			
		\end{enumerate}		
		
	\end{corollary}

	\begin{proof} The proof of part  (\ref{p:1dynkp}) goes as follows. By Theorem
		 \ref{t:maind},   $G$  is virtually  conjugated  to one and only one  of the groups given  in  Theorem  \ref{t:maind}.  From  Lemma \ref{l:dualtd},  Theorem  \ref{2.7} and Corollary \ref{2.10}	 we know the groups with  limit  set  a cone of lines over a circle: these  are   dual torus group  type II, Kleinian Inoue  and  ${\mathcal{K}}_2$  groups  and for   the remaining ones in Theorem \ref{t:maind}  its limit set is a line.
		
In order to 	proof of part  (\ref{p2dynkp})  we apply Lemmas  \ref{torus groups},  \ref{l:dualtd}, \ref{2.12},  \ref{2.18},   Theorem \ref{2.7}  and  Corollary \ref{2.10}	, the    Conze-Guivarc'h  limit set  is  either a real projective line or a point.
		\end{proof}
	

	Concerning the region of discontinuity, this is empty in all non-Kleinian cases. In the Kleinian case we have:
	
	\begin{corollary} \label{t. geo-disc}
		Let $G\subset {\rm PSL} (3,\C)$ be a purely parabolic complex Kleinian group. Then   $\Omega_{\Kul}$  is biholomorphic to:  
			\begin{enumerate}
				\item $\C^2$  if $\Lambda_{\Kul}$  is a line.
				\item   In other case  is $\C \times (\mathbb \H^+ \cup \mathbb \H^-)$ where $\H^{+}$ is the upper  half plane in $\C$ and $\H^{-}$ is the lower   half plane in $\C$.

			\end{enumerate}
		
	\end{corollary}
	\begin{proof}
		The proof follows  immediately  from Theorem \ref{t. geo-dyn-K} item (\ref{p:1dynkp})   because when the limit set is a line, then its complement in   $\P^2$ is  biholomorphic to  $\C^2$, and if   the limit set is a pencil of lines over a circle, then its complement is biholomorphic to    $\C \times (\mathbb \H^+ \cup \mathbb \H^-)$.
		\end{proof}

\section{Dynamics of triangularizable groups without loxodromic elements}
 In this section we use techniques of dynamical systems  in order to show   that discrete subgroups  of $\PSL(3,\C)$  without loxodromic elements  are  triangularizable, see Theorem  \ref{t:wal}. Moreover we show that these groups are either subgroups of  $\Heis(3,\C)$ or they  can be  described in a very precise way, see Theorem \ref{t:gcsp}.
  
\subsection{Purely parabolic  groups are virtually  triangularizable}   \label{s:ppt}


  Recall that a discrete group  $G\subset \PSL(3,\Bbb{C})$ acts strongly irreducibly on $\Bbb{P}^2_{\C}$ if  there are no points or lines with finite orbit. Also
a  group  $G\subset \PSL(3,\C)$ is   \emph{affine}   if $G$  has  an invariant  complex line in $\P^2$.
Now we have:

\begin{lemma}  \label{l:invs}
	If  $G\subset {\rm PSL}(3,\C)$ is a discrete group without loxodromic elements, then  $G$ is either affine or weakly-controllable. 
\end{lemma}
\begin{proof}
	By   \cite[Proposition 4.10]{BCN},   discrete groups in $\PSL(3,\C)$ acting strongly irreducibly on $\P^2$ contain loxodromic elements,   so we can assume that there is  a non-empty proper subspace $l\subset \P^2$  such that $l$ has a finite orbit under the action of $G$. Observe that by  duality  we can assume $l$ is a point; let $l_1,\ldots,l_k$ be the orbit of $l$  under $G$. 
	
	Let $U$ be the projective space generated by $\{l_1,\ldots,l_k\}$; clearly  $U$ is $G$-invariant. We claim that  $U$ is either a point or a line. Assume,  on the contrary,  that $U=\P^2$. Let $g\in G$ be a parabolic element, then there exist $s\in \{1,\ldots, k\}$ such that $l_{s}\notin \Lambda_{Kul}(\langle g\rangle )$ then $l_s$ has infinite orbit under the cyclic group $\langle g\rangle$, which is a contradiction.
\end{proof}

If $G\subset {\rm PSL}(3,\C)$   does not contain loxodromic elements, then  Lemma \ref{l:invs}  implies that $G$ has either an invariant line or an invariant pencil of lines. The following lemma gives restrictions upon the action of $G$ on the invariant line or pencil. 

\begin{lemma}\label{l:lt7}
	Let $G\subset {\rm PSL}(3,\C)$  be a discrete group without   loxodromic elements.
	\begin{enumerate}
		\item
		If   $G$ is affine,  then the action of $G$ on the invariant line does not  contain a subgroup conjugate to a    dense subgroup of ${\rm SO}(3)$.
		
		\item  If   $G$ is  weakly-controllable, then the control group of $G$ 
		does not  contain a subgroup conjugate to a    dense subgroup of ${\rm SO}(3)$.
	\end{enumerate}
\end{lemma}

\begin{proof}
 Let us prove only the case where $G$ is weakly-controllable, the proof in the affine case is similar.  As before, let $\Pi$ be the projection to the control group. Let us proceed by contradiction, 
		let $(g_n)_{n\in \Bbb{N}}$ be an enumeration of  $G$ and define  $H_m=\Pi(\langle g_1,\ldots g_m   \rangle )$. If each  group is finite, then  by   the classification of   subgroups in $\PSL(2,\C)$ with finite order,  we conclude that for  $m$ large $H_m$ is either cyclic or dihedral, and therefore the control  group  $\Pi(G)$ is  a subgroup of the  infinite dihedral group; this is not possible since $\Pi(G)$ is dense in $SO(3)$. Now, applying Selberg's Lemma to the $H_m$'s,  we deduce that $\Pi(G)$ contains an element with infinite order. On the other hand,   since   $\Pi(G)$ is dense  in $SO(3)$ and  by  Tits alternative  we conclude that  for $m$ large,  $H_m$   contains a  rank two free subgroup. To conclude,  let $g_1,g_2\in G$ be such that  $\Pi(g_1),\Pi(g_2)$ generate a rank two free group, then $g_1g_2g_1^{-1}g_2^{-1}$ has a lift $\rho\in \SL(3,\C)$, given by

\[
\rho=
\left (
\begin{array}{cc}
	1 & b\\
	{\bf 0} &  B 
\end{array}
\right )
\]
where $B\in {\rm SO}(3)$ has infinite order  and $b\in \C^2$. Clearly $\rho$ is non-diagonalizable, with unitary eigenvalues and infinite order. Therefore $G$ is non-discrete.
\end{proof}

Now we prove the following  extension of the Lie-Kolchin Theorem \cite {Stein}, in which we allow the existence of non-unipotent elements.

\begin{theorem} \label{t:wal}
	Let $G \subset {\rm PSL (3,\C)}$ be a discrete group without loxodromic elements. Then there exists  a normal subgroup $G_0$  of $G$ with finite index such that $G_0$ leaves invariant a full flag in $\P^2$. Hence the group $G_0$ is simultaneously triangularizable.
\end{theorem}
\begin{proof}
	Let $G \subset \PSL(3,\C)$ be a discrete group without loxodromic elements. Then by Lemma \ref{l:invs} we know that $G$ has a proper non-empty projective subspace $p$ invariant under $G$. Here we prove the case where  $p$ is a point;  the other case is analogous, considering a line $\ell$ as a point in $\dP^2$. Since $G$ does not contain loxodromic elements,    neither $\Pi (G)$ does,  and  by  Corollary 13.7 in \cite{CS} we have that  $\P^1-Eq(\Pi(G))$ is either empty or contains a single point. If $\P^1-Eq(\Pi(G))$ is a single point, then $G$ is simultaneously triangularizable. So we assume $\P^1= Eq(\Pi(G))$. Then Lemma \ref{l:lt7} implies that    $\Pi(G)$ is either finite or it is a  subgroup of the infinite dihedral group $Dih_\infty$. If $\Pi(G)$ is finite, it is  enough to consider $G_0=Ker(\Pi\vert_G) $. If $\Pi(G)\subset Dih_\infty$  then  consider  $G_0=\{g\in G:\Pi(g)\in Rot_\infty\}$.
\end{proof}

\begin{corollary} \label{c:trian}
	Let $G\subset {\rm PSL} (3,\C)$ be  a discrete group without loxodromic elements, then $G$ is  virtually  finitely generated.
\end{corollary}
\begin{proof}
	Since $G$ does not contain loxodromic elements, we know that $G$ contains a finite index subgroup which is triangularizable and therefore solvable. It is well known that discrete solvable groups are finitely generated,  see \cite{aus}. 
\end{proof}


\subsection{A Lie-Kolchin Theorem for purely parabolic groups} \label{s:screw}
The following   is  a slight extension  of the Lie-Kolchin Theorem.  

\begin{theorem} \label{t:gcsp}	Let $G$ be a purely parabolic discrete group in ${\rm PSL}(3,\C)$. Then  $G$ is either virtually unipotent or it
	contains a subgroup of finite index which is conjugate to:
	\[
	G=
	\left
	\{
	\begin{bmatrix}
	1  & w & 0 \\
	0 & 1 & 0\\
	0&0& \eta(w) 
	\end{bmatrix}: \,
	w\in W, \;n\in \Bbb{Z}
	\right 
	\} \;, 
	\] 
	with  $W$  a discrete additive subgroup of $\C$ and $\eta: W \rightarrow \Bbb{S}^1$ a group morphism.
\end{theorem}

  This subsection is divided into four parts: in  \ref{s:2sp} and \ref{s:asp}
we show that every discrete   solvable group with an   irrational ellipto-parabolic element  is commutative, see Corollary  \ref{c:ep1}.    In Subsection \ref{s:sa}  we give a list of all  Commutative Lie groups of ${\rm PSL}(3,\C)$. Finally, in  \ref{pollo} we prove   Theorem \ref {t:gcsp} and we also prove Theorem \ref {Lie-Kolchin}.

\subsubsection {Solvable groups with an   irrational ellipto-parabolic element} \label{s:2sp}
\begin{lemma} \label{e:ind2}
	Let $\alpha\in \Bbb{S}^1$ be an element with infinite order  and $a,b,c,x,y,z\in \C$. If $x,y$ are not both zero, then the group 
	\[
	G=
	\left \langle 
	g_1=
	\begin{bmatrix}
	1 & x & z\\
	0& 1& y\\
	0 & 0& 1
	\end{bmatrix},\, 
	g_2=
	\begin{bmatrix}
	1 & a & b\\
	0& \alpha^3& c\\
	0 & 0& 1
	\end{bmatrix}
	\right \rangle  \;,
	\]
	is non-discrete.
\end{lemma}
\begin{proof}
	Let $h\in \PSL(3,\C) $ be given by:
	\[
	h=
	\begin{bmatrix}
	1 & a(1-\alpha^3)^{-1} & b\\
	0& 1& -c(1-\alpha^3)^{-1}\\
	0 & 0& 1
	\end{bmatrix} \;.
	\]
	A straightforward computation shows:
	\[
	h g_2h^{-1}=
	\begin{bmatrix}
	1 & 0 & b+a c(1-\alpha^3)^{-1}\\
	0& \alpha^{3}& 0\\
	0 & 0& 1
	\end{bmatrix},\quad 
	hg_1h^{-1}=
	\begin{bmatrix}
	1 & x & z+cx(1-\alpha^3)^{-1} \\
	0& 1& y\\
	0 & 0& 1
	\end{bmatrix} \;.
	\]
	We take $a=c=0$. Set $g_n=g_2^n g_1g^{-n}_2$ and observe that we have:
	\[
	g_n=g_2^n g_1g^{-n}_2=
	\begin{bmatrix}
	1 & x\alpha^{-3n}& z \\
	0& 1& y\alpha^{3n}\\
	0 & 0& 1
	\end{bmatrix} \,.
	\]
	Clearly $(g_n)$ contains a convergent sequence of distinct elements, proving the lemma.
\end{proof}

\begin{lemma} \label{l:ppind}
	Let $\alpha\in \Bbb{S}^1$ be an element with infinite order and $x,y,z,\beta, \mu,\nu \in \C$. If   $x,y$ are not both zero, then  the group
	
	\[
	G=
	\left \langle
	g_1=
	\begin{bmatrix}
	1 & z &x\\
	0&1&y\\
	0&0&1\\ 
	\end{bmatrix}
	,\, 
	g_2=
	\begin{bmatrix}
	1 & \beta &\mu\\
	0&1&\nu\\
	0&0&\alpha^{-3}\\ 
	\end{bmatrix}
	\right \rangle \;,
	\]	
	is non-discrete.
\end{lemma}
\begin{proof} Notice first that $\beta=0$ implies that $g_2$ is an elliptic element with infinite order, which makes $G$ non-discrete.  So we assume that $\beta\neq 0$ and  $G$ is discrete. 
	  An easy computation shows:
		\[
		G_0=[G, G]=
		\left \{
		\begin{bmatrix}
		1 & 0 & a\\
		0 & 1 & b\\
		0 & 0 & 1\\
		\end{bmatrix}:(a,b)\in  \mathfrak L
		\right \}\,,
		\]
		where $\mathfrak L$ is a discrete  additive subgroup of $\C^2$. Consider $g=[g_{ij}]=[g_2,[g_2,g_1]]$. Then $g\in G_0 $ and $g_{13}g_{23}\neq 0$, so  after conjugating with an upper triangular element, if necessary, we can  assume that $(1,1)\in \mathfrak L$.  A straightforward computation shows:
	
	\[
	g_2^{n}
	\begin{bmatrix}
	1 & 0 &1 \\
	0 & 1 & 1\\
	0 & 0 & 1\\
	\end{bmatrix}
	g^{-n}_2=
	\begin{bmatrix}
	1 & 0 & \alpha^{3n}(n\beta+1)\\
	0 & 1 & \alpha^{3n}\\
	0 & 0 & 1\\
	\end{bmatrix} \,;
	\]
	thus   $\{\alpha^{3n}(n\beta+1,1):n\in \Z\}\subset \mathfrak L$. Now we claim:\\
	
	\noindent
	Claim 1. The set $A=\{\alpha^{3n}(n\beta+1,1):n=0,1,2, 3\}$ is  $\R$-linearly independent. Assume, on the contrary,  that there are $r_0,r_1,r_2,r_3\in \R$ not all equal to $0$, such that:
	\[
	\begin{array}{ll}
	0&=
	\sum_{j=0}^3
	r_j\alpha^{3j}
	(j\beta+1, 1)\\
	&	=	
	\left (\sum_{j=0}^3 r_j\alpha^{3j},0
	\right )
	+
	\sum_{j=0}^3
	\beta r_j\alpha^{3j} (j,1)\\
	&	=\sum_{j=0}^3
	\beta r_j\alpha^{3j}
	(	j,1) \, .
	\end{array} \;
	\] 
	Thus  $\{\alpha^{3j}(j,1):j=0,1,2, 3\}$ is  $\R$-linearly dependent, which is  a contradiction.\\ 
	
	\noindent
	Claim 2. $A=\{\alpha^{3j}(j,1):j=0,1,2, 3,4\}$ is a $\Bbb{Q}$-linearly dependent set. Observe  that  $B=\{\alpha^{3j}(\beta j+1,1):j=0,1,2, 3,4\}$ is  $\Bbb{Q}$-linearly dependent; then using similar arguments as in the previous claim, we get  that $A$  is  $\Bbb{Q}$-linearly dependent.\\
	
	\noindent
 Claim 3. There is $d\in \C^*$ such that $(d,0)\in \mathfrak L$. By Lemma \ref{l:latfund},   there exists  $c\in \C^*$ and $m_0,\ldots, m_5\in \Bbb{Z}$ such that 
	\[
	(c,	0)
	=\sum_{j=0}^5
	m_j\alpha^{3j}
	(j,1),
	\]
	thus
	\[
	(	c\beta,
	0) 
	=\sum_{j=0}^5
	m_j\alpha^{3j}
	(	j\beta+1,
	1).
	\]

	Finally,  let $(m_n)\subset \Z$ be such that $(\alpha^{3m_n})$ is a sequence of distinct elements which converge to 1 and $d\in \C^*$ is such that $(d,0)\in \mathfrak L$.  Then
	
	\[
	g_m=
	g_1^{m_n}
	\begin{bmatrix}
	1 & 0 &d\\
	0 & 1 & 0\\
	0 &0 & 1\\
	\end{bmatrix}
	g_1^{-m_n}
	=
	\begin{bmatrix}
	1 & 0 &\alpha^{3 m_n}d\\
	0 & 1 & 0\\
	0 &0 & 1\\
	\end{bmatrix}
	\xymatrix{
		\ar[r]_{n \rightarrow \infty}&} 
	\begin{bmatrix}
	1 & 0 &d\\
	0 & 1 & 0\\
	0 &0 & 1\\
	\end{bmatrix}
	\]
	which is a contradiction, so $G$ is non-discrete.
\end{proof}
\begin{lemma} \label{l:ppind1}
	Let $\alpha\in \Bbb{S}^1$ be an element with infinite order and $x,y,z,\beta,\mu,\nu \in \C$. If   $x,y$ are not both zero,  then the group
	\[
	G=
	\left \langle
	g_1=
	\begin{bmatrix}
	1 & x &y\\
	0&1&z\\
	0&0&1\\ 
	\end{bmatrix}
	,\, 
	g_2=
	\begin{bmatrix}
	\alpha^{-3} & \beta &\mu\\
	0&1&\nu\\
	0&0&1\\ 
	\end{bmatrix}
	\right \rangle
	\]	
	is non-discrete.
\end{lemma}
\begin{proof}
	Consider the group morphism $\rho:\PSL(3,\C)\rightarrow \PSL(3,\C)$ given by $\rho([M])=(M^t)^{-1}$, here $M^t$ denotes the transpose matrix of $M$. We claim that $\rho(G)$ is non-discrete, which will prove the lemma.  For this, notice that Lemma \ref{l:ppind}   implies:
	\[
	\begin{bmatrix}
	0 &0 &1\\
	0 &1 &0\\
	1 & 0& 0
	\end{bmatrix}
	\left  (
	\begin{bmatrix}
	1 &x &y\\
	0 &1 &z\\
	0& 0& 1
	\end{bmatrix} ^{t}
	\right )^{-1}
	\begin{bmatrix}
	0 &0 &1\\
	0 &1 &0\\
	1 & 0& 0
	\end{bmatrix}
	=	\begin{bmatrix}
	1 &-z & zx-y\\
	0 &1 &-x\\
	0 & 0& 1
	\end{bmatrix},
	\]

	\[
	\begin{bmatrix}
	0 &0 &1\\
	0 &1 &0\\
	1 & 0& 0
	\end{bmatrix}
	\left (
	\begin{bmatrix}
	\alpha^{-3} &\beta &\mu\\
	0 &1 &\nu\\
	0 & 0& 1
	\end{bmatrix}^{t}
	\right )^{-1}
	\begin{bmatrix}
	0 &0 &1\\
	0 &1 &0\\
	1 & 0& 0
	\end{bmatrix}
	=
	\begin{bmatrix}
	1 &-\nu  &\alpha^{3}(\beta \nu-\mu)\\
	0 &1 &-\beta \alpha^3\\
	0 & 0& \alpha^{3}
	\end{bmatrix} \,,
	\]
	and the result follows. \end{proof}

\subsubsection{Commutator group of solvable discrete groups containing irrational ellipto-parabolic elements}\label{s:asp}
In the following, if $g\in \GL(3,\C)$, then $g_{ij}$ will denote the $ij$-th element of the matrix $g$. 

\begin{definition}\label{def $U_+$} 
	Define a  group $U_+$ in $\PSL(3,\C)$ by: 
	\[
	U_+=
	\left\{ 
	\begin{bmatrix}
	g_{11} & g_{12}  &g_{13}\\
	0           & g_{22}   & g_{23}\\
	0           & 0  & g_{33}\\
	\end{bmatrix}
	: g_{11}g_{22}g_{33}  =1
	\right \}
	\]
	and the group morphisms 	$\Pi^*:U_+\rightarrow Mob (\C)$ and $\lambda_{12},\lambda_{23},\lambda_{13}:U_+\rightarrow \C^* $, given by:
	
	\[
	\Pi^*([g_{ij}])z=
	g_{11}g_{22}^{-1}z + g_{12}g_{22}^{-1} \;,
	\]
	\[
	\lambda_{12}([g_{ij}])=g_{11}g_{22}^{-1} \;,
	\]
	\[
	\lambda_{23}([g_{ij}])=g_{22}g_{33}^{-1} \;,
	\]
	\[
	\lambda_{13}([g_{ij}])=g_{11}g_{33}^{-1} \;.
	\]
	
\end{definition}   

Notice that the elements in
 $U_+$  are equivalence classes of matrices. Yet, since different representatives of the same projective transformation differ by multiplication by a scalar, the above homomorphisms are all well-defined.

	\begin{lemma} \label{l:nd}
		Let $G\subset U_+$ be a discrete group, then $G$ contains a  finite index torsion  free subgroup $G_{0}$  such that the following groups are torsion free: 		the control group 
		$\Pi(G_{0})$,  the   dual control group   $\Pi^*(G_{0})$, 
		$\lambda_{12}(G_{0})$, $\lambda_{13}(G_{0})$ and $\lambda_{23}(G_{0})$.
	\end{lemma}
	
	\begin{proof}  By Selberg's lemma, we can assume that $G$ is torsion free.
		Now consider the control group
		$\Pi\vert_G$. Notice that:
		
		\begin{enumerate} 
			\item[Step  1.] We can apply Selberg's lemma to the  group $\Pi(G)$  to get a finite index subgroup   $G_1 \subset \Pi(G)$ which is torsion free. 
			\item[Step  2.] Define  $\widetilde{G}=f^{-1}(G_1)$    and notice this is a finite index subgroup in  $G$.
			\item[Step 3.]  Using  Selberg's lemma again, we get a torsion free, finite index subgroup $G_2$  of  $\widetilde{G}$.
			\item[Step 4.]  Notice that  $G_2$, which  is  torsion free, has finite index in $G$ and its control group $\Pi(G_2)$ also is  torsion free. This proves the first statement.
		\end{enumerate}
		
		We may now follow  this same process with the groups $G_2$ and $\Pi^*\vert_{G_2}$, granting the existence of a finite index, torsion free subgroup   $G_3$ of $G$  for which  $\Pi^*(G_3)$ is torsion free.
		Notice this same process can be applied to the morphisms 
		$\lambda_{12}(G_{0})$, $\lambda_{13}(G_{0})$ y $\lambda_{23}(G_{0})$, thus proving the lemma.
	\end{proof} 		
	
		The following corollary is an   immediate consequence of Lemma \ref{l:nd} and it is of interest in itself:
		\begin{corollary}	 \label{11}
			Let $G$ be an upper  triangular discrete subgroup of ${\rm SL(3,\C)}$. Then $G$ has a finite  index subgroup that does not contain neither 
			elliptic nor rational screws nor rational ellipto-parabolic elements.
		\end{corollary}
		
		We refer to \cite[Chapter 4]{CNS} for the definition of  rational screws, which are all loxodromic elements.

	\begin{lemma} \label{l:gcsp} 
		Let $G\subset U_+$ be a discrete group such that   the groups $\lambda_{12}(G),$ $\lambda_{23}(G),$ $ \lambda_{13}(G)$ are torsion  free. If  $ g\in G$ is an irrational ellipto-parabolic  element, then  $g$ belongs to the center of $G$, 	{\it i.e.}, $g$ commutes with every element of $G$.
	\end{lemma}
	\begin{proof}
		Assume on the contrary, that there exists   an element $h=[h_{ij}]\in G$	such that $[g, h]\neq Id$.  Then there are $x,y,z\in \C$ such that 
		\[
		[g,h]=
		\begin{bmatrix}
		1 & x & y\\
		0 & 1 & z\\
		0 & 0 & 1
		\end{bmatrix}.
		\]
		Now  consider the following cases:\\
		
		\noindent
		Case  1. We have $o(\lambda_{12}(g))=o(\lambda_{23}(g))=\infty$, where $o(\,)$ means the order. Since $G$ is discrete we deduce that $\lambda_{13}(g)=1$.  By  Lemma   \ref{e:ind2}  we have  $x=z=0$ but $y\neq 0$. We deduce  $	\Pi^*[g,h]=\Pi[g,h]=Id$, $g_{13}\neq 0$ and $o(\lambda_{13}(h))=\infty$, therefore 
		\[
		h^ngh^{-n}=
		\begin{bmatrix}
		g_{11} & g_{12} &(\lambda_{13}(h))^ng_{13}\\
		0 & g_{11}^{-2}& g_{23}\\
		0 & 0 & g_{11}\\
		\end{bmatrix}.
		\]
		So the sequence
		$ (h^ngh^{-n})_{n\in \Z}$ contains a subsequence of distinct elements which  converges to a projective transformation and $G$ is non-discrete.\\   
		
				\noindent
		Case  2. We have  $o(\lambda_{12}(g))=\infty$ and $\lambda_{23}(g)=1$.
		Since $G$ is discrete we deduce that $g_{23}\neq 0$.  By Lemma   \ref{l:ppind1}, we deduce $x=y=0$ but $z\neq 0$. Therefore  $	\Pi^*[g,h]=Id$ and $o(\lambda_{23}(h))=\infty$.  Then:
		\[
		h^ngh^{-n}=
		\begin{bmatrix}
		g_{11}^{-2} & g_{12} &g_{13}\\
		0 & g_{11}& \lambda_{23}^n(h)g_{23}\\
		0 & 0 & g_{11}\\
		\end{bmatrix} \;.
		\]
		Thus
		$ (h^ngh^{-n})_{n\in \Z}$  contains a subsequence of distinct elements which  converges to a projective transformation.\\   
		
				\noindent
		Case  3. We have $o(\lambda_{23}(g))=\infty $ and $\lambda_{12}(g)=1$.
		Again, since  $G$ is discrete we deduce $g_{12}\neq 0$. By  Lemma   \ref{l:ppind}  we deduce $x=z=0$ but $y\neq 0$. Therefore  $\Pi[g,h]=Id$ and $o(\lambda_{12}(h))=\infty$. As in the previous cases we get:
		\[
		h^ngh^{-n}=
		\begin{bmatrix}
		g_{11} &  \lambda_{23}^n(h)g_{12} &g_{13}\\
		0 & g_{11}&g_{23}\\
		0 & 0 & g_{11}^{-2}\\
		\end{bmatrix} \,,
		\]
		so
		$ (h^ngh^{-n})$ contains a subsequence of distinct elements which  converges to a projective transformation.\\

		Thus we have shown  that under the assumption  that $G$ is not  commutative we get that $G$ is non-discrete, which  is a contradiction.
	\end{proof}

	\begin{lemma}  \label{13}
		Let $G\subset U_+$ be a discrete group such that   the groups $\lambda_{12}(G),$ $\lambda_{23}(G),$ $ \lambda_{13}(G)$ are torsion free.  If  $ G$ contains an irrational ellipto-parabolic element, then $G$  is commutative .
	\end{lemma}
	\begin{proof}
		We consider first the case where  $o(\lambda_{12}(g))=o(\lambda_{23}(g))=\infty$. Since $G$ is discrete we deduce that $\lambda_{13}(g)=1$. Then there exists   $h\in U_+$ such that
		\[
		h gh^{-1}=
		\begin{bmatrix}
		g_{11} & 0 &a\\
		0 & g_{11}^{-2}& 0\\
		0 & 0 & g_{11}\\
		\end{bmatrix},
		\]
		where $a\neq0 $.
		Since every element $\beta\in G$ commutes with $g$ we have   	\[
		h\beta h^{-1}=
		\begin{bmatrix}
		\beta_{11} & 0 &b\\
		0 & \beta_{11}^{-2}& 0\\
		0 & 0 & \beta_{11}\\
		\end{bmatrix}.
		\]
		This shows that  $G$ is  commutative.\\

		We can apply similar arguments   when either $o(\lambda_{12}(g))$ or $o(\lambda_{23}(g))$ is finite  to show that   in all cases $G$ is commutative. 
	\end{proof}

The following  result is a consequence of  Corollary \ref{11} and  Lemmas \ref{l:gcsp}, \ref{13}

\begin{corollary}  \label{c:ep1}
Every discrete   solvable group with an   irrational ellipto-parabolic element, is  commutative.
\end{corollary}
\subsubsection{Abelian Lie groups}\label{s:sa}

The following list of  Lie groups is used  in  Theorem \ref{l:lt4}. 
      
\begin{definition}\label{list of  Abelian groups} We set:
	\[
	C_1=
	\left \{
	\begin{bmatrix}
		\alpha^{-2} & 0  & 0\\
		0           & \alpha  & \beta \\
		0           & 0  & \alpha\\
	\end{bmatrix}
	: \alpha\in \C^*, \beta \in \C
	\right \}\; ,\,\;
	C_2=
	\left \{	
	\begin{bmatrix}
		1& 0  & \beta\\
		0           &  1& \gamma\\
		0           & 0  & 1\\
	\end{bmatrix}
	:\beta,\gamma\in \C
	\right \} \,,
	\]	
	\[
	C_3=
	\left \{
	\begin{bmatrix}
		1& 	\alpha & \beta \\
		0           & 1 & \alpha\\
		0           & 0  & 1\\
	\end{bmatrix}
	:\alpha,\beta \in \C
	\right \}\; ,\,\;	
	C_4=
	\left \{
	\begin{bmatrix}
		1& 	\alpha& \beta  \\
		0           & 1 & 0\\
		0           & 0  & 1\\
	\end{bmatrix}
	:\alpha,\beta \in \C
	\right \} \,,
	\]
	\[
	C_5=	Diag(3,\C)=
	\left \{
	\begin{bmatrix}
		\alpha& 	0 & 0 \\
		0           & \beta & 0\\
		0           & 0  & \alpha^{-1}\beta^{-1}\\
	\end{bmatrix}
	:\alpha,\beta \in \C^*
	\right \} \;.
	\]
\end{definition}

\begin{theorem} \label{l:lt4}
	Let $G\subset U_+$ be a commutative group. Then $G$ is conjugate to a group $\widetilde{ G}\subset C_j$ for some $j=1,2,3,4,5$.
\end{theorem}

 This  theorem  is proved  in  the Appendix  \ref{a:Abeliano} at the end of this paper.

\subsubsection{Proof of Theorem  \ref{t:gcsp}}

\begin{lemma} \label{l:pfd0}
	Let $G\subset U_+$ be a discrete torsion free group  such that the group  $Ker(\Pi\vert_G) $ is trivial  and each element in $g\in G$ has the form  
	\[
	\begin{bmatrix}
	\alpha^{-2} & 0&0 \\
	0 & \alpha  & \beta\\
	0 & 0& \alpha\\ 
	\end{bmatrix}\,,
	\]	
	where $\mid\alpha\mid=1 $.	Then there exists   $W\subset \Bbb{C} $  a  discrete additive subgroup and  a group morphism
	$ \eta:W\rightarrow \Bbb{S}^1$  such that: 
	\[
	G=
	\left \{
	\begin{bmatrix}
	\eta(w)^{-2} & 0&0 \\
	0 & \eta(w) & \eta(w)w\\
	0 & 0& \eta(w)\\ 
	\end{bmatrix}
	:w\in W
	\right \}.
	\] 
\end{lemma}

\begin{proof} 	Let us define 	$\zeta:G\rightarrow \C $ by 
	$\zeta(	[g_{ij} ] )= g_{23} g_{33}^{-1}$.
	A standard  computation shows that $\zeta$ is a group morphism and   $Ker(\zeta)$ is trivial. Then the following is a well defined group morphism
	\[
	\begin{array}{c}
	\eta:\zeta(G)\rightarrow \C^*\\
	x \mapsto \pi_{22}(\zeta^{-1}(x)).
	\end{array}
	\]
	Clearly
	\begin{equation}\label{e:formal}
	G=
	\left \{
	\begin{bmatrix}
	\eta(w)^{-2} & 0&0 \\
	0 & \eta(w) & \eta(w)w\\
	0 & 0& \eta(w)\\ 
	\end{bmatrix}
	:w\in \zeta(G)
	\right \}.
	\end{equation}

	We claim that the group $\zeta(G)$ is discrete. Assume on the contrary, that $\zeta(G)$  is non-discrete. Then  there exists  a sequence $(g_n)_{n\in\N}\subset G$ of distinct elements such that $(\zeta(g_n))_{n\in\N}$ is a sequence of distinct elements and 
	$\zeta(g_n) \xymatrix{
		\ar[r]_{n \rightarrow \infty}&} 1$.  Then 
	\[
	g_n
	=
	\begin{bmatrix}
	\eta(\zeta(g_n))^{-2} & 0&0 \\
	0 & \eta(\zeta(g_n)) & \eta(\zeta(g_n))\zeta(g_n)\\
	0 & 0& \eta(\zeta(g_n))\\
	\end{bmatrix} \,.
	\]
	Since $\eta(G)\subset \Bbb{S}^1$ we deduce  that $(g_n)$ contains a convergent subsequence, which is a contradiction. Therefore
	$\zeta(G)$ is discrete and we can take  $W=\zeta(G)$.
\end{proof}

\begin{lemma} \label{l:cdiag}
	Let $G\subset {\rm PSL}(3,\Bbb{C})$  be a discrete group  where each element has the form:
	\[
	g=
	\begin{bmatrix}
	a^{-2}& 0 &0\\
	0& a&0 \\
	0	&0& a\\
	\end{bmatrix}.
	\]
	Then  $G$ is virtually cyclic.
\end{lemma}
\begin{proof}
	Define $\rho_{12}:G\rightarrow \Bbb{R}$  by $\rho_{12}(g)=log(\vert \lambda_{12}\vert )$, clearly $\rho_{12}$ is a well defined group morphism, 
	$Ker(\rho_{12})=\{g\in g: \vert  \lambda_{12}\vert=1 \}$ and $\rho_{12}(G)$ is discrete. 
	Let $G_{0}\subset G$ be a torsion free subgroup of $G$ with finite index. 
	Clearly $\rho_{12}\vert_{G_{0}}$ is injective and $\rho_{12}(G_{0})$ is cyclic. 
\end{proof}

Now we  complete the proof of Theorem \ref{t:gcsp}\label{pollo}.
Let $G\subset U^+$ be a discrete group  which contains an element $g_0$ which satisfies $\max\{o(\lambda_{12}(g_0)),o(\lambda_{23}(g_0))\}=\infty $. By Lemma \ref{l:nd}, $G$ contains a finite index subgroup $G_{0}$  for which  the groups 	$\lambda_{12}(G_0)$, $\lambda_{23}(G_0)$, $\Pi^*(G_0)$, $\Pi(G_0)$ and  $G_0$ itself, are all torsion free and finitely generated. Then  by Lemma \ref{l:gcsp},  $G_0$  is .  Therefore  by  Theorem  \ref{l:lt4} the group $G_0$ is conjugate to a group $G_0\subset G$  and $G$ lives either in  $Diag(3,\C)$ or in $C_1$. 
Then the theorem follows from  Lemmas  \ref{l:pfd0} and \ref{l:cdiag}. $\qed$ \\

\section{Discrete subgroups in $\Heis(3,\C)$} \label{s:heis}
In this section we provide a full description of the discrete subgroups  in 	
$$
\Heis(3,\Bbb{C})=
\left 
\{
\begin{bmatrix}
1 & a & b\\
0 & 1 & c\\
0  & 0 & 1\\
\end{bmatrix}: a,b,c\in \Bbb{C}
\right 
\}	\;.
$$
We start with:
  \begin{proposition}
		The whole group $\Heis(3,\Bbb{C})$ is  solvable and purely parabolic.
	\end{proposition} 

This follows from the fact that every element in $\Heis(3,\C)$ has a lift to an upper diagonal matrix with all eigenvalues equal to 1. Then the proposition  follows from the classification of the elements in $\PSL(3,\C)$, see \cite[Ch. 4]{CNS}.

We split this section in four parts:   in Subsection \ref{ss:dis}, for each  discrete subgroup $G$ in $\Heis(3,\C)$ we construct a region where a   subgroup  acts properly discontinuously; as a consequence we obtain Theorem \ref{t:desc}  which is  a decomposition theorem for $G$.  In Subsection  \ref{s:ktrivial} we describe the subgroups  of $\Heis(3,\C)$  for which the kernel  of the control map, $Ker(\Pi\vert_G)$, is finite.  The main tool here is the description of  groups provided by  Theorem \ref{l:lt4}. In  \ref{s:kinfinite} we describe  the complex Kleinian  groups in $\Heis(3,\C)$ with  $Ker(\Pi\vert_G)$  infinite. 	 Finally, in  \ref{s:dis} we describe  the  discrete   non-Kleinian groups  in $\Heis(3,\C)$ with  $Ker(\Pi\vert_G)$  infinite.

\subsection{A discontinuity region for discrete subgroups of $\Heis(3,\C)$} \label{ss:dis}

Let us consider  $G\subset \Heis(3,\Bbb{C})$  a discrete group.

\begin{definition} We set:
	\[
	\begin{array}{l}	
	B(G)=\{ (g_n)\subset G: ( \Pi(g_n) )\textrm{ converges in $\PSL(2,\C)$}  \} \,;\\ 
	L(G)=\{ S\in{\rm SP}(3,\Bbb{C}): \textrm{ there is } (g_n)\in B(G) \textrm{ converging to } S \}  \,;\\ 
	\mathcal{L}(G)=\{
		{\ell\in \check{\Bbb{P}}^{2}_{\Bbb{C}}
	}: \textrm{ there is  }  S  \in L(G) \textrm{ satisfying }  Ker(S)=\ell   \}  \,;\\
	\Omega(G)=\Bbb{P}^2_{\Bbb{C}}-\overline{\bigcup_{\ell\in 	\mathcal{L}(G)}\ell }  \,. \\ 
	\end{array}
	\]
\end{definition}


Now we have:

\begin{lemma} \label{l:0}
	For each $\ell \in 	\mathcal{L}(G)$  there exist a sequence $(g_n)\subset G$ of distinct elements and $P \in {\rm SP}(3,\Bbb{C}) $, such that the sequence $\Pi(g_n)$ converges to  $Id$,  $(g_n)$ converges to  $P$ and $Ker(P)=\ell$.
	
\end{lemma}
\begin{proof}
	Let  $\ell \in 	\mathcal{L}(G)$, then   there exists a sequence $(h_n)\subset G$ of distinct elements and $P \in \SP(3,\Bbb{C}) $, such that $(\Pi(h_n))$ is a convergent sequence,  $h_n$ converges to  $P$ and $Ker(P)=\ell$.
	So we can  assume that:
	\[
	h_n=	
	\begin{bmatrix}
	1 & x_n & y_n \\
	0 & 1 & z_n \\
	0 & 0 & 1 \\
	\end{bmatrix}
	\textrm{ and }
	P=
	\begin{bmatrix}
	0& x & y \\
	0 & 0& 0 \\
	0 & 0 & 0 \\
	\end{bmatrix}.
	\]
	Set $a_n=max\{\mid x_n\mid, \mid y_n\mid \}$,  then it is not hard to check that there is    a  subsequence $(m_n)\subset (n)$ such that $a_na_{m_n}^{-1}$ converges to $0$ as $n$ goes to $\infty$. Then
	\[
	g_n=
	h_{n}^{-1}h_{m_n}=
	\begin{bmatrix}
	1 & x_{m_n}-x_n & -y_n+x_n z_n+y_{m_n}-x_n z_{m_n} \\
	0 & 1 & z_{m_n}-z_n \\
	0 & 0 & 1 \\
	\end{bmatrix}.
	\]
	Clearly $\Pi(g_n)$ converges to $Id$ and $g_n \xymatrix{
		\ar[r]_{n \rightarrow \infty}&}  P$.  
\end{proof}

\begin{lemma} \label{l:omega}
	Let $G$ be  discrete and such that   $Ker(\Pi\vert_G) $ is infinite and $\Pi(G)$ is non-trivial. If $\Omega(G)$ is non-empty, then:
		\begin{enumerate}
		\item \label{i:ome1}  The group $G$ acts properly  discontinuously on $\Omega(G)$;
		\item  \label{i:ome2}  The set $\Omega(G)$ is the largest open set on which $G$ acts properly discontinuously.
		\item   \label{i:ome3} Each connected component of   $\Omega(G)$ is homeomorphic to $\R^4$.
	\end{enumerate}
\end{lemma}
\begin{proof} 
	We first prove (\ref{i:ome1}). It is clear that $\Omega(G)$ is  open  and $G$-invariant;  also,  since $Ker(\Pi\vert_G)$ is infinite and $\Pi(G)$ is non-trivial, we have $\overleftrightarrow{e_1,e_2}\subset \P^2-\Omega(G)$.
	Now 	let $K \subset \Omega(G)$   be a compact set and $K(G ) = \{g \in G   : g(K ) \cap K\neq  \emptyset \}$. Assume that $K ( G)$ is infinite. Let $(g_n )$ be an enumeration of $K ( G)$, then there exists a subsequence of $(g_n)$, still denoted $(g_n)$, such that     either $( \Pi(g_n ))$ converges to a projective transformation or $ \Pi(g_n )\xymatrix{		\ar[r]_{n \rightarrow \infty}&} [e_2]$  uniformly on $\overleftrightarrow{e_2,e_3}-\{e_2\}$.  If $( \Pi(g_n ))$ converges to   a  projective transformation, we can find a subsequence $(h_n)\subset (g_n)$ and $\alpha\in L(G)$ such that $h_n$ converges to $\alpha$. Thus  $Ker(\alpha)\in \mathcal{L}(G)$ and $Im(\alpha)=\{e_1\}$, therefore the accumulation set of    $\{h_n(K):n\in \Bbb{N}\}$ is $\{e_1\}$. Now, if   $ \Pi(g_n )\xymatrix{
		\ar[r]_{n \rightarrow \infty}&} [e_2]$  uniformly on $\overleftrightarrow{e_2,e_3}-\{e_2\}$, then
	\[
	\{ (g_n) : n \in \Bbb{N}\} \subset  \{g \in   \Pi(G) : g(\pi (K)) \cap  \pi(K) \neq \emptyset \},
	\] which is not possible since $\overleftrightarrow{e_1,e_2}\subset \P^2-\Omega(G)$ and  we have proved Part (\ref{i:ome1}). \\

	Now we prove (\ref{i:ome2}). 	  		Let  $\Omega\subset \Bbb{P}^2_{\Bbb{C}}$ be open, non-empty, $G$ invariant and such that $G$ acts properly discontinuously on $\Omega$ and   $\ell\in \mathcal{L}(G)$. Then there are $(g_n)\in B(G)$ and $P\in L(G)$ such that $Ker(P)=\ell$ and $(g_n)$ converges to $P$.  By  Lemma \ref{l:0}  we can 	   assume that  $\Pi(g_n)$ converges to $Id$. 
	Proceeding  as in 
	Lemma \ref{l:0}, we conclude  
	\[
	g_n^{-1} =
	\begin{bmatrix}
	1 & -x_n & x_nz_n -y_n\\
	0 & 1 & -z_n\\
	0 & 0 & 1
	\end{bmatrix}
	\xymatrix{ \ar[r]_{n \rightarrow  \infty}&}
	P \;.
	\]
	By  Lemma \ref{l:lambda}  we deduce  $\ell\cap \Omega=\emptyset$.\\

	Now we prove (\ref{i:ome3}).  If $\Pi(G)$ is discrete, then  $\Omega(G)$ is $\Omega(Ker(\Pi\vert_G) )$ by definition. From the  definition of $\Omega(Ker(\Pi\vert_G))$ we get that $\Omega(Ker(\Pi\vert_G))=Eq(Ker(\Pi\vert_G))$ and  from  Example \ref{e:if} we know that $Eq(Ker(\Pi\vert_G))$     is either $\C^2$ or $(\C\times \H^+)\cup (\C\times \H^-) $, proving the claim. So we now assume that $\Pi(G)$ is non-discrete;   define  $$\mathcal{C}(G)=\overleftrightarrow{e_2,e_3}\cap \overline{\bigcup_{\ell\in 	\mathcal{L}(G)}\ell }   \;,$$
	then $\mathcal{C}(G)$ is a closed $\Pi(G)$-invariant  set in $\overleftrightarrow{e_2,e_3}$, so $\mathcal{C}(G)$  is closed in $\P^2$ and $\overline{\Pi(G)}$-invariant  set.  Since $\Pi(G)$ is non-discrete,  there exists  an additive  Lie subgroup $H\subset \C $ such that  $\overline{\Pi(G)}=\{z+b:b\in H\}$.  Since   $\mathcal{C}(G)-\{e_2\}=\overline{\Pi(G)}(\mathcal{C}(G)-\{e_2\})$, we deduce that 
	$\P^2-\Omega(G) $  is a pencil of lines over  a  union of real projective lines  in $\overleftrightarrow{e_2,e_3}$, which are pairwise parallel in $\overleftrightarrow{e_2,e_3}-\{e_2\} $. Thus each connected component of  $\Omega(G)$ is a fiber bundle over  $\R \times \R$ with fiber $\C$, hence homeomorphic to $\R^4$.
	   \end{proof}

\begin{definition}
	We say that 	the sequences  $(a_n,b_n),(x_n, y_n)\subset  \C^2$ {\it are  co-bounded}    if   both sequences converge to $\infty$ and the   sequence
	$
	\left (\frac{\mid a_n\mid+\mid b_n\mid }{\mid x_n\mid+\mid y_n\mid }
	\right )
	$ is bounded and bounded away from $0$.
\end{definition}

\begin{lemma} \label{l:pseudolatice}
	Let $(a_n),(b_n),(c_n), (x_n),(y_n),(z_n)\subset  \Bbb{C}$ be sequences of distinct elements such that:
	\begin{enumerate}
		\item 	 $(c_n)$ and $(z_n)$ converge to $0$, 
		\item $(a_n,b_n),(x_n,y_n)$ are  co-bounded;
		\item $[a_n:b_n] \xymatrix{
			\ar[r]_{n \rightarrow \infty}&} [a:b]$   for some $a,b$;
		\item $[x_n:y_n]\xymatrix{
			\ar[r]_{n \rightarrow \infty}&} [x:y]$   for some $x,y$;
		\item $[a:b]\neq [x:y]$.
		
	\end{enumerate}
	Then  there exists  $w\in \C\setminus\{0\}$ such that  for each $k,m\in \Bbb{N}\setminus\{0\}$ we get: 
	\[
	g(n,k,m)=
	\begin{bmatrix}
	1 & a_n & b_n\\
	0 & 1 & c_n\\
	0 &0  &1
	\end{bmatrix}^k
	\begin{bmatrix}
	1 & x_n & y_n\\
	0 & 1 & z_n\\
	0 &0  &1
	\end{bmatrix}^m
	\xymatrix{	\ar[r]_{n \rightarrow \infty}&} 
	\begin{bmatrix}
	0 & ka+mxw &kb+ myw\\
	0 & 0 & 0\\
	0 &0  &0
	\end{bmatrix} \;.
	\]
	
\end{lemma}
\begin{proof} Define  $r_n=max\{\vert a_n\vert, \vert b_n\vert \},$ $s_n=max\{\vert x_n\vert, \vert y_n\vert \}$ and $t_n=max\{s_n,r_n\}$.  Since  $(a_n,b_n),(x_n,y_n)$ are  co-bounded we can assume  there are $r,s\in \Bbb{R}\setminus\{ 0\}$ such that $r_nt_n^{-1}  \xymatrix{ \ar[r]_{n \rightarrow \infty}&} r$  and $s_nt_n^{-1}  \xymatrix{ \ar[r]_{n \rightarrow \infty}&} s$. Moreover, since  $[a_n:b_n] \xymatrix{
		\ar[r]_{n \rightarrow \infty}&} [a:b]$ and 
	$[x_n:y_n]\xymatrix{
		\ar[r]_{n \rightarrow \infty}&} [x:y]$,
	we deduce  that there are $u,v\in \C^*$ such that 
	
	\[
	\begin{array}{l}
	r_n^{-1}(a_n,b_n) \xymatrix{
		\ar[r]_{n \rightarrow \infty}&} u(a,b) \;,\\ 
	s_n^{-1}(x_n,y_n)\xymatrix{
		\ar[r]_{n \rightarrow \infty}&} v(x,y) \;.
	\end{array}
	\]
	Then an easy  computation  shows:
	
	\[
	g(n,k,m)
	\xymatrix{ \ar[r]_{n \rightarrow \infty}&} 
	\begin{bmatrix}
	0 & ka+mxw &kb+ myw\\
	0 & 0 & 0\\
	0 &0  &0
	\end{bmatrix} \,,
	\]
	where  $w=vs(ur)^{-1}$.
\end{proof}

\begin{lemma} \label{l:controlreal}
	Let $G\subset {\rm Heis}(3,\Bbb{C})$ be a Kleinian group such that     $\Pi(G)$ is non-discrete and $\Bbb{P}^2_{\Bbb{C}}-\Omega(G)$ contains more than a line. Then $\overline{\Pi(G)}$ is isomorphic to $\Bbb{R}$.
\end{lemma}
\begin{proof} We have that $\Pi(G)$ is a non-discrete subgroup of $\C$, thus  $\overline{\Pi(G)}$ must be   isomorphic to $\C$,  $\R \oplus \Z$  or $\R$,   see Theorem 3.1 in \cite{wal1} . Since  $G$ is complex Kleinian  we deduce that $\overline{\Pi(G)}$ is  isomorphic to either    $\R \oplus \Z$  or $\R$. Let us assume that $\overline{\Pi(G)}$ is isomorphic to   $\R \oplus \Z$. After conjugation, if necessary,  we can assume that  there exists  $s>0$   such   that $\overline{\Pi(G)}=\{r+msi:r\in \R, m\in\Z\}$. Moreover, since $\mathcal{L}(G)$ contains more than a line, we can find a line $\ell\in \mathcal{L}(G)$ containing $e_1$ such that $\ell= \overleftrightarrow {e_1,[0:u:1]}$ where $Im(u)\neq 0 $. On the other hand,  by Lemma \ref{l:0} we can find   $(g_n)\subset G $ and $P\in \SP(3,\C)$     such that 
	$\Pi(g_n) \xymatrix{	\ar[r]_{n \rightarrow \infty}&} Id$, 
	$g_n  \xymatrix{	\ar[r]_{n \rightarrow \infty}&} P$ and $\ell=Ker(P)$. Thus  there are sequences $(a_n),(b_n), (c_n)\subset \C$ such that 
	$max\{\vert a_n\vert ,\vert b_n\vert \}   \xymatrix{	\ar[r]_{n \rightarrow \infty}&}\infty $,
	$c_n   \xymatrix{	\ar[r]_{n \rightarrow \infty}&} 0$, $[a_n:b_n ]  \xymatrix{	\ar[r]_{n \rightarrow \infty}&} [a:b]$ and 
	\begin{equation}\label{e:gn}
	g_n=
	\begin{bmatrix}
	1 & a_n & b_n \\
	0 & 1 & c_n\\
	0 &0 & 1
	\end{bmatrix};\,\,\,\,\,
	P=
	\begin{bmatrix}
	0 & a & b \\
	0 & 0 &0\\
	0 &0 & 0
	\end{bmatrix}\,.
	\end{equation}
	Thus    	$\ell=\overleftrightarrow{[0:b:-a] ,e_1}$ and $b=-ua$.\\
	
	\noindent
	Claim 1. There are  functions    $f_2:\Z\rightarrow  \{\textrm{real projective subspaces  of } \overleftrightarrow{e_2,e_3}\}$ and    $f_1:\Z\rightarrow \Bbb{C}$ such that:
	\begin{enumerate}

	\item $Sgn(Im(f_1(m)))=Sgn(-m)$ for $m$ large,  here $Sgn$ is the function sign;
	\item $\vert Im(f_1(m)) \vert    \xymatrix{	\ar[r]_{\vert m \vert \rightarrow  \infty}&} \infty $;
	\item the point $[0:b:-a]$ is in $\bigcap_{m\in \Z} f_2(m)$;
	\item for each $m\in \Z$ we have $[0:f_1(m):1]\in f_2(m)$;
	\item $\bigcup_{m\in \Z}\bigcup_{p\in f_2(m)}\overleftrightarrow{e_1,p}\subset \Bbb{P}^2_{\C}\setminus\Omega(G)$.\\ 
	 \end{enumerate}
	
	Let $h \in G$ be of the form
	$$
	h=
	\begin{bmatrix}
	1 & x & y\\
	0 & 1 &  is\\
	0 & 0 & 1\\
	\end{bmatrix} \; , \; \hbox{with} \; \; x, y \in \C \; \hbox{and} \; s \in \R \;.
	$$
	
	If $g_n$ is given by Equation \ref{e:gn}, then for each $m\in \Z$ we have:
	
	\[
	h^{-m} g_n h^m=
	\begin{bmatrix}
	1 & a_n & b_n-c_n m x+is m a_n \\
	0 & 1 & c_n \\
	0 & 0 & 1 \\
	\end{bmatrix}
	\xymatrix{	\ar[r]_{n \rightarrow \infty}&}\
	\begin{bmatrix}
	0 & a & b+is m a \\
	0 & 0 & 0 \\
	0 & 0 & 0 \\
	\end{bmatrix}.
	\]
If for each $m\in \Z$	 we apply  Lemma 	 \ref{l:pseudolatice} to the respective sequences induced by the sequences $(g_{n})_{n\in \Bbb{N}}$ and $(	h^{-m} g_n h^m)_{n\in \Bbb{N}}$
	we deduce there exists  $w_m\in \C^*$ such that: 
	\[
	C_m=
\bigcup_{k,l\in \Bbb{Z}.}
	\overleftrightarrow
	{
		e_1,
		[
		0:
		kbw_m+l(b+is m a):
		- ka w_m-la
		]
	}
	\subset \P^2-\Omega(G) \;.
	\]
	If for each $m\in \Z$   we define $g_2(m )$  as  the closure of the set 
	$$
	[Span_\Bbb{\Z}( \{w_m(b, -a),(b+is m a,-a)
	\})-\{0\}],
	$$ 
	then by Lemma   5.11 in \cite{BCN} we have that $g_2(m)$  
	is a real projective space that contains $[0:b:-a]$. Now define $f_1(m)=u-is m $  and observe  that  $C_m=\bigcup_{p\in f_2(m)}\overleftrightarrow{e_1,p}$  and  $[0:u-is m a:1]\in f_2(m) $ for all $m\in \Z$.\\
	
	To conclude the proof let $f_1$ and $f_2$ be the functions given above, then  
$$
G\left (\overline{\bigcup_{m\in \Z}\bigcup_{p\in f_2(m)}\overleftrightarrow{e_1,p}}\right )=
\overline{
\bigcup_{m\in \Z} \bigcup_{p\in \Pi(G)f_2(m)}\overleftrightarrow{e_1,p}
}
 = \Bbb{P}^2_{\C} \;.$$
	This  yields $\Omega(G)=\emptyset$, which is a contradiction. 	
\end{proof}

The proof of  the following lemma is straightforward and it is  left to the reader:


\begin{lemma}\label{mathfrak h}
Set:
$$ \mathfrak{h}_\C=
	\left \{
	\begin{pmatrix}
	0 & a &b\\
	0 &0& c\\
	0 & 0 &0
	\end{pmatrix}:a,b,c\in \C
	\right \} \;.
	$$
	Then the map ${\bf exp}:\mathfrak{h}_\C \rightarrow  {\rm Heis}(3,\C)\,,$ given by
	$$
	\begin{array}{c}
		{\bf exp}
	\begin{pmatrix}
	0 & a &b\\
	0 &0& c\\
	0 & 0 &0
	\end{pmatrix}
	=
	\begin{bmatrix}
	1 & a &b+2^{-1}ac\\
	0 &1& c\\
	0 & 0 &1
	\end{bmatrix}
	\end{array}
	$$
	is a diffeomorphism with inverse
	${\bf log}: {\rm Heis}(3,\C)\rightarrow \mathfrak{h}_\C$ given by 
	
	$$
	{\bf log }
	\begin{bmatrix}
	1 & a &b\\
	0 &1& c\\
	0 & 0 &1
	\end{bmatrix}
	=
	\begin{pmatrix}
	0 & a &b-2^{-1}ac\\
	0 &0& c\\
	0 & 0 &0
	\end{pmatrix}.
	$$
\end{lemma}

Now we prove  Theorem \ref{semidirec}   stated in the introduction:


\begin{proof}
	Let us show part  (\ref{p:1r}).	We know that $G$ is finitely generated, therefore    $\Pi(G)$ is finitely generated. If $n=rank(\Pi(G))$, let $H=\{g_1,\ldots,g_n\}\subset G$ be such that $\Pi(G)$ is generated by $\Pi(H)$.   Let us consider the following  equivalence  relation in $G$: let us say that $a\sim b$ if and only if $[a,b]=Id$.  If $A_1,\ldots, A_n$ are the equivalence classes in $G$ induced by $\sim$, then define $B_0=Ker(\Pi\vert_G) )$  and $B_i=\langle A_i\rangle$. Now it is clear that $G=B_0\rtimes \cdots \rtimes B_n$.\\ 
	
	Now let us prove  Part (\ref{p:2r}). Let $\{h_1,...h_k\}\subset Ker(\Pi\vert_G) $ be a minimal generating set for $Ker(\Pi\vert_G) $ and  let $\{g_1,\ldots,g_n\}\subset G$  be such that  $\{\Pi(g_1),\ldots,\Pi(g_n)\}$  is  a minimal generating set for  $\Pi(G)$.  
	Set  $$V=\left \{\sum_{j=1}^k  \alpha_j{\bf log} (h_j)+\sum_{j=1}^n \beta_j {\bf log} (g_j):k_j,l_j\in \Z \right \}.$$
	
	\noindent
	Claim 1. If $\mathfrak{h}_\C$ is as in Lemma \ref{mathfrak h}, then 
	$V$ is an  additive subgroup of $\mathfrak{h}_\C$ with rank $n+k$.
	For this,
	assume there are $\alpha_j, \beta_j $'s in $\Z$ such that  $$J=\sum_{j=1}^k \alpha_j {\bf log} (h_j)+\sum_{j=1}^n \beta_j {\bf log}  (g_j)=0.$$  
	We can assume that the $g_j$ and $h_j$ can be expressed in the following way:
	$$
	h_j=
	\begin{bmatrix}
	1 & u_j &v_j\\
	0 &1& 0\\
	0 & 0 &1
	\end{bmatrix} \textrm{ and } 	 
	g_j=
	\begin{bmatrix}
	1 & x_j &y_j\\
	0 &1& z_j\\
	0 & 0 &1
	\end{bmatrix} \;.
	$$
Since the  $h_j $ generate $Ker(\Pi\vert_G)$ we have that $Span_{\Z}\{(u_j,v_j):j=1,\ldots, k\}$ is discrete;  and since the $\Pi(g_j)$ generate $\Pi(G)$ we get that  $\{z_1,\ldots z_n\}$ is a $\Z$-linearly independent set. An easy computation shows:
	$$
	J=\begin{pmatrix}
	0 & \sum_{j=1}^k
	\alpha_j u_j+\sum_{j=1}^n
	\beta _j x_j &\sum_{j=1}^k
	\alpha_j v_j+\sum_{j=1}^n
	\beta_j(y_j-2^{-1}x_jz_j)\\
	0 &0& \sum_{j=1}^n
	\beta_jz_j\\
	0 & 0 &0
	\end{pmatrix}.
	$$
	Since  ${\bf exp}(J)=Id$, we get the following system of equations:
	
	\begin{eqnarray*}
		\sum_{j=1}^k
		\alpha_j u_j+\sum_{j=1}^n
		\beta_j x_j =0,\\
		\sum_{j=1}^n
		\beta_jz_j=0,\\
		\sum_{j=1}^k
		\alpha_j v_j+\sum_{j=1}^n
		\beta_j(y_j-2^{-1}x_jz_j)=0.
	\end{eqnarray*}
	Since $\{z_1,\ldots z_n\}$ is linearly independent over $\Z$  we conclude $\beta_1=\ldots=\beta_n=0$. Hence  $ \sum_{j=1}^k
	\alpha_j( u_j,v_j)=0$ and therefore  $\alpha_1=\ldots=\alpha_k=0$, proving the claim.\\

	Let us define 
	$$
	\sqrt{[G,G]}=\{h\in \Heis(3,\C):h^2\in[G,G]\}.
	$$
It is clear that  	$\sqrt{[G,G]}$ is a discrete  subgroup   contained in the center of  $\Heis(3,\C)$.
	
	\vskip.2cm
	\noindent
	Claim 2. $\left \langle G\cup \sqrt{[G,G]}\right \rangle $ is a discrete subgroup of $\Heis(3,\C)$. Assume, on the contrary,   that there exists a sequence $(f_n)\subset \left \langle G\cup \sqrt{[G,G]}\right \rangle $   of distinct elements such that $f_n \xymatrix{ \ar[r]_{n \rightarrow  \infty}&}  Id$, thus $f_n^2 \xymatrix{ \ar[r]_{n \rightarrow  \infty}&}  Id$. Since $(f_n^2)\subset  G$ and $G$ is discrete we deduce
	  $f_n^2=Id$  for $n$ large,  which is a contradiction.\\
	
		\noindent
	Claim 3. ${\bf log}\left (\left \langle G\cup \sqrt{[G,G]}\right \rangle \right )$ is an additive discrete subgroup of $\mathfrak{h}_\C$. For this, let $a,b,c,x,y,z\in \C$ be such that:
$$
\gamma_1=
\begin{pmatrix}
	0 & a &b-2^{-1}ac\\
	0 &0& c\\
	0 & 0 &0
\end{pmatrix},
\gamma_2=
\begin{pmatrix}
0 & x &y-2^{-1}xz\\
0 &0& z\\
0 & 0 &0
\end{pmatrix}\in {\bf log}\left (\left \langle G\cup \sqrt{[G,G]}\right \rangle \right ).$$
An easy calculation shows:
	\[
	{\bf exp}(\gamma_1-\gamma_2)={\bf exp}(\gamma_1){\bf exp}(\gamma_2)^{-1}
	\begin{bmatrix}
	1 & 0&2^{-1}(az-cx)\\
	0 &1& 0\\
	0 & 0 &1
	\end{bmatrix} \,,
	\]
and 
	\[
	\begin{pmatrix}
1 & 0&2^{-1}(az-cx)\\
0 &1& 0\\
0 & 0 &1
\end{pmatrix}^2
=[{\bf exp}(\gamma_1), {\bf exp}(\gamma_2)] \;.
		\]
Hence  	${\bf exp}(\gamma_1-\gamma_2) \in 	\left \langle G\cup \sqrt{[G,G]}\right \rangle $. Since  ${\bf exp}$ is a diffeomorphism
with inverse   ${\bf log}$, 
 Claim 2 implies that  ${\bf log}\left (\left \langle G\cup \sqrt{[G,G]}\right \rangle \right )$ is  an additive  discrete group.\newline

To	finish the proof of  Part (\ref{p:2r}) we notice that    $V$ is a subgroup of the additive discrete group 
 $ {\bf log}\left (\left \langle G\cup \sqrt{[G,G]}\right \rangle \right )\subset \mathfrak{h}_\C$  and   $dim_\R(\mathfrak{h}_\C)=6$.\\

	Now we prove Part (\ref{p:3r}).  Let us assume that $G$ is  complex Kleinian. By  Lemma \ref{l:controlreal},   $G$ leaves invariant each connected component of $\Omega(G)$, and    each of these  is contractible  by Lemma  \ref{l:omega}. Hence,   by Theorem \ref{t:obdim}, the obstruction dimension of $G$  satisfies $obdim(G)\leq 4$. On the other hand  each $ B_j$ is a finitely generated, torsion free  group, and  it is well known that this kind of groups  are   semi-hyperbolic, see \cite{AB}.  Therefore Corollary \ref{c:obdim} yields 	$\sum_{i=0}^n obdim(B_i)\leq obdim (G)\leq 4$. Finally notice that $B_j=\Z^{k_j}$ for some  $k_j$ and $obdim(B_j)=k_j=rank(B_j)$ by \cite[2.2 ]{BF}. 
\end{proof}

\begin{corollary}\label{c:poly}
	If $G\subset {\rm Heis}(3,\C)$ is a discrete group, then $G$ is polycyclic.
	\end{corollary}

 Recall  that polycyclic means that the group  is solvable and every subgroup is finitely generated. Polycyclic groups actually are finitely presented \cite{LR}.

\subsection{Triangular purely parabolic groups with  trivial kernel} \label{s:ktrivial}

In this subsection we study purely parabolic groups with an invariant full flag and finite kernel. Now recall from Section \ref{s:screw} that $U_+$ is the subgroup of $\PSL(3,\C)$ of classes of  upper triangular matrices $(g_{ij})$ with 
	$g_{11}g_{22}g_{33}  =1$, and 
	we defined a group morphism 	$\Pi^*:U_+\rightarrow Mob (\C)$  by
	$$
	\Pi^*([g_{ij}])z=
	g_{11}g_{22}^{-1}z + g_{12}g_{22}^{-1} \;.
	$$

\begin{lemma} \label{l:pfd2}
	Let $G\subset {\rm Heis}(3,\C)$ be a commutative discrete group. If $Ker(\Pi^*\vert_G)$  and 	$Ker(\Pi\vert_G) $ are trivial, then  there exist   $W\subset \C$ an additive subgroup 	and 
	$ L:W\rightarrow \Bbb{C}$  a group morphism  such that:

	\begin{enumerate}

		\item\label{i:pa1} The group  $G$ is conjugate to:
		\[
		{\mathcal{K}_0}(W,L)=
		\left \{
		\left[		\begin{array}{lll}
		1 & \xi &L(\xi )+2^{-1}\xi^2 \\
		0 & 1 & \xi \\
		0 & 0& 1\\ 
		\end{array}
		\right]:\xi\in W
		\right \}.
		\]
		\item \label{i:pa2}The  Kulkarni limit set is: 
		\[
		\Lambda_{Kul}({\mathcal{K}_0}(W,L))=\overleftrightarrow{e_1,e_2}=\P^2-Eq(\mathcal {L}),
		\]
		and its complement $\Lambda_{Kul}({\mathcal{K}_0}(W,L))$  is the largest open set on which the group acts properly discontinuously.
		
		\item\label{i:pa3}  The group ${\mathcal{K}_0}(W,L)$ is free   with rank  at most four.
		\item \label{i:pa4}If $W$ is discrete then $L$ admits a linear extension to  the real vector space $Span_\R(W)$.
		
	\end{enumerate}
\end{lemma}

\begin{proof} Let us show Part (\ref{i:pa1}).
	Consider the following auxiliary function
	\[
	\begin{array}{c}
	\zeta:G\rightarrow \C^2\\
	g\mapsto (\pi_{12}(g),\pi_{23}(g)) \;.
	\end{array} 
	\]
	By definition  $\zeta$ is a monomorphism.  Set 
	\[
	\begin{array}{c}
	\kappa:\zeta(G)\rightarrow \Bbb{C}\\
	x \mapsto \pi_{13}(\zeta^{-1}(x)).
	\end{array} 
	\]
	It is clear that we have: 
	\[
	G
	=
	\left \{
	\begin{bmatrix}
	1 & \pi_1( x)&\kappa (x) \\
	0 & 1 & \pi_2(x)\\
	0 & 0& 1\\ 
	\end{bmatrix}
	:x\in W
	\right \}.
	\]
	Now let  $x,y\in \zeta(G)$, then $A=B$ where $A,B$ are: 
	\[
	A= \begin{Small}
	\begin{bmatrix}
	1 & \pi_1(x)&\kappa (x) \\
	0 & 1 & \pi_2(x)\\
	0 & 0& 1\\ 
	\end{bmatrix}
	\begin{bmatrix}
	1 & \pi_1(y)&\kappa (y) \\
	0 & 1 & \pi_2(y)\\
	0 & 0& 1\\ 
	\end{bmatrix}=
	\begin{bmatrix}
	1 & \pi_1(x+y)&\kappa (x)+\pi_1(x)\pi_2(y)+\kappa(y) \\
	0 & 1 & \pi_2(x+y)\\
	0 & 0& 1\\ 
	\end{bmatrix} \end{Small}
	\] 
	\[
	B= \begin{Small}
	\begin{bmatrix}
	1 & \pi_1(y)&\kappa (y) \\
	0 & 1 & \pi_2(y)\\
	0 & 0& 1\\ 
	\end{bmatrix}
	\begin{bmatrix}
	1 & \pi_1(x)&\kappa (x) \\
	0 & 1 & \pi_2(x)\\
	0 & 0& 1\\ 
	\end{bmatrix}
	=
	\begin{bmatrix}
	1 & \pi_1(x+y)&\kappa (x)+\pi_1(y)\pi_2(x)+\kappa(y) \\
	0 & 1 & \pi_2(x+y)\\
	0 & 0& 1\\ 
	\end{bmatrix} \end{Small}
	\]
	Then for every  $x,y\in \zeta(G)$ we have:
	\[
	\begin{array}{l}
	\kappa(x+y)=\kappa(x)+\kappa(y)+\pi_1(x)\pi_2(y) \,,\\
	\pi_1(x)\pi_2(y)=\pi_1(y)\pi_2(x) \,.
	\end{array}
	\]
	By Lemma  \ref{l:rad} there exists   an additive subgroup $W\subset \C$ and $\mu\in \C^* $ such that $\zeta(G)=W(1,\mu)$. Let us define 
	\[
	h=
	\begin{bmatrix}
	1 & 0 &0\\
	0 & \mu^{-1/2} &0\\
	0 &	0                 & \mu^{1/2} \\
	\end{bmatrix} \;,
	\] 
	and observe that: 
	\[
	h Gh^{-1}
	=
	\left \{
	\begin{bmatrix}
	1 & \xi &\widetilde {\kappa} (\xi) \\
	0 & 1 & \xi \\
	0 & 0& 1\\ 
	\end{bmatrix}
	:\xi \in W
	\right \},
	\]
	where  $\widetilde {\kappa}:W\rightarrow \C$ satisfies	 $\widetilde {\kappa}(\xi_1+\xi_2)=\widetilde{\kappa}(\xi_1)+\widetilde{\kappa}(\xi_2)+\xi_1\xi_2$. To conclude define   $L:W\rightarrow \C$
	by $L(\xi)=\widetilde {\kappa}(\xi)-2^{-1}\xi^2$, then:
	\[
	L(\xi_1+\xi_2)=\widetilde {\kappa}(\xi_1+\xi_2)-2^{-1}(\xi_1+\xi_2)^2=
	\widetilde {\kappa}(\xi_1)+\widetilde {\kappa}(\xi_2)-2^{-1}\xi_1^2-2^{-1}\xi_2^2=
	L(\xi_1)+L(\xi_2),
	\]
	proving Part (\ref{i:pa1}).\\
	
	Let us prove Part  (\ref{i:pa2}). Let $(g_m)\subset G$ be a sequence of distinct elements of $G$, then there exists  $(x_n)\subset W$ a sequence of distinct elements such that
	\[
	g_m=
	\begin{bmatrix}
	k_m^{-1} & x_mk_m^{-1} & k_m^{-1}(L(x_m)+x^2/2)\\
	0 & k_m^{-1} & x_m k_m^{-1}\\
	0&0&k_m^{-1}\\
	\end{bmatrix} \,,
	\]
	where $k_m=max\{\vert x_m\vert, \vert L(x_m)+x_m^2/2\vert \}$. If $(g_{n_m})$ is a subsequence of $(g_m)$ such that $(g_{n_m})$  converges to $P\in \SP(3,\C)-\PSL(3,\C)$, then there are $a,b\in \C$ such that $\mid a\mid+\mid b\mid \neq 0 $ and 
	\[
	P=
	\begin{bmatrix}
	0 & a& b\\
	0 &0 &a\\
	0 &0 &0\\
	\end{bmatrix} \,.
	\]
	This shows that $Eq(G)=\P^2-\overleftrightarrow{e_1,e_2}=\C^2$ since,  by  Proposition \ref{p:pkg}, $Eq(G)\subset \Omega_{Kul}(G)$ and $\Lambda_{Kul}(G)$ always contains a line.
	Then $Eq(G)= \Omega_{Kul}(G)$. If $\Omega\subset \P^2$ is any open set on which $G$ acts properly discontinuously, then $\P^2-\Omega$ contains a complex line, say $\ell$. If $g\in G-\{Id\}$, then $g^{m}\ell \xymatrix{
		\ar[r]_{m \rightarrow \infty}&} \overleftrightarrow{e_1,e_2}$.\\

	In order to prove Part (\ref{i:pa3}) we observe that $G$ is an  group acting properly discontinuously and freely on $\C^2$, thus the rank of $G$ must be at most four, see the proof of Proposition 5.9 in  \cite{BCN1}. The last part of the theorem  is immediate.
\end{proof}

As a consequence  of Lemma 5.10 in  \cite{BCN} we get the following result.

\begin{lemma} \label{l:pfi2a}
	Let $G\subset {\rm Heis}(3,\C)$ be a commutative discrete  group, then:
	\begin{enumerate}
		\item  \label{l:ipfi1}
		If $Ker(\Pi^*\vert_G)$ is non-trivial, then
		there is a discrete  additive  subgroup $\mathcal{L}$ of $\C^2$ with rank at most four, 
		such that:   
		\[
		G=
		\left 
		\{
		\begin{bmatrix}
		1 & 0 &y\\
		0 & 1 & z\\
		0 & 0 & 1
		\end{bmatrix}
		:
		(y,z)\in \mathcal{L}
		\right
		\} \,.
		\] 
		\item  \label{l:ipfi2}
		If $Ker(\Pi\vert_G) $ is non-trivial, then there    exists a discrete    additive subgroup  $\mathcal{L}\subset \Bbb{C}^2$  such that:
		\[ G=
		\left \{
		\begin{bmatrix}
		1 & x&y \\
		0 & 1 & 0\\
		0 & 0& 1\\ 
		\end{bmatrix}
		:x\in \mathcal{L}
		\right \}.
		\]
		\hskip-15pt Moreover, if $G$ is complex Kleinian, then $\mathcal{L}$ has rank at most 2.
	\end{enumerate}				
\end{lemma}

\subsection{Groups with infinite Kernel}   \label{s:kinfinite}	
	We consider now discrete groups $G\subset \Heis(3,\C)$ whose control map has infinite kernel, {\it i.e.}, 
	$Ker(\Pi\vert_G)$ is infinite.

\begin{lemma}  \label{core}
	If $G$ is complex Kleinian group  with infinite kernel, then:
	\begin{enumerate}
		\item  \label{l:ic2} We have that  $Ker(\Pi\vert_G) =\Z^k$ where $1\leq k\leq 2$.
		\item   \label{l:ic3} We have that $\Lambda_{Kul}(Ker(\Pi\vert_G) )=L_0(Ker(\Pi\vert_G) )$ is either a line or a pencil of lines over a circle, where  $L_0$ is as in Definition \ref{def Kulkarni limit set}.
		\item  \label{l:ic4} If the set  $ \Lambda_{Kul}(Ker(\Pi\vert_G) )$ is a   line,  then there exists    a  discrete additive subgroup  $W$ of $\C$ such that $G$ is conjugate to: 
		\[
		G_W=
		\left 
		\{
		\begin{bmatrix}
		1 & 0 & w\\
		0 & 1& 0\\
		0 & 0& 1\\
		\end{bmatrix}: w\in W
		\right 
		\}
		\]
		and $rank(Ker(\Pi\vert_G) )\leq 2$.
		\item  \label{l:ic5}
		If  $\Lambda_{Kul}(Ker(\Pi\vert_G) )$ is a pencil of lines over a circle, then  the rank  of  $Ker(\Pi\vert_G) $ is two and  the  groups  $\Pi^*(Ker(\Pi\vert_G) )$ and  $\pi_{23}(Ker(\Pi\vert_G) )$ are non-trivial.

		\item \label{l:ic6} If the group $\Pi(G)$ is non-trivial.  Then 
		the group 	 $\Pi^*(Ker(\Pi\vert_G) )$ is non-trivial  if and only if $\Lambda_{Kul}(Ker(\Pi\vert_G) )$ is a pencil of lines over a circle.

	\end{enumerate}
\end{lemma}	

\begin{proof}
	The proofs of parts  (\ref{l:ic2}) and (\ref{l:ic3}) follow from Example \ref{e:if}.	 Let us prove
	Part (\ref{l:ic4}).  It is clear  that there exists $\mathcal{L}\subset \C^2$ an $\R$-linearly independent set  such that $G=\mathcal{T}^*(\mathcal{L})$,  where $G=\mathcal{T}^*(\mathcal{L})$ is  given as in Example  \ref{e:if}. We know that 
	$$
	\Lambda_{Kul}( Ker(\Pi\vert_G) )= \bigcup_{p\in S}
	\overleftrightarrow{e_1,p}\,,
	$$
	where $S$ is the closure of the set $\big\{Span_\Bbb{Z}\{(y,-x):(x,y)\in \mathcal{L}\}\setminus\{0\}\big\}$.   Since $\Lambda_{Kul}( Ker(\Pi\vert_G) )$ is a single line, from Lemma  \ref{{ladd}} we deduce   that $S$ is either a single point or it contains exactly two $\C$-linearly dependent vectors. Let us assume 
		that $S$ contains exactly two $\C$-linearly dependent vectors, the other case is similar;  so 
		there exists $\alpha\in \C$  and  $(x,y)\in \mathcal{L}$, such that  one has:
	
	\[
	G=
	\left
	\{
	\begin{bmatrix}
	1 & (n +m\alpha )x & (n +m\alpha )y\\
	0 & 1 &0\\
	0 & 0 &1\\
	\end{bmatrix}: n,m\in \Z
	\right 
	\} \,.
	\]
	
	Let $r\in \R^*$ be such that $x\neq yr$,  then a simple computation shows:  
	
	\[
	\begin{bmatrix}
	1&0&0\\
	0 & r& 1\\
	0& x & y
	\end{bmatrix}
	\begin{bmatrix}
	1 & (n +m\alpha )x & (n +m\alpha )y\\
	0 & 1 &0\\
	0 & 0 &1\\
	\end{bmatrix}
	\begin{bmatrix}
	1&0&0\\
	0 & r& 1\\
	0& x & y
	\end{bmatrix}^{-1}
	=
	\begin{bmatrix}
	1 & 0 & m+n\alpha  \\
	0 & 1 & 0 \\
	0 & 0 & 1\\
	\end{bmatrix}\,,
	\]
	proving  (\ref{l:ic4}). Notice that 
	    Part (\ref{l:ic5})  follows from  Example \ref{e:if}, so let us prove 
	(\ref{l:ic6}). 
	
	Since $\Pi^*(Ker(\Pi\vert_G) )$ and $\Pi(G)$ are both non-trivial, we deduce that there are $a,b,x,y,z\in \C$  and $g,h \in G$ such that $az\neq 0$ and
	\[
	g=
	\begin{bmatrix}
	1 & a & b\\
	0 & 1 &0\\
	0 & 0 &1\\
	\end{bmatrix};\,\, \,
	h=
	\begin{bmatrix}
	1&x&y\\
	0 & 1& z\\
	0& 0& 1
	\end{bmatrix} \,.
	\]
	By  a straightforward computation we  find:
	
	\[
	h gh^{-1}=
	\begin{bmatrix}
	1 & a & b-az\\
	0 & 1 &0\\
	0 & 0 &1\\
	\end{bmatrix}.
	\]
	In order   to conclude the proof we only need to observe that $(a,b)$ and $(a ,b-az)$ are $\C$-linearly independent vectors.
\end{proof}

As in \cite{BCN}, we use the notation 
$\mu(U)$ to denote  the maximum number of complex projective lines
in general position contained in $\P^2- U$.

\begin{lemma} \label{l:cif}
	Let $G\subset {\rm Heis}(3,\C)$ be complex Kleinian group such that  $\Pi(G)$ is non trivial and $\mu (\Omega_{Kul}(Ker(\Pi\vert_G )))=2$, then 
	\begin{enumerate}
		
		\item \label{l:ica1} The group $\Pi(G)$ is discrete.
		\item \label{l:ica2} The rank of $\Pi(G)$ is equal to one.
	\end{enumerate}
\end{lemma}

\begin{proof}
	Assume  $\Pi(G)$ is not discrete.  Then we can assume there exists  a sequence $(g_n)\subset G$ such that $\Pi(g_n )$ is a sequence of distinct elements  converging to $Id$.  On the other hand, since $\Lambda_{Kul}(Ker(\Pi\vert_G)) $ is a pencil of lines over a circle, there exists  $g\in Ker(\Pi\vert_G) $ such that $\Pi^*(g)\neq Id$.  If $g_n$ and $g$ are given respectively by
	\[
	g_n=
	\begin{bmatrix}
	1 & a_n & b_n\\
	0 & 1 &c_n\\
	0 & 0 &1\\
	\end{bmatrix};\,\, \,
	g=
	\begin{bmatrix}
	1&x&y\\
	0 & 1& 0\\
	0& 0& 1
	\end{bmatrix},
	\]
	then 
	\[
	g_ng g_n^{-1}=
	\begin{bmatrix}
	1 & x & y-xc_n\\
	0 & 1 &0\\
	0 & 0 &1\\
	\end{bmatrix}
	\xymatrix{ \ar[r]_{n \rightarrow  \infty}&} 
	\begin{bmatrix}
	1 & x & y\\
	0 & 1 &0\\
	0 & 0 &1\\
	\end{bmatrix},
	\]
	 which contradicts that $G$ is discrete.
	
	Now we assume that $\Pi(G)$  has rank $\ge 2$. Let $h_1,h_2,h\in G$ be such that $\langle \Pi(h_1), \Pi(h_2)\rangle=\Pi(G) $, $h\in Ker(\Pi\vert_G) $ and $\Pi^*(h)\neq Id$. Set:
	\[
	h_1=
	\begin{bmatrix}
	1 & a & b\\
	0 & 1 &c\\
	0 & 0 &1\\
	\end{bmatrix};\,\, \,
	h_2=
	\begin{bmatrix}
	1&x&y\\
	0 & 1& z\\
	0& 0& 1
	\end{bmatrix};\,\, \,
	h=
	\begin{bmatrix}
	1&u&v\\
	0 & 1& 0\\
	0& 0& 1
	\end{bmatrix},
	\]
	then
	
	\[
	[h^{-1},h_1]=
	\begin{bmatrix}
	1 & 0 & -uc\\
	0 & 1 &0\\
	0 & 0 &1\\
	\end{bmatrix};\,\, \,
	[h^{-1},h_2]=
	\begin{bmatrix}
	1 & 0 & -uz\\
	0 & 1 &0\\
	0 & 0 &1\\
	\end{bmatrix}.\,\, \,
	\]
	Then  $\{(u,v),(0,-uc), (0,-uz)\}$ is an $\R-$linearly independent set, which is not possible.
\end{proof}

\begin{proposition} \label{t:pifa}
	If $G\subset {\rm Heis}(3,\C)$ is    complex Kleinian  such that      $\Pi(G)$ is non trivial and  $\mu(\Omega_{Kul}(Ker \Pi \vert_ G))=2$.  Then there exist  $x,y\in \C$, $p,q,r\in \Z$ such that $p,q$ are co-primes, $q^2$ divides $r$ and  $G$ is conjugate to 
	\small{\[
	H=
	\left \{
	\left[
	\begin{array}{ccc}
	1 & k+l pq^{-1}+m x & lr^{-1} +m \left(k+lpq^{-1}\right)+ \begin{pmatrix} m\\ 2\end{pmatrix} x+m y\\
	0 & 1 & m \\
	0 & 0 & 1 \\
	\end{array}
	\right]:(k,l,m)\in \Bbb{Z}
	\right \}
	\;.
	\]}
\end{proposition}
\begin{proof}
	By Lemma \ref{l:cif} we know that $\Pi(G)$ is discrete and has  rank equal to 1;  and by Lemma \ref{core} we have $rank(Ker(\Pi\vert_G) )=2$ and $\Pi^*(Ker(\Pi\vert_G) )$ is non-trivial. Thus by Theorem  \ref{t:desc}   there exist $\{(a,b),(c,d)\}$ a $\C$-linearly independent set and $u,v,w\in \C$ such that 
	\begin{equation} \label{e:grupo}
	G=
	\left \{
	\begin{bmatrix}
	1 & ka+lc & kb+ld\\
	0 & 1 &0\\
	0 & 0 & 1
	\end{bmatrix}
	\begin{bmatrix}
	1 & u  & v\\
	0 & 1 &w\\
	0 & 0 & 1
	\end{bmatrix}^n: k,l,n\in \Z
	\right \},
	\end{equation}
	$aw\neq 0$.  A simple computation  shows: 
	\[
	\begin{bmatrix}
	\frac{1}{a} & 0 & 0 \\
	0 & 1 & \frac{b}{a} \\
	0 & 0 & w \\
	\end{bmatrix}
	\begin{bmatrix}
	1 & a & b\\
	0 & 1 & 0\\
	0 & 0 &1
	\end{bmatrix}
	\begin{bmatrix}
	\frac{1}{a} & 0 & 0 \\
	0 & 1 & \frac{b}{a} \\
	0 & 0 & w \\
	\end{bmatrix}^{-1}
	=
	\begin{bmatrix}
	1 & 1 & 0\\
	0 & 1 & 0 \\
	0 & 0 & 1 \\
	\end{bmatrix}
	=g_1 \,,
	\]
	
	\[
	\begin{bmatrix}
	\frac{1}{a} & 0 & 0 \\
	0 & 1 & \frac{b}{a} \\
	0 & 0 & w \\
	\end{bmatrix}
	\begin{bmatrix}
	1 & c & d\\
	0 & 1 & 0\\
	0 & 0 &1
	\end{bmatrix}
	\begin{bmatrix}
	\frac{1}{a} & 0 & 0 \\
	0 & 1 & \frac{b}{a} \\
	0 & 0 & w \\
	\end{bmatrix}^{-1}
	=
	\begin{bmatrix}
	1 & \frac{c}{a} & \frac{d}{a w}-\frac{b c}{a^2 w} \\
	0 & 1 & 0 \\
	0 & 0 & 1 \\
	\end{bmatrix}=g_2 \,,
	\]

	\[
	\begin{bmatrix}
	\frac{1}{a} & 0 & 0 \\
	0 & 1 & \frac{b}{a} \\
	0 & 0 & w \\
	\end{bmatrix}
	\begin{bmatrix}
	1 & u & v\\
	0 & 1 & w\\
	0 & 0 &1
	\end{bmatrix}
	\begin{bmatrix}
	\frac{1}{a} & 0 & 0 \\
	0 & 1 & \frac{b}{a} \\
	0 & 0 & w \\
	\end{bmatrix}^{-1}
	=
	\begin{bmatrix}
	1 & \frac{u}{a} & \frac{v}{a w}-\frac{b u}{a wa} \\
	0 & 1 & 1 \\
	0 & 0 & 1 \\
	\end{bmatrix}=g_3 \,.
	\]
	Now by Equation \ref{e:grupo} we deduce  that 
	\[
	G_1=\{g_1^kg_2^lg_3^n:k,l,n\in \Z\}
	\]
	is a group conjugate to $G$. On the other hand,  $G_1$ is a group  if and  only if   
	\[
	g_3 g_i g_3^{-1}\in \langle  g_1,g_2 \rangle  \textrm{ for } i=1,2.
	\] 
	The last statement is  equivalent to
	\[
	(0,1),(0,ca^{-1})\in Span_\Z(\{   (1,0),(ca^{-1}, (aw)^{-1}( d- bc a^{-1}   ))\}).
	\] 
	Now the conclusion  follows  from Lemma \ref{l:ratlat}.
\end{proof}

\begin{proposition} \label{t:pifa1}
	Let  $G\subset {\rm Heis}(3,\C)$ be   complex Kleinian  such that      $\Lambda_{Kul}(Ker(\Pi\vert_G) )$ is a line  and $\Pi(G)$ is discrete.
	Then  $G$ is conjugate to one of the following groups:
	
	\begin{enumerate}
		\item 
		\begin{equation}\label{e:fn1}
		\mathcal{T}(\mathcal{L})=
		\left \{
		\left[
		\begin{array}{lll}
		1 & 0  &y\\
		0 & 1 & z\\
		0 & 0& 1\\ 
		\end{array}
		\right]:(y,z)\in \mathcal{L}
		\right \},
		\end{equation}
		where  $\mathcal{L}\subset \C^2$ is an additive subgroup such that $\pi_2(\mathcal{L})$ is discrete.
		\item 
		\begin{equation}\label{e:fn2}
	{\mathcal{K}_0}(W_1,W_2,L)=	
		\left \{
		\begin{bmatrix}
		1 & x & L(x)+x^2/2 +w\\
		0 & 1& x\\
		0 & 0 &1\\
		\end{bmatrix}:
		w\in W_2, x\in W_1
		\right \},
		\end{equation}
		where  $W_1,W_2\subset \C$  are  additive discrete subgroups  and   $L:W_1\rightarrow \C$   is  a  group morphism.

		\item 
		\begin{equation}\label{e:fn4}
		\mathcal{K}=	
		\left \{
		\begin{bmatrix}
		1 &0  &w\\
		0 & 1 & 0 \\
		0 & 0 & 1 \\
		\end{bmatrix}
		\begin{bmatrix}
		1 &1 &x\\
		0 & 1 & 1 \\
		0 & 0 & 1 \\
		\end{bmatrix}^n
		\begin{bmatrix}
		1 &a &b\\
		0 & 1 & c \\
		0 & 0 & 1 \\
		\end{bmatrix}^m:m,n\in \Z, w\in W
		\right \},
		\end{equation}
		where $W\subset \C$ is an additive discrete subgroup, $a-c\in W$ and $c\notin \R$.
	\end{enumerate}
\end{proposition}
\begin{proof}
	Since $\Pi(G)$ is discrete   we deduce  $Ker(\Pi\vert_G) $ and $\Pi(G)$ are torsion free Abelian groups with   rank less than or equal to 2. For simplicity  we may assume that  $rank(Ker(\Pi\vert_G) )=rank(\Pi(G))=2$ since, as we will see in the proof, any other possibility will be covered by this case.  Now by Theorem  \ref{t:desc} there exist  $W\subset \C$ an additive discrete subgroup with rank 2  and $a,b,c,x,y,z\in \C$ such that: 
	\begin{equation} \label{e:fn0}
	G=
	\left \{
	\begin{bmatrix}
	1 & 0 & w\\
	0 & 1 &0\\
	0 & 0 & 1
	\end{bmatrix}
	\begin{bmatrix}
	1 & x  & y\\
	0 & 1 &z\\
	0 & 0 & 1
	\end{bmatrix}^m
	\begin{bmatrix}
	1 & a  & b\\
	0 & 1 &c\\
	0 & 0 & 1
	\end{bmatrix}^n: w\in W, m,n\in \Z
	\right \},
	\end{equation}
	and $zc^{-1}\notin\R$. Now consider the following cases:
	
	\vskip.2cm
	\noindent
	Case 1.  $xb-za=0$. Let us consider  the following sub-cases:
	
	\vskip.2cm
	\noindent
	Sub-case 1. $x=a=0$. Then from Equation  \ref{e:fn0} we see that  $G$ is conjugate to the torus group given by Equation \ref{e:fn1}. 
	
	\vskip.2cm
	\noindent
	Sub-case 2. $xa\neq 0$. Observe that:
	\[
	g_w=
	\begin{bmatrix}
	\frac{1}{x} & 0 & 0 \\
	0 & 1 & 0 \\
	0 & 0 & z \\
	\end{bmatrix}
	\begin{bmatrix}
	1&0   & w\\
	0 & 1 & 0 \\
	0 & 0 & 1 \\
	\end{bmatrix}
	\begin{bmatrix}
	\frac{1}{x} & 0 & 0 \\
	0 & 1 & 0 \\
	0 & 0 & z \\
	\end{bmatrix}^{-1}
	=
	\begin{bmatrix}
	1 & 0 & w(xz)^{-1} \\
	0 & 1 & 0 \\
	0 & 0 & 1 \\
	\end{bmatrix} \,,
	\]

	\[
	g=
	\begin{bmatrix}
	\frac{1}{x} & 0 & 0 \\
	0 & 1 & 0 \\
	0 & 0 & z \\
	\end{bmatrix}
	\begin{bmatrix}
	1&x   & y \\
	0 & 1 & z \\
	0 & 0 & 1 \\
	\end{bmatrix}
	\begin{bmatrix}
	\frac{1}{x} & 0 & 0 \\
	0 & 1 & 0 \\
	0 & 0 & z \\
	\end{bmatrix}^{-1}
	=
	\begin{bmatrix}
	1 & 1 & \frac{y}{x z} \\
	0 & 1 & 1 \\
	0 & 0 & 1 \\
	\end{bmatrix} \,,
	\]  
	\[
	h=
	\begin{bmatrix}
	\frac{1}{x} & 0 & 0 \\
	0 & 1 & 0 \\
	0 & 0 & z \\
	\end{bmatrix}
	\begin{bmatrix}
	1&a  & b \\
	0 & 1 & c \\
	0 & 0 & 1 \\
	\end{bmatrix}
	\begin{bmatrix}
	\frac{1}{x} & 0 & 0 \\
	0 & 1 & 0 \\
	0 & 0 & z \\
	\end{bmatrix}^{-1}
	=
	\begin{bmatrix}
	1 & \frac{a}{x} & \frac{b}{x z} \\
	0 & 1 & \frac{a}{x} \\
	0 & 0 & 1 \\
	\end{bmatrix}.
	\]  
	By Lemma \ref{l:pfd2} there  exists a group morphism   $L:W_1=Span_\Z(\{1,ax^{-1}\})\rightarrow \C$ such that:
	\[
	\langle h,g\rangle=
	\left
	\{
	\begin{bmatrix}
	1 & r & L(r)+2^{-1}r^2\\
	0 & 1 & r\\
	0 & 0 & 1\\
	\end{bmatrix}:
	r\in W_1
	\right 
	\}. 
	\]
	Now by  Equation \ref{e:fn0}  we have  that $G$ is conjugate to an Abelian  Kodaira group as in Equation  \ref{e:fn2}.
	
	From Lemma  \ref{core} it is clear that the previous cases cover all the possibilities for the case $xb-za=0$.\\
	
	\noindent
	Case 2. $xb-za\neq 0$.
	
	\vskip.2cm
		\noindent
Sub-case 1.  $x=0$, $a\neq 0$. The following equations:	\[
	g_w=
	\begin{bmatrix}
	\frac{1}{a} & 0 & 0 \\
	0 & 1 & 0 \\
	0 & 0 & z \\
	\end{bmatrix}
	\begin{bmatrix}
	1&0   & w\\
	0 & 1 & 0 \\
	0 & 0 & 1 \\
	\end{bmatrix}
	\begin{bmatrix}
	\frac{1}{a} & 0 & 0 \\
	0 & 1 & 0 \\
	0 & 0 & z \\
	\end{bmatrix}^{-1}
	=
	\begin{bmatrix}
	1 & 0 & w(az)^{-1} \\
	0 & 1 & 0 \\
	0 & 0 & 1 \\
	\end{bmatrix},
	\]

	\[
	g=
	\begin{bmatrix}
	\frac{1}{a} & 0 & 0 \\
	0 & 1 & 0 \\
	0 & 0 & z \\
	\end{bmatrix}
	\begin{bmatrix}
	1&0  & y \\
	0 & 1 & z \\
	0 & 0 & 1 \\
	\end{bmatrix}
	\begin{bmatrix}
	\frac{1}{a} & 0 & 0 \\
	0 & 1 & 0 \\
	0 & 0 & z \\
	\end{bmatrix}^{-1}
	=
	\begin{bmatrix}
	1 & 0 & \frac{y}{a z} \\
	0 & 1 & 1 \\
	0 & 0 & 1 \\
	\end{bmatrix},
	\]  
	\[
	h=
	\begin{bmatrix}
	\frac{1}{a} & 0 & 0 \\
	0 & 1 & 0 \\
	0 & 0 & z \\
	\end{bmatrix}
	\begin{bmatrix}
	1&a  & b \\
	0 & 1 & c\\
	0 & 0 & 1 \\
	\end{bmatrix}
	\begin{bmatrix}
	\frac{1}{a} & 0 & 0 \\
	0 & 1 & 0 \\
	0 & 0 & z \\
	\end{bmatrix}^{-1}
	=
	\begin{bmatrix}
	1 & 1 & \frac{b}{a z} \\
	0 & 1 & \frac{c}{z} \\
	0 & 0 & 1 \\
	\end{bmatrix},
	\]  
	together with  Equation \ref{e:fn0}, imply  that $G$ is conjugate   to a group of the form given by Equation \ref{e:fn4}.

	\vskip.2cm
		\noindent
	Sub-case 2.  $x\neq 0$ and $a\neq 0$. Then analogous arguments show that  $G$ is conjugate to a group of the form given  by Equation \ref{e:fn4}. 
\end{proof}

In a  similar way  one can show the following proposition:
\begin{proposition} \label{l:abcor2cind}
	Let $G\subset {\rm Heis}(3,\Bbb{C})$ be a  commutative complex Kleinian group such that  $Ker(\Pi\vert_G) $ is infinite and     $\Pi(G)$ is  non-discrete, then:
	
	\begin{enumerate} 
		\item If   $\Pi^*(G)$ is trivial, then there exists  $\mathcal{L}\subset \C^2$ an additive discrete subgroup of rank at most four, such that $\pi_2(\mathcal{L})$ is non-discrete and  $G$ is conjugate to: 
		\[
		\mathcal{T}(\mathcal{L})=
		\left
		\{
		\begin{bmatrix}
		1 & 0 & a\\
		0 & 1 & b\\
		0 & 0 & 1
		\end{bmatrix}:(a,b)\in \mathcal{L}
		\right 
		\} \,.
		\]
		\item  If   $\Pi^*(G)$  is non trivial,  then   there exist additive subgroups   $W_1,  W_2\subset \C$    such that   $W_1$ is non-discrete,  $W_2$ has rank  1 and   $G$ is conjugate to:
		
		\[
		{\mathcal{K}_0}(W_1,W_2,L)=	
		\left \{
		\begin{bmatrix}
		1 & x & L(x)+x^2/2 +w\\
		0 & 1& x\\
		0 & 0 &1\\
		\end{bmatrix}:
		w\in Span_\Z(W_2),\,  x\in Span_\Z(W_1)
		\right \}
		\]
		where  $L:W_1\rightarrow \C$   is a group morphism.
	\end{enumerate}
\end{proposition} 




\begin{lemma}\label{l:2genind}
	Let  $a,b,c\in \C$ and $r\in \Bbb{R}-\Bbb{Q	}$, and let		$G\subset {\rm Heis}(3,\Bbb{C})$ be  the group given by 
	\[
	G=
	\left 
	\langle 
	A=	\begin{bmatrix}
	1 & 1& 0\\
	0 & 1& 1\\
	0 & 0& 1\\
	\end{bmatrix},\,
	B=	\begin{bmatrix}
	1 & a+r& b\\
	0 & 1& r\\
	0 & 0& 1\\
	\end{bmatrix}		
	\right 
	\rangle \,.
	\]
	Then
	\begin{enumerate}
		\item \label{i:12gind} $G$ is commutative if and only  if $a=0$. 
		\item \label{i:122gind} If $a\neq 0$, then $\P^2-\Omega(G)$ is a cone of lines over a circle.

	\end{enumerate}
\end{lemma}  

\begin{proof}
	The proof of Part (\ref{i:12gind}) is straightforward.  To prove Part (\ref{i:122gind}) notice  that Theorem  \ref{t:desc} implies:
	\[
	G= \begin{small}
	\left 
	\{
	g_{mnk}=
	\begin{bmatrix}
	1 & (m+nr)+na & a k+
	\begin{pmatrix}
	m\\
	2
	\end{pmatrix}
	+m n r+
	\begin{pmatrix}
	n\\
	2
	\end{pmatrix}
	r(r+a ) \\
	0 & 1 & m+n r \\
	0 & 0 & 1 \\
	\end{bmatrix}:k,m,n\in\Z
	\right
	\}\end{small}
	\]
	By elementary algebra we have:

	\begin{small} 
		\begin{equation} \label{e:limset}
		\frac{a (2k+n^2 r)+rn (-a-r+1)+ (m+nr)^2-(m+nr)}{2}=a k+
		\begin{pmatrix}
		m\\
		2
		\end{pmatrix}
		+m n r+
		\begin{pmatrix}
		n\\
		2
		\end{pmatrix}
		r(r+a ).
		\end{equation}
	\end{small}

		\noindent
	Claim 1. $\P^2-\Omega(G)$ contains more than one line. To prove this claim it is enough to show that $\P^2-\Omega(G)$ contains a line different from $\overleftrightarrow{e_1,e_2}$.  Let $(a_n),(b_n)\in \Z$  be sequences such that $a_n+b_n r  \xymatrix{		\ar[r]_{n \rightarrow \infty}&}  0$; let us assume that  all the elements in the sequence $(a_n)$  are either odd or even. Let $k_0\in \Bbb{N}$ be an even number   such that  $$k_0\vert a\vert>\vert  r(-a-r+1)\vert,  $$
	and define the following sequence:
	
	\[
	c_n=
	\left \{
	\begin{matrix}
	2^{-1}b_n(a_n+k_0+1) & \textrm{ if } a_n  \textrm{ is odd} \,,\\
	2^{-1}b_n(a_n+k_0) & \textrm{ if } a_n  \textrm{ is even}\,.\\
	\end{matrix}
	\right. 
	\]
	Clearly $(c_n)\subset \Z $ and 
	\[
	g_{a_n,b_n,c_n}
	\xymatrix{	\ar[r]_{n \rightarrow \infty}&} 
	g=
	\begin{bmatrix}
	0 & 2a & w_0+r(-a-r+1)\\
	0 & 0 & 0\\
	0 & 0 &0\\
	\end{bmatrix}\,,
	\]
	where $w_0$ is either $k_0$ or $k_0+1$. Hence $Ker(g)$ is a complex line distinct from $\overleftrightarrow {e_1,e_2}$.\\
	
	\noindent
	Claim 2. $\P^2-\Omega(G)$ is contained in a pencil of lines over  an Euclidean circle.  Let $(a_n),(b_n),(c_n)\in \Z$  be sequences such that $a_n+b_n r  \xymatrix{		\ar[r]_{n \rightarrow \infty}&}  0$. Assume that
	\[
	g_{a_n,b_n,c_n}
	\xymatrix{	\ar[r]_{n \rightarrow \infty}&} 
	g=
	\begin{bmatrix}
	0 & x & y\\
	0 & 0 & 0\\
	0 & 0 &0\\
	\end{bmatrix}.
	\] 
	If $x\neq 0$  we get:
	\[
	\begin{array}{ll}
	x& =\lim_{n \rightarrow \infty} 2(a_n+b_nr+b_na)b_n^{-1}=2a,\\
	y&=\lim_{n \rightarrow \infty} (a (2c_n+b_n^2 r)+rb_n (-a-r+1)+ (a_n+b_nr)^2-(a_n+b_nr))b_n^{-1}\\
	&=r (-a-r+1)+a\lim_{n \rightarrow \infty}  (2c_nb_n^{-1}+b_n r).
	\end{array}
	\]
	Thus $y=sa+1-r$ for some $s\in \R$. Therefore:
	\[
	\P^2-\Omega(G)\subset  \overleftrightarrow{e_1,e_2}\cup  \bigcup_{s\in \R} \overleftrightarrow{e_1,[0:sa+1-r:-2a]}\,.
	\]

	Finally, since $\Pi(G)$ is conjugate to a dense subgroup of $\R$ and $\P^2-\Omega(G)$ has  more than two lines  we deduce  $\P^2-\Omega(G)$  contains a pencil of lines over an Euclidean circle. 
\end{proof}


\begin{lemma} \label{p:controlnodiscreto}
	Let $G\subset {\rm Heis}(3,\Bbb{C})$ be a non-Abelian  Kleinian group such that  $Ker(\Pi\vert_G) $ is infinite and    $\Pi(G)$ is non-discrete, then:
	\begin{enumerate}
		\item \label{i:ci1} The set  $\Lambda_{Kul}(Ker(\Pi\vert_G) )$ is a complex line;
		
		\item \label{i:ci2} The set $\Bbb{P}^2_{\Bbb{C}}-\Omega(G)$ contains more than one line;
		\item \label{i:ci3} The group $\Pi(G)$ is conjugate to a subgroup of $\R$;

		\item \label{i:ci4} The rank of the group $\Pi(G)$ is equal to two.
	\end{enumerate}
	
\end{lemma}

\begin{proof}
	Part (\ref{i:ci1}) follows from Lemma  	\ref{l:cif}. Let us prove 
	(\ref{i:ci2}). Since $G$ is non-commutative, there are $x,y,z,a,b,c\in \C$ such that $\{z,c\}$ is  $\R$-linearly dependent  but it is a $\Z$-linearly independent set and also:   $xc-az\neq 0$ and 
	\[
	g=
	\begin{bmatrix}
	1 & x & y\\
	0 & 1 & z\\
	0 & 0 & 1\\
	\end{bmatrix}, h=
	\begin{bmatrix}
	1 & a & b\\
	0 & 1 & c\\
	0 & 0 & 1\\
	\end{bmatrix}
	\in G \,.
	\]
	Since $[g,h]\neq Id$ we can assume $x\neq 0$.  A simple computation shows:
	\[
	g_1=
	\begin{bmatrix}
	\frac{1}{x} & 0 & 0 \\
	0 & 1 & \frac{y}{x} \\
	0 & 0 & z \\
	\end{bmatrix}
	\begin{bmatrix}
	1 & x & y\\
	0 & 1 & z\\
	0 & 0 & 1\\
	\end{bmatrix}
	\begin{bmatrix}
	\frac{1}{x} & 0 & 0 \\
	0 & 1 & \frac{y}{x} \\
	0 & 0 & z \\
	\end{bmatrix}^{-1}
	=
	\begin{bmatrix}
	1 & 1 & 0\\
	0 & 1 & 1\\
	0 & 0 & 1\\
	\end{bmatrix} \,,
	\]
	\[
	h_1=
	\begin{bmatrix}
	\frac{1}{x} & 0 & 0 \\
	0 & 1 & \frac{y}{x} \\
	0 & 0 & z \\
	\end{bmatrix}
	\begin{bmatrix}
	1 & a & b\\
	0 & 1 & c\\
	0 & 0 & 1\\
	\end{bmatrix}
	\begin{bmatrix}
	\frac{1}{x} & 0 & 0 \\
	0 & 1 & \frac{y}{x} \\
	0 & 0 & z \\
	\end{bmatrix}^{-1}
	=
	\begin{bmatrix}
	1 & \frac{a}{x} & \frac{b}{x z}-\frac{a y}{x^2 z} \\
	0 & 1 & \frac{c}{z} \\
	0 & 0 & 1 \\
	\end{bmatrix}\,.
	\]
	Then Part (\ref{i:ci2}) follows  by applying  Lemma \ref{l:2genind}  to the group $\langle g_1,h_1\rangle$.\\

	The proof of (\ref{i:ci3}) is immediate  from Lemma   	\ref{l:controlreal}, so let us 
	prove Part  (\ref{i:ci4}). Assume that $G$ has a non-commutative subgroup  $H$ of type ${\mathcal{K}_4}$ such that $\Pi(H)$ is non-discrete and has rank    3.  Since $H$ is not commutative,  by  Theorem \ref{t:desc}, the previous parts of this lemma and after conjugation, if necessary,  we can find an additive discrete subgroup $W\subset \C$,  $a,b,c,r,s,t\in \C$ such that   $a\neq 0 $,  $\{1,t,c\}$ is  $\R$-linearly dependent   but  $\Z$-linearly independent     and:
	\begin{small}
		\[
		H=
		\left \{
		\begin{bmatrix}
		1 &0  &w\\
		0 & 1 & 0 \\
		0 & 0 & 1 \\
		\end{bmatrix}
		\begin{bmatrix}
		1 &1 &0\\
		0 & 1 & 1 \\
		0 & 0 & 1 \\
		\end{bmatrix}^n
		\begin{bmatrix}
		1 &a+c &b\\
		0 & 1 & c \\
		0 & 0 & 1 \\
		\end{bmatrix}^m
		\begin{bmatrix}
		1 &r+t &s\\
		0 & 1 & t \\
		0 & 0 & 1 \\
		\end{bmatrix}^k
		:k,m,n\in \Z, w\in W
		\right \}.
		\]
	\end{small}
	Since   $H$  is  a group,  this means   $a,r, rc - a t\in W$.  By Kronecker Theorem  \cite[Theorem 4.1]{wal},  $W$ is non-discrete, which is a contradiction. 
\end{proof}	

The  proof of  the following proposition is left to the reader.

\begin{proposition} \label{p:iir}
	Let $G\subset {\rm Heis}(3,\Bbb{C})$ be a non-Abelian complex  Kleinian group such that  $Ker(\Pi\vert_G) $ is infinite and    $\Pi(G)$ is non-discrete. Then
	there are  a rank one  additive discrete subgroup $W\subset \C$,  $a\in W$, and $b,c\in \C$,   such that $\{1,c\}$ is  $\R$-linearly dependent  but   $\Z$-linearly independent     and up to conjugation we have: 
	\[
	G=
	\left \{
	\begin{bmatrix}
	1 &0  &w\\
	0 & 1 & 0 \\
	0 & 0 & 1 \\
	\end{bmatrix}
	\begin{bmatrix}
	1 &1 &0\\
	0 & 1 & 1 \\
	0 & 0 & 1 \\
	\end{bmatrix}^n
	\begin{bmatrix}
	1 &a+c &b\\
	0 & 1 & c \\
	0 & 0 & 1 \\
	\end{bmatrix}^m
	:m,n\in \Z, w\in W
	\right \}\,.
	\]
\end{proposition}

\vskip.2cm

	Now we prove 
	\begin{theorem} \label{Lie-Kolchin}
		Let $G$ be a purely parabolic discrete group in ${\rm PSL} (3,\C)$. Then $G$ is    virtually  finitely presented, torsion free and solvable. Also,   $G$    is  virtually   either      unipotent   and conjugate  to the projectivization of a subgroup of ${\rm Heis}(3,\C)$,  or else it is  
		an  Abelian group of  rank at most two, with an  irrational ellipto-parabolic element, and it is of the form:
		\[
		{Ell}(W,\mu)=
		\left \{
		\left[
		\begin{array}{lll}
			\mu(w) & \mu(w)w&0\\
			0& \mu(w)&0\\ 
			0&0& \mu(w)^{-2}\\ 
		\end{array}
		\right ]
		:w\in W
		\right \},
		\] 
		\noindent where   $W$  is a  discrete additive subgroup of $\C$ and  
		$ \mu:W\rightarrow \Bbb{S}^1$ is  a group morphism.
		
	\end{theorem}
	\begin{proof}
	Let $G$ be a discrete group in $\PSL(3,\C)$ with no loxodromic elements. By 
	Theorem \ref{t:wal} we have that  $G$ contains a finite index subgroup which is conjugate to a group  $G_0$ which is the projectivization of an upper triangular group of matrices. Then 
	Lemma \ref{l:nd} grants the existence of a finite index, torsion free  subgroup  $G_1$ of $G_0$, for which the following groups are all torsion free: $\Pi(G_1)$, $\Pi^*(G_1)$, $\lambda_{12}(G_1)$, $\lambda_{13}(G_1)$ and  $\lambda_{23}(G_1)$.   Now use Theorems  \ref{t:wal} and \ref{t:gcsp}  applied to  $G_1$  and deduce that either  $G_1$ is a  unipotent subgroup of $\Heis(3,\C)$ or  else it is   Abelian  of rank at most 2 of the form stated  in Theorem \ref{Lie-Kolchin}.	
	Now, if 
	$G_2$ is a discrete subgroup of  $\Heis(3,\C)$ then by  Corollary \ref{c:poly} we have that  $G_2$ is solvable and finitely presented. 
\end{proof}

Finally we have:\\

\noindent{\bf Proof of Theorem \ref{t:mainck}:} Let $G_1$ be a subgroup of $G$ which is triangularizable. 	Let  $G_0\subset G_1$ be a subgroup of finite index, such that $G_0,\lambda_{12}(G_0), \lambda_{23}(G_0),\Pi(G_0)$,
$Ker(\Pi\vert_{G_0}) $ are torsion free, see Lemma  \ref{l:nd}. If $G_0$ contains a parabolic element $g$ satisfying  $$
max\{o(\lambda_{12}(g)),o( \lambda_{23}(g))\}=\infty,
$$
then the result follows from Theorem \ref{t:gcsp}. Thus   we can assume  that $G\subset \Heis(3,\C)$.
If $ Ker(\Pi\vert_{G_0})$	is trivial we deduce that $G_0$ is commutative. 	
The proof in this case follows from  Theorem   \ref{l:lt4} and  Lemmas  \ref{l:pfd0},  \ref{l:pfd2}. 	If $Ker(\Pi\vert_{ G_0})$	is non-trivial the result follows from Lemma  \ref{core}  and Propositions  \ref{t:pifa}, \ref{t:pifa1},  \ref{l:abcor2cind} and \ref{p:iir}.
\qed

\subsection{Discrete groups of $ \Heis(3,\C )$ which are not complex  Kleinian}\label{s:dis}

\begin{proposition} \label{p:pifa1d}
	Let  $G\subset {\rm Heis}(3,\C)$, then $G$ is a   discrete non-commutative group  such that     $\Pi(G)$ is discrete and non-trivial  if and only if
	$G$ is conjugate to either 
	$$
	\mathcal{K}=	
	\left \{
	\begin{bmatrix}
	1 &u  &v\\
	0 & 1 & 0 \\
	0 & 0 & 1 \\
	\end{bmatrix}
	\begin{bmatrix}
	1 &1 &x\\
	0 & 1 & 1 \\
	0 & 0 & 1 \\
	\end{bmatrix}^n
	\begin{bmatrix}
	1 &a+c &b\\
	0 & 1 & c \\
	0 & 0 & 1 \\
	\end{bmatrix}^m:m,n\in \Z, (u,v)\in \mathcal{L}
	\right \},
	$$
	or
	$$
	W_{ x,a,b}=	
	\left \{
	\begin{bmatrix}
	1 &u  &v\\
	0 & 1 & 0 \\
	0 & 0 & 1 \\
	\end{bmatrix}
	\begin{bmatrix}
	1 &1 &x\\
	0 & 1 & 1 \\
	0 & 0 & 1 \\
	\end{bmatrix}^n
	:n\in \Z, (u,v)\in W
	\right \}\,,
	$$
	where  $a,b,c,x\in \C$, $c\notin \R$ and $\mathcal{L}\subset \C^2$ is an additive discrete subgroup satisfying that  
	$Span_\Z\{(0,a), (0,\pi_1(\mathcal{L})), (0,c\pi_1(\mathcal{L}))\}\subset  \mathcal{L}$ and $rank(\mathcal{L})\geq 3$.
\end{proposition}
\begin{proof}
	Let us assume that	$G$  is a   discrete non-commutative group  such that     $\Pi(G)$ is discrete  and non-trivial. Without loss of generality let us assume that $\Pi(G)$ has rank two. Then by Theorem \ref{t:desc} 	there are   $a,b,c,x\in \C$, $c\notin \R$ and $\mathcal{L}\subset \C^2$ is an additive discrete subgroup such that $G$ is conjugate to the group:
	$$
	\left \{
	\begin{bmatrix}
	1 &u  &v\\
	0 & 1 & 0 \\
	0 & 0 & 1 \\
	\end{bmatrix}
	\begin{bmatrix}
	1 &1 &x\\
	0 & 1 & 1 \\
	0 & 0 & 1 \\
	\end{bmatrix}^n
	\begin{bmatrix}
	1 &a+c &b\\
	0 & 1 & c \\
	0 & 0 & 1 \\
	\end{bmatrix}^m:m,n\in \Z, (u,v)\in \mathcal{L}
	\right \}\,.
	$$
	It is clear that $G$ acts properly  discontinuously on $\Omega(Ker(\Pi\vert_G))$.  Since $G$ is not complex Kleinian we deduce $rank(\mathcal{L})\geq 3$.  For $w=(u,v)\in \mathcal{L}$ and $k,l,m\in \Z$, define: 
	$$
	g( w,k,l,m)=	
	\begin{bmatrix}
	1 &u  &v\\
	0 & 1 & 0 \\
	0 & 0 & 1 \\
	\end{bmatrix}
	\begin{bmatrix}
	1 &1 &x\\
	0 & 1 & 1 \\
	0 & 0 & 1 \\
	\end{bmatrix}^k
	\begin{bmatrix}
	1 &a+c &b\\
	0 & 1 & c \\
	0 & 0 & 1 \\
	\end{bmatrix}^l \;.
	$$
	Let $w_i=(u_i,v_i)\in  \mathcal{L}$ $(i=1,2)$ and $k,l,m,n\in \Z$; a straightforward computation shows:
	\[
	g( w_1, k,l,)	g( w_2,m,n)^{-1}=g(w_1-w_2+w,k-m,l-n)\,
	\]
	where  $w=u_2 (c l-c n+k-m)-m a (l-n)$. Thus $$Span_\Z\{(0,a),(0,\pi_1(\mathcal{L})),(0,c\pi_1(\mathcal{L}))\} \subset  \mathcal{L}.$$
\end{proof}

The proof of  the following lemma is a slight modification of the proof  of  Part (\ref{l:ica2}) in Lemma   \ref{l:cif},  so we omit it.

\begin{lemma}
	Let $G\subset {\rm Heis}(3,\C)$ be a discrete  non-commutative group such that $\Pi(G)$ is non-discrete. Then $\Lambda_{Kul}(Ker(\Pi\vert_G))$ is a single complex projective line.
\end{lemma}

The next result  is a direct  consequence of Proposition \ref{p:iir} and we include it without proof:
\begin{proposition}  \label{p:rank2}
	Let $G\subset {\rm Heis}(3,\Bbb{C})$ be a non-Abelian discrete but not Kleinian group such that  $Ker(\Pi\vert_G) $ is infinite and    $\Pi(G)$  is a rank two non-discrete group. Then
	we can find a rank two  additive discrete subgroup $W\subset \C$,  $a\in W$, $b,c\in \C$   such that $\{1,c\}$ is  $\R$-linearly dependent  but   $\Z$-linearly independent     and up to conjugation we have: 
	\[
	G=
	\left \{
	\begin{bmatrix}
	1 &0  &w\\
	0 & 1 & 0 \\
	0 & 0 & 1 \\
	\end{bmatrix}
	\begin{bmatrix}
	1 &1 &0\\
	0 & 1 & 1 \\
	0 & 0 & 1 \\
	\end{bmatrix}^n
	\begin{bmatrix}
	1 &a+c &b\\
	0 & 1 & c \\
	0 & 0 & 1 \\
	\end{bmatrix}^m
	:m,n\in \Z, w\in W
	\right \} \,.
	\]
\end{proposition}

\begin{lemma} \label{l:m4}
	Let $G\subset {\rm Heis}(3,\C)$ be a discrete  non-commutative group such that $\Pi(G)$   has rank at least   3, then:  
	\begin{enumerate}
		\item \label{p:2} For every 1-dimensional real subspace $\ell \subset \C$ we have $rank( \ell \cap \Pi(G))\leq 2  $.  
		\item  \label{p:3}  We have $rank(Ker(\Pi\vert_G ))=2$ and $3\leq rank(\Pi(G))\leq 4$.
	\end{enumerate}
\end{lemma}
\begin{proof}
	Let us prove Part (\ref{p:2}). Assume there is a   real line  $\ell \subset \C$ for which  $rank( \ell \cap \Pi(G))\geq 3  $. Then there are  $x\in \C^* $  and $r,s\in \Bbb{R}^*$ such that  $Span_\Z\{1,r,s\}$ is a rank three group  and  $Span_\Z\{x,rx,sx\}\subset \Pi(G)$. Let $d,e,f,g,h,j\in \C$ be such that: 
	\[
	g_1=
	\begin{bmatrix}
	1 & d &e\\
	0 & 1& x\\
	0 & 0& 1\\
	\end{bmatrix},\,
	g_2=
	\begin{bmatrix}
	1 & f &g\\
	0 & 1& rx\\
	0 & 0& 1\\
	\end{bmatrix},\,
	g_3=
	\begin{bmatrix}
	1 & u &v\\
	0 & 1& sx\\
	0 & 0& 1\\
	\end{bmatrix}
	\in  G	\,.
	\]
	A straightforward computation shows that for every $k,l,m,n,o,p\in \Z$:
	\[
	g_1^kg_2^lg_3^{m-p} g_2^{-o} g_1^{-n}g_3^{p-m}g_2^{o-l}g_1^{n-k}=
	\begin{bmatrix}
	1 & 0 &w\\
	0 & 1& 0\\
	0 & 0& 1\\
	\end{bmatrix} \;,
	\]
	where $w=x (d n ((l - o ) r + s (m - p )) - u (m - p) (o r + n ) + 
	f (os (m - p ) + n (-l  + o) ))$. Thus 
	\begin{small}
		\[
		h_1=
		\begin{pmatrix}
		1 & 0 & x (d  s - h  )\\
		0 & 1 & 0\\
		0 & 0 & 1\\
		\end{pmatrix},\, 
		h_2=
		\begin{pmatrix}
		1 & 0 & x (d  r - f )\\
		0 & 1 & 0\\
		0 & 0 & 1\\
		\end{pmatrix},\, 
		h_3=
		\begin{pmatrix}
		1 & 0 & x ( f s-u  r )\\
		0 & 1 & 0\\
		0 & 0 & 1\\
		\end{pmatrix}\in G\,.
		\]
	\end{small}
	Since $f s-u   r= -s(d  r - f )+r(d  s - u )$, we conclude that $Span_\Z\{d  s - u  ,d  r - f  ,f s-u   r \}$ is non-discrete. Thus $G$ is non-discrete, which is a contradiction.\\
	
	Now we prove Part (\ref{p:3}). From  Lemma \ref{l:r2dnkr3d}  and  Proposition   \ref{p:iird3c}  we get  that $rank(Ker(\Pi\vert_G ))=2$; on the other hand, by Theorem \ref{t:desc}  we have $3\leq rank(\Pi(G))\leq 4$.
\end{proof}

\begin{proposition} \label{p:iird3c} Let
	$G\subset {\rm Heis}(3,\Bbb{C})$ be a non-Abelian discrete group. Then      $\Pi(G)$ has rank 3 and is dense in $\C$ if and only if
	we can find a rank two additive discrete subgroup $W\subset \C$,  and $x,a,b,c,d,e,f \in \C$   such that $G$ is conjugate to:
	\begin{small}
		\[
		H=
		\left \{
		\begin{bmatrix}
		1 &0  &w\\
		0 & 1 & 0 \\
		0 & 0 & 1 \\
		\end{bmatrix}
		\begin{bmatrix}
		1 &1 &x\\
		0 & 1 & 1 \\
		0 & 0 & 1 \\
		\end{bmatrix}^k
		\begin{bmatrix}
		1 &a+c &b\\
		0 & 1 & c \\
		0 & 0 & 1 \\
		\end{bmatrix}^l
		\begin{bmatrix}
		1 &d+f&e\\
		0 & 1 & f \\
		0 & 0 & 1 \\
		\end{bmatrix}^m
		:k,m,n\in \Z, w\in W
		\right \},
		\]
	\end{small}
	where $a,b,c,d,e,f$ are subject to the conditions:
	\begin{enumerate}
		\item $\{a,d, af-dc\}\subset W$\,;
		\item $\vert a \vert + \vert d\vert \neq 0 $\,;
		\item for every  real line $\ell\subset \C$ we have $\ell\cap Span_{\Z}\{1,c,f\}$ has rank at most two.
	\end{enumerate}
\end{proposition}
\begin{proof} 
	For $w\in W$ and $k,l,m\in \Z$, define: 
	$$
	g( w,k,l,m)=	
	\begin{bmatrix}
	1 &0  &w\\
	0 & 1 & 0 \\
	0 & 0 & 1 \\
	\end{bmatrix}
	\begin{bmatrix}
	1 &1 &x\\
	0 & 1 & 1 \\
	0 & 0 & 1 \\
	\end{bmatrix}^k
	\begin{bmatrix}
	1 &a+c &b\\
	0 & 1 & c \\
	0 & 0 & 1 \\
	\end{bmatrix}^l
	\begin{bmatrix}
	1 &d+f&e\\
	0 & 1 & d \\
	0 & 0 & 1 \\
	\end{bmatrix}^m \;.
	$$
	Let $u,v\in  W$ and $k,l,m,o,p,q\in \Z$; a straightforward computation shows:
	\[
	g( u,k,l,m)	g( v,o,p,q)^{-1}=g(u-v+w,k-o,l-p,m-q)\,,
	\]
	where $w=-a l n + a o (n + f (m - p)) - d (n + c o) (m - p)$.  Thus $H$ is a group if and only if $a,d, af-dc\in W$. We notice that  if  $H$ is a group, then $H$ is non commutative if and only if $\vert a\vert +\vert d\vert \neq 0$. Clearly  $\Pi(H)$ has rank three and it is a dense group in $\C$ whenever $Span_\Z(\{1,c,d\})$ is dense in $\C$. Observe that if  $H$ is  non-commutative and $Span_\Z\{1,c,f\}$ is dense in $\C$, then $H$   is discrete if and only if  there are sequences $(w_n)\subset W$ and $(k_n),(l_n),(m_n)\subset \Z$ such that $(g_n=g(w_n,k_n,l_n,m_n))$  is a sequence of  distinct elements satisfying $g_n \xymatrix{\ar[r]_{n \rightarrow  \infty}&} Id$. By Lemma \ref{l:r2dnkr3d}, $H$ is non-discrete if and only if $W$ is non-discrete or there are sequences $(w_n)\subset W$ and $(k_n),(l_n),(m_n)\subset \Z$ such that:
	\begin{enumerate}
		\item $(k_n+l_n c + m_n f)$ is a sequence of  distinct elements converging to $0$\,;
		\item $ (l_n a+m_n d)$ converges to $0$\,.
	\end{enumerate} 
	Now observe that 	Lemma \ref{l:addis}  and the previous facts are  equivalent to the non-discreteness of   $Span_\Z(a,d)$.
\end{proof}

Similar arguments show:

\begin{proposition} \label{p:iirnd3c}
	Let 	$G\subset {\rm Heis}(3,\Bbb{C})$ be a non-Abelian discrete group such that     $\Pi(G)$ has rank four.   Then 
	there exist a rank two additive discrete subgroup $W\subset \C$ and  $x,a,b,c,d,e,f ,g,j\in \C$    such that $G$ is conjugate to:
	\[
	H=
	\left \{ g_ug_1^ kg_2^lg_3^mg_4^^n
	:k,l,m,n\in \Z, w\in W
	\right \} \,,
	\]
	where 
	\[
	g_u=\begin{bmatrix}
	1 & 0 & u\\ 
	0& 1& 0\\
	0 & 0 & 1
	\end{bmatrix};\,
	g_1= 
	\begin{bmatrix}
	1& 1& x\\
	0 & 1& 1\\
	0& 0 & 1
	\end{bmatrix};\,
	g_2=
	\begin{bmatrix}
	1 & a + c& b\\
	0 & 1& c\\
	0 & 0& 1
	\end{bmatrix};
	\]
	\[
	g_3=
	\begin{bmatrix}
	1 & d + f& e\\
	0 & 1& f\\
	0 & 0& 1\\
	\end{bmatrix};\,
	g_4=
	\begin{bmatrix}
	1 & g + j& h\\
	0 & 1& j\\
	0 & 0& 1\\
	\end{bmatrix} \,,
	\]
	and	$x,a,b,c,d,e,f,g,h,j$ are subject to the conditions:
	
	\begin{enumerate}
		\item $\{ a,d, g, dj- gf, af-cd, aj - cg \}\subset W$ \,;
		\item  $\vert a\vert+\vert d\vert+\vert g\vert \neq 0$ \,;
		\item for every real line $\ell\subset \C$ we have $\ell\cap Span_\Z(\{1,c,f,j\})$ has rank at most two.
	\end{enumerate}
\end{proposition}

\subsection{Proof of  Theorem \ref{t:maind}}\label{s:maind}

Let $G$ be a discrete group without loxodromic elements which is not (complex) Kleinian. By  Corollary  \ref{c:trian}  and   Lemma \ref{l:nd}, $G$ contains  a torsion free subgroup $G_0$ of finite index   which is triangularizable, it  is not  Kleinian and the following groups are torsion free: $G_0,\lambda_{12}(G_0), \lambda_{23}(G_0),\Pi(G_0)$,
$Ker(\Pi\vert_{G_0}) $.  Theorem \ref{t:gcsp}   implies that     $G_0$  does not contain an irrational ellipto-parabolic element,  for otherwise   the group $G_0$ would be   Kleinian. So    we can assume  that $G_0\subset \Heis(3,\C)$.  

If  $G_0$	is commutative, then Theorem \ref{l:lt4} and Lemma \ref{l:omega} imply that $\Pi(G_0)$ is trivial, and the result follows from  Lemma \ref{l:pfi2a}. 

If $G_0$ is non-commutative, then $\Pi(G_0)$ may or may not be discrete. If 
$\Pi(G_0)$ is discrete, then  the result  follows by  Proposition \ref{p:pifa1d}. 
If  $\Pi(G_0)$ is non-discrete,  then by  Lemma
\ref{l:m4} we have 	that $\Pi(G_0)$ has rank 2, 3 or 4. Let us look at each of these cases:

If the group $\Pi(G)$ has rank two, then   the result follows from   Proposition \ref{p:rank2}.

If the group $\Pi(G)$ has rank three, then  the result  follows  from  Lemma  \ref{l:rc3} and Proposition \ref{p:iird3c}.

If the   group  $\Pi(G)$ has rank four, then  the result  follows from  Proposition  \ref{p:iirnd3c} and Lemma \ref{cent4}.
\qed\\

{\bf  Proof of Theorem  \ref{t:main1}:}
This  uses  Theorems \ref{t:mainck}, \ref{t:maind}, 3,2  and section \ref{s:examples} where the families of purely parabolic groups are described.

Proof of part  (1).   If  $G$ is a complex  Kleinian group then by Theorem \ref{t:mainck}, $G$ is virtually conjugate  to  either an  elliptic group  or to a discrete subgroup  of  ${\rm Heis} (3,\C)$. If  $G$ is discrete but non-Kleinian, then by  Theorem \ref{t:maind} we get  that $G$ is virtually conjugate to a discrete subgroup of   ${\rm Heis} (3,\C)$.

Proof of part (2) items (a)  and  (b).  If  $G$ is complex  Kleinian then, by Theorem \ref{t:mainck},   $G$ is virtually conjugate to one and only one of the following groups: 
  Elliptic, Torus, dual torus group  type I  and  type II,    
 Kleinian Inoue ,  ${\mathcal{K}_0}$, ${\mathcal{K}_1}$	 and ${\mathcal{K}_2}$.	
	By Lemmas \ref{2.1}, \ref{torus groups}, \ref{ladd}, \ref{l:dualtd}, \ref{2.12},  \ref{l. non-Abelian Kodaira}, \ref{217},  \ref{2.18}  and Definition \ref{2.4}  the only groups groups whose limit set is exactly one line are: 
	   Elliptic groups,
		  Torus groups,
		dual Torus type I groups,     
		 ${\mathcal{K}_0}$ groups  and 
		${\mathcal{K}_1}$ groups.	

And by Lemmas \ref{ladd}, \ref{l. non-Abelian Kodaira} , \ref{217}, \ref{2.18},   Theorem \ref{2.7}	 and definition \ref{2.4} those groups whose  limit set  is a cone of lines over a circle are:
	dual torus  type II groups,    
	Kleinian Inoue groups, 
and 	${\mathcal{K}_2}$ groups.	

The proof of  Part (2)  Item (C)  follows directly from Theorem \ref{t:maind}  and the corresponding list  is the following
	 dual torus  group  type III,  
 non  Kleinian Inoue groups,  Extended Inoue groups 
and ${\mathcal{K}_3}$,	
 ${\mathcal{K}_4}$	 and 
	 ${\mathcal{K}_5}$ groups.	
$\square$

	\section{Appendix: Abelian subgroups  of $U^+$}  \label{a:Abeliano} \mbox { }
	
 Since we have not found Theorem \ref{l:lt4} in the literature, we now prove it for completeness. The claim is that if $U_+$ is the group  in Definition  \ref{def $U_+$} and
	 $G\subset U_+$  is a commutative subgroup, then $G$ is conjugate to a subgroup $\widetilde{ G}$ of one of the Abelian Lie groups  $C_j$ in Definition \ref{list of Abelian groups},  for some $j=1,2,3,4,5$.  Notice that 
since  $G$ is commutative, we have that $\Pi^*(G)$ and $\Pi(G)$ are Abelian. Now consider the following cases:\\
		
		\noindent 
		Case 1. The groups  $\Pi^*(G)$ and $\Pi(G)$ contain a parabolic element. In this case, Since $\Pi^*(G),\Pi(G)\subset Mob(\C)$	 are Abelian, we deduce that 
		$\Pi^*(G)$ and $\Pi(G)$ are purely parabolic, {\it i.e}, $G\subset Ker(\lambda_{12})\cap Ker(\lambda_{13})$. \\ 
		
		\noindent
		Claim 1. There is  an element $h\in G$ such that $\Pi(h)$ and 	 $\Pi^*(h)$ are parabolic. Let $g_1,g_2\in G $ be such that $\Pi(g_1)$ and $\Pi(g_2)$ are parabolic,  then, taking a power of $g_2$ if necessary, we can assume that $\Pi(g_1g_2)$ and $\Pi^*(g_1g_2)$ are not the identity. Since $G\subset Ker(\lambda_{12})\cap Ker(\lambda_{13})$ we deduce that $\Pi(g_1g_2)$ and $\Pi^*(g_1g_2)$ are both parabolic.\\
		
		Let $h\in G$ be the element given by the previous claim, then
		
		$$
		h=\begin{bmatrix}
			1 & a &b \\
			0 &1 & c\\
			0 & 0 &1
		\end{bmatrix}
		$$
		where $ac\neq 0$. Let us define  $h_0\in \PSL(3,\C)$ by
		
		$$
		h_0=\begin{bmatrix}
			a^{-1} & 0 &0 \\
			0 &1 & 0\\
			0 & 0 &c
		\end{bmatrix}.
		$$ 
		Then a straightforward computation shows that  for every $g=[g_{ij}]\in h_0Gh_0^{-1}$ we have:
		$$
		[h_0hh_0^{-1},g]
		= 
		\begin{bmatrix}
			1& 0 &-g_{12}+g_{23} \\
			0 &1 & 0\\
			0 & 0 &1
		\end{bmatrix}		 \;.
		$$
		Since $G$ is Abelian we deduce $g_{12}=g_{23}$. 
		
		\vskip.2cm
				\noindent
		Case 2. The group  $\Pi^*(G)$ contains a parabolic element  but  $\Pi(G)$ does not. Under this assumption, we deduce  $\Pi^*(G)$ is purely parabolic and there exists  $w\in \C$ such that $\Pi(G) w=w$, hence  $G\subset Ker(\lambda_{12})$.   We define 
		\[
		h=
		\begin{bmatrix}
			1 & 0 & 0\\
			0 & 1 & w\\
			0 & 0 & 1\\
		\end{bmatrix}.
		\]
		By 	a straightforward computation we show that for  every  $g\in G$, there exists  $c_g\in \C$ such that: 
		\[
		h gh^{-1}=
		\begin{bmatrix}
			g_{11} & g_{12}  & c_g\\
			0           & g_{11}   & 0\\
			0           & 0  & g_{11}^{-2}\\
		\end{bmatrix} \;.
		\]
		We notice  that $G_1=h Gh^{-1}$ leaves invariant the line $\overleftrightarrow{e_1,e_3}$, so  $\Pi_1:G_1\rightarrow Mob(\C) $, given by  $\Pi_1([g_{ij}])=g_{11}g_{33}^{-1}z+g_{13}g_{33}^{-1}$, is a well 
		defined group morphism. Now we only need to consider the following sub-cases:\\

				\noindent
		Sub case 1. The group $\Pi_1(G_1)$ contains a parabolic element. Then  $\Pi_1(G_1)$ is purely parabolic, which shows that $G_1\subset Ker(\lambda_{13})$.\\
		
				\noindent
		Sub case 2. The group $\Pi_1(G_1)$ does not contain a parabolic element. Then there exists $p\in \C$ such that $\Pi_1(G_1) p=p$. We define 
		\[
		h_1=
		\begin{bmatrix}
			1 & 0 & p\\
			0 & 1 & 0\\
			0 & 0 & 1\\
		\end{bmatrix}.
		\]
		It is clear  that for every  $g\in G_1$ we have 
		\[
		h_1h gh^{-1}h_1^{-1}=
		\begin{bmatrix}
			g_{11} & g_{12} &0 \\
			0           & g_{11}   & 0\\
			0           & 0  & g_{11}^{-2}\\
		\end{bmatrix}.
		\]
	It follows  that in this case  the group is conjugate to a subgroup in $C_1$.\\
		
				\noindent
		Case 3.	The group  $\Pi^*(G)$ does not contain a parabolic element  but  $\Pi(G)$ does.  We deduce that $\Pi(G)$ is purely parabolic and there exists  $z\in \C$ such that $\Pi^*(G) z=z$. Clearly  $G\subset Ker(\lambda_{23})$; we    define 
		\[
		h=
		\begin{bmatrix}
			1 & z & 0\\
			0 & 1 & 0\\
			0 & 0 & 1\\
		\end{bmatrix}.
		\]
		Then  for every  $g\in G$ there exists  $c_g$ such that: 
		\[
		h gh^{-1}=
		\begin{bmatrix}
			g_{11}^{-2} & 0 &c_g \\
			0           & g_{11}   & g_{13}\\
			0           & 0  & g_{11}\\
		\end{bmatrix}.
		\]
		Now we can consider $\Pi_2=\Pi_{e_2,\overleftrightarrow{e_1,e_3}} $ and we have that  $\Pi_2(G)\subset Mob(\C)$ is an Abelian group. So we must consider the following sub cases:\\
		
		\noindent
		Sub-case 1. The group $\Pi_2(G)$ contains a parabolic element. We  get   that $\Pi_2(G)$ is purely  parabolic, which shows that $G\subset Ker(\lambda_{13})$.\\
		
		\noindent
		Sub-case 2. The group $\Pi_2(G)$ does not contain a parabolic element. Again there exists  $p\in \C$ such that $\Pi_2(G) p=p$. Define 
		\[
		h_1=
		\begin{bmatrix}
			1 & 0 & p\\
			0 & 1 & 0\\
			0 & 0 & 1\\
		\end{bmatrix}\,.
		\]
		One can show  that for every  $g\in G$:
		\[
		h_1h gh^{-1}h_1^{-1}=
		\begin{bmatrix}
			g_{11}^{-2} & 0 &0 \\
			0           & g_{11}   & g_{13}\\
			0           & 0  & g_{11}\\
		\end{bmatrix}\,.
		\]
		
		\noindent
		Case 4. The groups  $\Pi^*(G)$ and $\Pi(G)$ do not  contain  parabolic elements.  In this setting  there are $z,w\in \Bbb{C}$ such that $\Pi^*(G)z=z$  and $\Pi(G)w=w$. Define 
		\[
		h=
		\begin{bmatrix}
			1 & z & 0\\
			0 & 1 & w\\
			0 & 0 & 1\\
		\end{bmatrix} \,.
		\]
		Then for every  $g=[g_{ij}]\in G$ there exists  $c_g\in \C$ such that: 
		\[
		h gh^{-1}=
		\begin{bmatrix}
			g_{11} & 0  & c_g \\
			0           & g_{22}   & 0\\
			0           & 0  & g_{33}\\
		\end{bmatrix} \,.
		\]
		Consider the following sub-cases:\\
		
		\noindent
		Sub case 1. The group $\Pi_2(G)$ contains a parabolic element. Then  $\Pi_2(G)$ is purely parabolic, which shows that $G\subset Ker(\lambda_{13})$.\\
		
		\noindent
		Sub case 2. The group $\Pi_2(G)$ does not contain a parabolic element.  We know there exists  $p\in \C$ such that $\Pi_2(G) p=p$, let 
		\[
		h_1=
		\begin{bmatrix}
			1 & 0 & p\\
			0 & 1 & 0\\
			0 & 0 & 1\\
		\end{bmatrix}\,.
		\]
		Then the subgroup    $h_1h Gh^{-1}h_1^{-1}$  contains only diagonal elements.
\qed

	\section{Appendix: Technical Lemmas on additive subgroups of $\C^2$ }
	Now we state and prove some technical lemmas used along this work.
	
	\begin{lemma}\label{l:ind}
		Let $\vartheta \in \Bbb{S}^1\setminus\{\pm 1\}$, then $$\beta_1=\{(0,1),(\vartheta,\vartheta),(2\vartheta^2,\vartheta^2), (3\vartheta^3,\vartheta^3)\}\subset \C^2 \,,$$ is an  $\R$-linearly independent set. 
	\end{lemma}
	\begin{proof}
		If $\vartheta =\cos (\theta)+i\sin (\theta) $, then  $\beta_1$ is  $\R$-linearly independent  because  the determinant
		
		\[
		\begin{vmatrix}
			0 & 0 & 1 &0\\
			\cos (\theta) & \sin (\theta) & \cos (\theta) & \sin (\theta)\\
			2\cos (2\theta) & 2\sin (2\theta) & \cos (2\theta) & \sin (2\theta)\\
			3\cos (3\theta) & 3\sin (3\theta) & \cos (3\theta) & \sin (3\theta)\\
		\end{vmatrix}=4(\sin \theta)^4
		\]
		is equal to $0$ if and only if $\vartheta =\pm 1 $.
	\end{proof}

	\begin{lemma} \label{l:latfund}
		Let $\vartheta\in \Bbb{S}^1\setminus\{\pm 1\}$  be a complex number satisfying:
		\begin{enumerate}
			\item The set $$\beta_2=\{(0,1),\vartheta(1,1),\vartheta^2(2,1), \vartheta^3(3,1), \vartheta^4(4,1)\}\subset \C^2,$$
			is  $\Bbb{Q}$-linearly dependent;
			\item  The number $Re(\vartheta)$ is  not a root of the polynomial:
			$$192x^7-64x^6+496x^5+288x^4+510x^3+209x+8.$$
		\end{enumerate}
		Then  there exists    $\alpha\in \C^*$ such that: $$(\alpha,0)\in Span_{\Z}\{\vartheta^{j}(j,1):j\in \{0,\ldots,{5} \}\}.$$
		
	\end{lemma} 
	\begin{proof} 
		Since $\beta_2$ is a $\Bbb{Q}$-linearly dependent set,  there are $m_0,m_1,m_2,m_3,m_4\in \Z$ such that: 
		\[
		m_4\vartheta^4(4,1)=m_3\vartheta^3(3,1)+m_2\vartheta^2(2,1)+m_1\vartheta(1,1)
		+m_0(0,1) \,,
		\]
		and $m_4\neq 0$. Thus we get the following equations:
		\[
		4m_4\vartheta^3=3m_3\vartheta^2+2m_2\vartheta+m_1\,, 
		\] 
		\[
		m_4\vartheta^4=m_3\vartheta^3+m_2\vartheta^2+m_1\vartheta+m_0 \,.
		\] 
		Let us consider $p_1(x), p_2(x)\in \Z[x]$ given by:
		\[
		\begin{array}{l}
			p_1(x)=-4m_4x^3+3m_3x^2+2m_2 x+m_1 \,,\\
			p_2(x)=-m_4x^4+m_3x^3+m_2x^2+m_1x+m_0 \,.
		\end{array}
		\] 
		Clearly $p_1(\vartheta)=p_2(\vartheta)=0$, so there are  $r_1,r_2,r_3\in \R$ such that: 
		\[
		\begin{array}{ll}
			p_1(x)&=-4m_4(x-\vartheta)(x-\vartheta^{-1})(x-r_1)\\
			&=-4m_4x^3+4m_4(2Re(\vartheta)+r_1)x^2-4m_4(1+2r_1Re(\vartheta))x+4m_4r_1 \,,\\
			p_2(x)&=-4m_4(x-\vartheta)(x-\vartheta^{-1})(x^2+r_2x+r_3)\\
			&=-4m_4x^4-4m_4(-2Re(\vartheta)+r_2)x^3-4m_4(1-2r_2Re(\vartheta)+r_3)x^2\\
			&-4m_4(r_2-2Re(\vartheta)r_3)x-4m_4r_3 \,.
		\end{array}
		\] 
		
		By comparing the coefficients of $p_1$ and $p_2$ with the previous equations we get: 
		\begin{equation}\label{e:walde}
			\begin{array}{l}
				m_1=4m_4r_1 \,,\\
				2m_2=-4m_4(1+2r_1Re(\vartheta)) \,,\\
				3m_3=4m_4(2Re(\vartheta)+r_1) \,,\\
				m_0=-4m_4r_3 \,, \\
				m_1=-4m_4(r_2-2r_3Re(\vartheta)) \,, \\
				m_2=-4m_4(1-2r_2Re(\vartheta)+r_3) \,,\\
				m_3=-4m_4(r_2-2Re(\vartheta)) \,.\\
			\end{array}
		\end{equation}
		
		This yields the following linear system:
		\[
		\begin{array}{l}
			r_1+r_2-2r_3Re(\vartheta)=0 \,,\\
			2r_1Re(\vartheta)+4r_2Re(\vartheta)-2r_3=1 \,,\\
			r_1+3r_3=4Re(\vartheta) \,.
		\end{array}
		\]
		Solving the system by Cramer's rule we get:
		\begin{equation}\label{e:sol}
			r_1=\frac{Re(\vartheta)(16Re^2(\vartheta)-7)}{2(1-Re^2(\vartheta))}\,;\quad
			r_2=\frac{-Re(\vartheta)(8Re^2(\vartheta)+3)}{2(1-Re^2(\vartheta))}\,;\quad
			r_3=\frac{4Re^2(\vartheta)-1}{2(1-Re^2(\vartheta))}.
		\end{equation}
		On the other hand,	from the first 3 equations in the System (\ref{e:walde}) we deduce that  $Re(\vartheta)=pq^{-1}$, where $p,q\in \Z$ are co-primes.  Let us define:
		\[
		\begin{array}{l}
			n_0=(4Re^2(\vartheta)-1)q^4 \,,\\
			n_1=(-Re(\vartheta)-16Re^3(\vartheta)))q^4 \,,\\
			n_2=(1+8Re^2(\vartheta)+16Re^4(\vartheta))q^4 \, ,\\
			n_3=(-7Re(\vartheta)-4Re^3(\vartheta))q^4 \,, \\
			n_4=(2-2Re^2(\vartheta))q^4\,.\\
		\end{array}
		\] 
		which are integers.  From Equation (\ref{e:sol})  we deduce: 
		\[
		n_4\vartheta^4(4,1)=n_3\vartheta^3(3,1)+n_2\vartheta^2(2,1)+n_1\vartheta(1,1)
		+n_0(0,1) \;.
		\]
		This implies: 
		\[
		\vartheta^5=n_4^{-2}(n_0n_3+(n_0n_4+n_1n_3)\vartheta+(n_1n_4+n_2n_3)\vartheta^2+(n_2n_4+n_3^2)\vartheta^3) \;.
		\]
		Thus: 
		\[
		\begin{array}{ll}
			n_4^2\vartheta^5(5,1)&=n_0n_3(0,1)+(n_0n_4+n_1n_3)\vartheta(1,1)+(n_1n_4+n_2n_3)\vartheta^2(2,1)\\
			&+(n_2n_4+n_3^2)\vartheta^3(3,1)+(5n_0n_3+4(n_0n_4+n_1n_3)\vartheta+3(n_1n_4+n_2n_3)\vartheta^2\\
			&+2(n_2n_4+n_3^2)\vartheta^3)(1,0) \,.
		\end{array}
		\]
		Finally let us  show that 
		\[
		5n_0n_3+4(n_0n_4+n_1n_3)\vartheta+3(n_1n_4+n_2n_3)\vartheta^2
		+2(n_2n_4+n_3^2)\vartheta^3\neq 0 \;.
		\]
We		define  $p_3(x)=5n_0n_3+4(n_0n_4+n_1n_3)x+3(n_1n_4+n_2n_3)x^2
		+2(n_2n_4+n_3^2)x^3.$ We need to show  $p_3(\vartheta )\neq 0$. Assume, on the contrary, that   $p_3(\vartheta)=0$.
		
		Now we notice: 
		\[
		n_2n_4+n_3^2=q^8(2+63Re^2(\vartheta)+72Re^4(\vartheta)-16Re^6(\vartheta))\neq 0 .
		\]
		Hence  $p_3(x)$    is a cubic polynomial with coefficients in $\Z$.
		Finally, since  $\vartheta$ is a root of $p_3(x)$,  there exists  $r_0\in \Bbb{R}$ such that 
		\[
		\begin{array}{ll}
			p_3(x)&=2(n_2n_4+n_3^2)(x-\vartheta)(x-\vartheta^{-1})(x-r_0)\\
			&=2(n_2n_4+n_3^2) (x^3-(r_0+2Re(\vartheta))x^2+(1+2r_0Re(\vartheta))x- r_0) \,.
		\end{array}
		\]
		By comparing the quadratic coefficients of $p_3$ we obtain:
		\[
		-2(n_2n_4+n_3^2)(2Re(\vartheta)+r_0)=3(n_1n_4+n_2n_3) \,.
		\] 
		Substituting the values of the $n_i$'s we get the following equivalent equation:  
		\[
		192Re(\vartheta)^7-64Re(\vartheta)^6+496Re(\vartheta)^5+288Re (\vartheta)^4+510Re (\vartheta)^3+209 Re(\vartheta)+8=0 \,,
		\]
		which contradicts our initial hypothesis.
	\end{proof}
	
	\begin{lemma}\label{l:rad}
		Let $\mathcal{L}\subset \Bbb{C}^2$ be an additive subgroup such that for each $x,y\in \mathcal{L}$ we have $\pi_1(x)\pi_2(y)=\pi_1(y)\pi_2(x)$, then:
		
		\begin{enumerate} 
			
			\item  \label{i:tt1}  If $Ker(\pi_1)\cap \mathcal{L}$ and $Ker(\pi_2)\cap \mathcal{L}$ are trivial, then there is $\mu\in \C^*$ and $W$ an additive group of $\C$ such that $\mathcal{L}=\{r(1,\mu):r\in W\}$.

			\item \label{i:tt2} If $Ker(\pi_1)\cap \mathcal{L}$ in non trivial, then there  is   an additive group  $W$ of $\C$ such that $\mathcal{L}=\{(r,0):r\in W\}$.
			
			\item \label{i:tt3}If $Ker(\pi_2)\cap \mathcal{L}$ is non-trivial, then there is an additive group   $W$  of $\C$ such that $\mathcal{L}=\{(0,r):r\in W\}$.		
			
		\end{enumerate}
	\end{lemma}
	\begin{proof}
		Let us show (\ref{i:tt1}). Clearly $\mathcal{L}=\{(\pi_1(x),\pi_2(x)): x\in \mathcal{L} \}$. Let us define  $\mu_x=\pi_2(x)/\pi_1(x)$;  by  hypothesis  $\mu_x$ does not depend on $x$, then $$\mathcal{L}=\{(\pi_1(x),\pi_1(x)\mu_x):x\in \mathcal{L}\}.$$

		In order to prove (\ref{i:tt2}) it is  enough to show that $\pi_2(\mathcal{L})$ is trivial.  Assume  on the contrary that there exists  $y\in   \mathcal{L}$ such that $\pi_2(y)\neq 0$. Consider an element  $x\in Ker(\pi_2)\cap \mathcal{L}-\{{\bf 0}\}$, thus $0\neq \pi_1(x)\pi_2(y)=\pi_1(y)\pi_2(x)= 0$, which is a contradiction. The proof of Part (\ref{i:tt3}) is similar.
	\end{proof}

	\begin{lemma} \label{l:ratlat}
		Let $\mathcal{L}=\{(1,0),(c,d)\}\subset \C^2$ be an $\R$-linearly independent set. Then     $(0,1),(0,c)\in Span_\Z(\mathcal{L})$ if and only if  there are  $p,q,r\in \Bbb{N}$ such that $p,q$ are  co-primes, $q^2$ divides $r$,   $c= pq^{-1},$ and  $d=r^{-1}$. 
	\end{lemma}
	\begin{proof}

		Since $(0,1),(0,c)\in Span_\Z(\mathcal{L})$ we deduce that there are $k_1,k_2,k_3,k_4\in \Z$ such that 
		\begin{eqnarray*}
			k_1 +k_2c=0 \,,\\
			k_2d=1\,,\\
			k_3 +k_4c=0\,,\\
			k_4d=c\,.
		\end{eqnarray*}
		From the first two equations 
		we deduce $d=k_2^{-1}$, $c=-k_1k_2^{-1}$. Let $p,q\in \Bbb{N}$ be co-primes such that $c=pq^{-1}$; substituting  in the last two equations we get: 
		\begin{eqnarray*}
			k_3 q+k_4p=0 \,,\\
			k_4q=pk_2 \,.
		\end{eqnarray*}
		From the first equation we see that  $q$ divides $ k_4$, thus there exists  $m\in \Z$ such that $k_4=q m$; substituting  in the last equation we get:
		\[
		mq^2=pk_2.
		\]
		It follows that   $q^2$ divides $k_2$.
		Conversely,  let us assume that $p,q$ are co-primes  such that $c=pq^{-1}$ and $r=q^2 n$, then :
		\begin{eqnarray*}
			-p^2n(1,0)+qpn( pq^{-1},(q^2 n)^{-1})=(0, pq^{-1})\,,\\
			-pqn(1,0)+qqn( pq^{-1},(q^2n)^{-1})=(0,1) \,.
		\end{eqnarray*}
	\end{proof}

	\begin{lemma}\label{l:r2dnkr3d}
		Let  $a,c,d,f\in\C $ be such that 	$\vert a \vert + \vert d\vert \neq 0 $,  	 $Span_{\Z}\{1,c,f\}$ is a rank three group and for every real subspace $\ell\subset \C$  we have  that $Span_{\Z}\{1,c,f\}\cap \ell$ has rank at most two.  Then 
		$Rank(Span_\Z\{a,d, af-dc\})\geq 2$.
	\end{lemma} 
	\begin{proof}
		The result is trivial if $a=0$ or $d=0$ or $ad^{-1}\notin \Bbb{Q}$,  so we  assume that   $0 \ne a=rd $ for some $r\in \Bbb{Q}$.  We  consider  the following cases:
		
		\vskip.2cm
		\noindent
		Case 1.  $Span_{\Z}\{1,c,f\}$ is dense in $\C$. Then we must have $f=s+tc$ where $s,t\in \R$ and $\{1,s,t\}$ is $\Bbb{Q}$-linearly independent. Thus,  $af-dc=(rs+(rt-1)c)d$;  to conclude we observe that  $rs+(rt-1)c\notin \R$.
		
		\vskip.2cm
		\noindent
		Case 2.   $c\in \R-\Bbb{Q}$ and $f\notin  \R$. Then we have $af-dc=(rf-c)d$;   to finish we observe that   $rf-c\notin \R$.
	\end{proof}

	\begin{lemma}\label{l:addis}
		Let  $a,c,d,f\in\C $ be such that 	$\vert a \vert + \vert d\vert \neq 0 $,   $Span_{\Z}\{1,c,f\}$ is dense in $\C$  and  	  there are sequences $(k_n),(l_n),(m_n)\subset \Z$ such that:
		\begin{enumerate}
			\item  $(k_n+l_nc+m_nf) $  is  a sequence  of distinct	  elements converging to $0$;
			\item  $(l_na+m_n d)$ converges to 0.
		\end{enumerate}
		Then  $Span_\Z\{a,d\}$ is non-discrete.
	\end{lemma}
	\begin{proof}
		Assume, on the contrary,     that $Span_\Z\{a,d\}$ is discrete. Then  $ad\neq 0 $ and $qa=pd$ where $p,q$ are  are non-zero  integers. On the other hand, since $Span_\Z\{a,d\}$ is discrete,   by Assumption (2) we have $l_na+m_n d=0$ for $n$ large, so we can assume that $l_n p+m_nq=0$ for $n$ large. Hence:
		$$
		\frac{k_n}{m_n} \xymatrix{\ar[r]_{n \rightarrow  \infty}&} \frac{-fp +cq}{p}\,.
		$$
		That is $-fp +cq\in \R$. On the other hand, since  $f=r+sc$ where $r,s\in\R$ satisfies that $\{1,r,s\}$ is $\Bbb{Q}$-linearly independent, we deduce  $-fp +cq=-(r+sc)p+cq\in \R$. Thus $c\in \R$, which is a contradiction. 
	\end{proof}

	In the next lemma we consider a condition as in Example \ref{H group}:
	\begin{lemma} \label{l:rc3}	Let $a,c,d,f\in \C$ be such that $ a \neq 0$  and for every  real line $\ell\subset \C$ we have that $\ell\cap Span_\Z\{1,c,f\}$ has rank at most two.
		Then 	$U=Span_\Z\{a,d, af-dc\}$ is discrete if and only if  $Span_\Z\{a,d\}$ is discrete and one of the following  statements is true:
		\begin{enumerate}
			\item $d=0$ and $f\notin \R$;
			\item $a=rd$ for some  $r\in \Bbb{Q}$;
			\item $ad^{-1}\notin \R$,
			and    there are $r_1,r_2\in\Bbb{Q}$ such that 
			\[
			c=\frac{a(f-r_1)}{d}-r_2 \,.
			\]
		\end{enumerate}
		
	\end{lemma}
	\begin{proof} It is clear that if  $d=0$  and $f\notin \R$, then $U$ is discrete. So we  assume that 
		if $a=rd$ with $r\in\Bbb{Q}$, then $U=dSpan_\Z\{r,1, rf-c\}$. Since $f=s_1+s_2c$ where $s_1,s_2\in\R$ satisfy that $\{1,s_1,s_2\}$ is $\Bbb{Q}$-linearly independent, then   
		$ rf-c=rs_1+(s_2-1)c\in \R $ if and only if $s_2=1$, which is not possible. Thus $U$ is discrete.  
		Finally 
		observe that if  $ad^{-1}\notin\R$, then the   equation
		\[
		\frac{a}{d}=\frac{c+r_2}{f-r_1}
		\]
		is equivalent to the   discreteness of $U$. 
	\end{proof}
	
	\begin{remark}
		
		We need that the closure of every rank 3 subgroup in  $W=Span_\Z\{1,c,f\}$ be either dense in $\C$ or isomorphic as a Lie group to $\R\oplus \Z$. This comes from the fact that  $W$ is going to play the role of a control group, and we know that  control groups of rank 3 satisfy this. In particular,  there are no  control groups of rank 3 which are dense subgroups and isomorphic to  $\R$ (cf.  \cite{roy}).
		
	\end{remark}
	
	\begin{lemma}\label{cent4}
		Let $a,c,d,g,f, h\in \C$ be such that:
		\begin{enumerate}
			\item   $\vert a\vert+\vert d\vert+\vert g\vert \neq 0$.
			\item   $W=Span_\Z(\{1,c,f,h\})$ is a rank four group. 
			\item 	 For every 1-dimensional real subspace $\ell \subset \C$ we have $rank( \ell \cap W)\leq 2  $.  
			\item 	$\overline {Span_\Z\{1,c,f\}}=\alpha \R\oplus \beta \Z$ where  $\alpha,\beta\in \C^*$ and $\alpha\beta^{-1}\notin \R$.
		\end{enumerate}
		Let $U=Span_{\Z}\{ a,d, g, dh- gf, af-cd, ah - cg \}$, then $U$ is discrete if and only if  $Span_{\Z}\{a, d, g\}$ is discrete, $(\vert a\vert + \vert d\vert)(\vert a\vert + \vert c\vert )(\vert c\vert + \vert d\vert )\neq 0$ and one of the following occurs:
		\begin{enumerate}
			\item  $a=0$
			$($resp. $d=0$, $g=0)$  and there are $r_0,r_1,r_2,r_3\in\Bbb{Q}$ such that $r_1\neq 0$ and
			$$
			(r_2-r_0)^2+4r_1r_3<0;
			$$
			$$
			x_1=\frac{r_2+r_0\pm \sqrt{(r_2-r_0)^2+4r_1r_3}}{2};
			$$
			\[
			x_2=x_3\left (\frac{r_2-r_0\pm \sqrt{(r_2-r_0)^2+4r_1r_3}}{2r_1}\right );\
			\]
			$$
			x_4= (x_5-r_4)\left (\frac{r_2-r_0\pm \sqrt{(r_2-r_0)^2+4r_1r_3}}{2r_1}\right )-r_5
			$$
			where $x_1=c $ $($resp. $f$, $h)$, $x_2=d$  $($resp. $a$, $a)$, $x_3=g$   $($resp. $g$, $d)$, $x_4=f$   $($resp. $c$, $c)$, and $x_5=h$   $($resp. $h$, $f)$.\\

			\item $ad^{-1}\notin \R$ and there are $r_1,r_2,s_1,t_1,s_2, t_2,s_3,t_3 \in \R$ such that: 
			\[
			g=r_1 a+r_2 d \quad ;\quad
			r_2 t_2\neq t_3 \,,
			\]
			\[
			f=\frac{A_2\pm\left(c+ t_2\right)\sqrt{A_1} }{2 \left(r_2 t_2-t_3\right)}\quad ;\quad
			j=\frac{A_3\pm\left(c r_2+t_3\right) \sqrt{A_1} }{2 \left(r_2 t_2-t_3\right)} \;,
			\]
			where:
			\begin{small}
				$$
				\begin{array}{l}
					A_1=\left(-r_2 s_2+r_1 t_2-s_3-t_1\right){}^2-4 \left(r_2 s_1 t_2-r_1 s_2 t_3+r_2 s_2 s_3-r_1 t_1 t_2+s_3 t_1-s_1 t_3\right)\,,\\
					A_2=-c r_2 s_2-c r_1 t_2+c s_3-c t_1+r_2 s_2 t_2-r_1 t_2^2+s_3 t_2-2 s_2 t_3-t_1 t_2\,,\\
					A_3=r_2 \left(c r_1 t_2+s_3 \left(c+2 t_2\right)-c t_1-s_2 t_3\right)+t_3 \left(-r_1 \left(2 c+t_2\right)-s_3-t_1\right)-c r_2^2 s_2\,.\\
				\end{array}
				$$
				
				\item 
				$ad^{-1}\notin \Bbb{Q}$ and  $gd^{-1}\notin \Bbb{Q}$ and  there are 
				$r_2,s_1,s_2,s_3,t_1,t_2,t_3\in \Bbb{Q}$  such that $r_2t_2\neq 0$,
				$
				a=r_2 d;\,\, 
				$ and 
				\[
				c= \frac{1}{2} \left(A_3\mp\sqrt{A_1}\right)\quad ;\quad
				j= \frac{A_2\pm\sqrt{A_1} \left(f+t_1\right)}{2 t_2}\,,
				\]
				where 
				\begin{small}
					$$
					\begin{array}{l}
						A_1=2 r_2 s_2 t_1-4 r_2 s_1 t_2+r_2^2 t_1^2-2 r_2 t_1 t_3-2 s_2 t_3+4 s_3 t_2+s_2^2+t_3^2\,,\\
						A_2=-f r_2 t_1-f s_2+f t_3-r_2 t_1^2-s_2 t_1+2 s_1 t_2+t_3 t_1 \,,\\
						A_3=2 f r_2+r_2 t_1-s_2-t_3\,.\\
					\end{array}
					$$ \end{small}
			\end{small}	
		\end{enumerate}
	\end{lemma}

	\begin{proof} Let us assume $U$ is discrete.  Clearly,  $Span_\Z\{a,d,g\}$ is discrete;  we claim:
		
		\vskip.2cm
		\noindent
		Claim 1. $ \vert d\vert+ \vert g\vert\neq 0$: just notice  that $g=d=0$ implies that $U=aSpan_{\Z}\{ 1, f, h\}$ is not discrete. Similarly one has that 
		$ \vert d\vert+ \vert a\vert\neq 0$  and $ \vert a\vert+ \vert g\vert\neq 0$.
		
		\vskip.2cm
		\noindent
	 Claim 2.  There are not  $r_0,r_1\in \Bbb{Q}$ such that $a=r_0d$, $g=r_1d$.  Assume on the contrary that there are such $r_0,r_1\in \Bbb{Q}$. Set: $$U=d Span_\Z\{ r_0,1, r_1, h- r_1f, r_0f-c, r_0h - cr_1 \}.$$ Let us consider  $$U_2=Span_\Z \{ 1,  h- r_1f, r_0f-c, r_0h - cr_1 \}.$$ Observe  that  $h- r_1f\notin \R$,  for otherwise $h- r_1f\in \Bbb{Q}$ and therefore $\{h,f\}$ are $\Bbb{Q}$-linearly dependent; since $U$ is discrete we conclude that there are $r_1,r_2\in \Bbb{Q}$ such  that $r_0f- c=r_1+ r_2(h- r_1f)$, thus $\{1,c,f,h\}$ is $\Bbb{Q}$-linearly dependent, which is a contradiction.\\
		
		From the previous claims we deduce $adg=0$.  Now let us study the case $a=0$, the cases $d=0$ or $g=0$ are similar and we leave them for the reader.

		\vskip.2cm
		\noindent
		Claim  3.  It is not possible that $adg\neq 0$ and $ad^{-1}\notin \R$.  Assume, on the contrary,    that there are $r_1,r_2,s_1,s_2,s_3,t_1,t_2,t_3\in \Bbb{Q}$ such that: 
		\begin{eqnarray*}
			g=r_1 a+r_2 d,\\
			dh-gf=s_1 a+t_1d,\\
			af-cd=s_2 a+t_2d,\\
			ah-cg=s_3\,.
		\end{eqnarray*}
		Substituting the value of $g$ given in the first equation  in the other  equations we get:
		\[
		\frac{a}{d}=
		\frac{h-fr_2-t_1}{s_1+fr_1}=
		\frac{t_2+c}{f-s_2}=
		\frac{cr_2+t_3}{h-cr_1-s_3}.
		\]
		Hence, we obtain the following system of polynomial equations:
		\begin{eqnarray*}
			(h-fr_2-t_1)(f-s_2)=(s_1+fr_1)(t_2+c),\\
			(h-fr_2-t_1)(h-cr_1-s_3)=(s_1+fr_1)(cr_2+t_3),\\
			(t_2+c)(h-cr_1-s_3)=(f-s_2) (cr_2+t_3).
		\end{eqnarray*}
		A straightforward computation  shows that  this system   has  non-trivial solutions if and only if $r_2 t_2\neq t_3$, and in that case the solutions are:
		\[
		f_\pm=\frac{A_2\pm\left(c+ t_2\right)\sqrt{A_1} }{2 \left(r_2 t_2-t_3\right)}\quad , \quad
		h_\pm=\frac{A_3\pm\left(c r_2+t_3\right) \sqrt{A_1} }{2 \left(r_2 t_2-t_3\right)}\;,
		\]
		where:	
		\begin{small}
			$$
			\begin{array}{l}
				A_1=\left(-r_2 s_2+r_1 t_2-s_3-t_1\right){}^2-4 \left(r_2 s_1 t_2-r_1 s_2 t_3+r_2 s_2 s_3-r_1 t_1 t_2+s_3 t_1-s_1 t_3\right),\\
				A_2=-c r_2 s_2-c r_1 t_2+c s_3-c t_1+r_2 s_2 t_2-r_1 t_2^2+s_3 t_2-2 s_2 t_3-t_1 t_2,\\
				A_3=r_2 \left(c r_1 t_2+s_3 \left(c+2 t_2\right)-c t_1-s_2 t_3\right)+t_3 \left(-r_1 \left(2 c+t_2\right)-s_3-t_1\right)-c r_2^2 s_2\,,\\
			\end{array}
			$$
		\end{small}
		and $\sqrt{A_1}\notin\Bbb{Q}$.
		
		Similarly one finds that it  is not possible to have $adg\neq 0$, $ad^{-1}\in \Bbb{Q}$ and $gd^{-1}\notin \R $.

		\vskip.2cm
		\noindent
		Claim 4. If $a\neq 0$, then  $w=dg^{-1}\notin \R$.  Assume that there is $R\in \Bbb{Q}-\{0\}$ such that $d=Rg$. In this case $U=gSpan_\Z\{ R, 1, Rh- f, cR, c \}$. Let us consider 
		$U_2=Span_\Z\{  1, Rh- f, c \}$, since $U$ is discrete we conclude that there are $R_1,R_2\in \Bbb{Q}$ such  that $Rh- f=R_1+ R_2C$, thus $\{1,c,f,h\}$ is $\Bbb{Q}$-linearly dependent, which is a contradiction.		
		Thus $W=\{d,g,dh-gf,cd,cg\}$. On the other hand there are $r_0,r_1,r_2,r_3\in \R$ such that $c=r_0+r_1w=r_2+r_3w^{-1}$. Therefore 
		\[
		\begin{array}{l}
			cg=r_0g+d r_1,\\
			cd=r_2d+r_3g\,.
		\end{array}
		\]
		Hence $r_0,r_1,r_2,r_3\in \Bbb{Q}$. Since $r_0+r_1w=r_2+r_3w^{-1}$ we conclude that  $w$ is a solution of the polynomial $r_1w^2+(r_0-r_2)w-r_3=0$, that is 
		\[
		w=\frac{r_2-r_0\pm \sqrt{(r_2-r_0)^2+4r_1r_3}}{2r_1}.
		\]
		Finally, since $U$ is discrete we deduce that there are $r_4,r_5\in \Bbb{Q}$ such that $dh-gf=r_4d+r_5g$ which is equivalent to:
		\[
		w=\frac{f+r_5}{h-r_4}.
		\]
	\end{proof}

\section*{Funding}
The research of W. Barrera has been partially supported by Conacyt Proyecto Ciencia de Frontera 2019-21100 via Faculty of Mathematics, UADY and by Conacyt-SNI 45382. The research of  A. Cano  was supported by Conacyt-SNI  104023 and  PAPIIT UNAM  IN110219.
The research of J.P. Navarrete has been partially supported by Conacyt Proyecto Ciencia de Frontera 2019-21100 via Faculty of Mathematics, UADY and by Conacyt-SNI 35874.  The research of  J. Seade  was partially supported by CONACYT-SNI  1170,  PAPIIT IN110517 and  CONACYT project 282937.

\section*{Declarations}
All authors have contributed equally to the paper and they have no conflict of interest.

\section*{Acknowledgments}
The authors would like to thank to the  Cuernavaca Unit of UNAM's Instituto de Matem\'aticas,  to  UADY's Facultad de Matem\'aticas  and
their people, for the hospitality and kindness while working on
this article. We are also grateful to J. F. Estrada,  R. I. García,  N. Gusevskii  and A. Verjovsky  
for fruitful conversations.

\bibliographystyle{amsplain}

\begin{thebibliography}{10}
	
	\bibitem{AB}
	J. M. 	Alonso, M. R.  Bridson, {\it Semihyperbolic groups},  Proc. London Math. Soc. (3) 70 (1995), no. 1, 56–114.
	
	\bibitem{aus}
	L. Auslander, {\it Discrete solvable matrix groups},
	Proc. Amer. Math. Soc. 11 (1960), 687-688.
	
	\bibitem{BCN}
	W. Barrera, A. Cano, J. P.  Navarrete, {\it
		On the number of lines in the limit set for discrete subgroups of $\PSL(3,\mathbb{C})$}, Pacific Journal of Mathematics,
	Vol. 281 (2016), No. 1, 17–49.
	
	
	
	\bibitem{BCN1} 
	W. Barrera, A. Cano, J. P.  Navarrete, 
	{\it  One line complex Kleinian groups},  Pacific Journal of Mathematics, 
	Vol. 272 (2014), No. 2, 275–303.
	
	\bibitem{BCN2}
	W.  Barrera, A. Cano, J.  P.  Navarrete,
	{\it  Subgroups  of}
	$\PSL(3,\Bbb{C})$ {\it with four lines in general position in its limit set}, 
	Conform. Geom. Dyn., 15 (2011),  160-176.
	
	\bibitem{BCN3}
	W.  Barrera, A. Cano, J.  P.  Navarrete,
	{\it The limit set of   subgroups  of}
	$\PSL(3,\Bbb{C})$ ,  Math. Proc. Cambridge Phil. Soc., vol. 150, Issue 01, 2011,   129-146.
	
	\bibitem{BCN4}
	W.  Barrera, A. Cano, J.  P.  Navarrete, {\it On the  number of lines in the limit set  for discrete subgroups of  } $\PSL(3,\C)$, Pacific Journal of Mathematics,  Vol. 281, No. 1, 2016, 17-49.
	
	
	
	\bibitem{BCNS} 
	W.  Barrera, A. Cano, J.  P.  Navarrete, J. Seade, {\em Complex Kleinian groups}. In ``Geometry, Groups and Dynamics", C. S. Aravinda,  W. M. Goldman {\it et al} (eds.), A. M. S. Contemporary Mathematics  639 (2015),1-41.
	
	\bibitem{BCNS1} 
	W.  Barrera, A. Cano, J.  P.  Navarrete, J. Seade, {\em Elementary groups in $\PSL(3,\C)$.} Tentative title, in preparation.
	
	\bibitem{BCNS2} 
	W.  Barrera, A. Cano, J.  P.  Navarrete, J. Seade, {\em On Sullivan's dictionary in complex dimension two}. Tentative title, in preparation.
	
	\bibitem{BCNS3} 
	W.  Barrera, A. Cano, J.  P.  Navarrete, J. Seade, {\em Quotients of open invariant sets in $\P^2$ divided by the action of purely parabolic discrete groups}. Tentative title, in preparation.
	
	
	\bibitem{BGN}
	W.  Barrera, R. García,  J.  P.  Navarrete, {\it Three-dimensional Sol manifolds and complex Kleinian groups}, Pacific Journal of Mathematics, Vol. 294 (2018), No. 1, 1-18.
	
	
	\bibitem{BNU} 
	W.  Barrera, J.  P.  Navarrete, A. Urquiza, 
	{\it Duality of the Kulkarni Limit Set for Subgroups of   $\PSL(3,\C)$}, Bulletin of the Brazilian Mathematical Society, New Series
	June 2018, Volume 49, Issue 2,  261–277. 
	
	
	\bibitem{BBDZ}
	K. A. Bencsath, M. C. Bonanome, M. H. Dean, M. Zyman, {\it Lectures on Finitely Generated Solvable Groups},  SpringerBriefs in Mathematics, 2013.
	
	
	\bibitem{BF}
	M. Bestvina, M. Feighn,  {\it Proper actions of lattices on contractible manifolds}, Invent. math. 150 (2000), 237-256.
	
	\bibitem{BKK}
	M. Bestvina, M. Kapovich, B. Kleiner, {\it Van Kampen’s embedding obstruction for discrete groups}, Invent. math. 150 (2002), 219–235.
	
	
	
	\bibitem{Araceli} A. Bonifant, M.  Lyubich, S. Sutherland (eds.) {\em 
		Frontiers in Complex Dynamics: In Celebration of John Milnor's 80th Birthday}.
	Princeton Mathematical Series 2014.
	
	
	\bibitem{Borel} A. Borel, {\em Groupes lin\'eaires alg\'ebriques}. Ann. Math.  64 (1956), 20-82.
	
	
	
	\bibitem{Bow} B. H. Bowditch.
	{\em Convergence groups and configuration spaces.}  In  ``Geometric group theory down under". Proceedings of a special year in geometric group theory, Cossey, John (ed.) et al. Canberra, Australia, 1996. Berlin: de Gruyter. 23-54 (1999).
	
	\bibitem{CL} 
	A. Cano, L. Loeza, {\it 
		Two-dimensional Veronese groups with an invariant ball},  International Journal of Mathematics, Vol. 28, No. 10, 2017.
	
	\bibitem{CLU} 
	A. Cano, L. Loeza, A. Ucan-Puc, {\em Projective cyclic groups in higher dimensions}. Linear Algebra and its Applications 
	531 (2017) 169-209. 
	
	
	\bibitem{CNS}
	A. Cano, J. P. Navarrete, J. Seade, {\it Complex Kleinian
		Groups}, Birkh\"auser, Progress in Mathematics, Vol. 303, 2013.
	
	
	\bibitem{CS} A.   Cano, J. Seade, {\em On the Equicontinuity Region of Discrete  Subgroups of $\PU(1,n)$}. Journal of Geometric Analysis  20  (2010),  no. 2, 291--305.
	
	\bibitem{CS2} A.   Cano, J. Seade, {\em On Discrete groups of automorphisms of $\mathbb{P}^2_{\mathbb{C}}$}. Geometria Dedicata 168 (2014), 9-60.
	
	
	\bibitem{CPS} A. Cano, J. Parker, J. Seade, {\em 
		Action of R-Fuchsian groups on $CP^2$}.  Asian Journal of Mathematics 20 (2016), p. 449-474.
	
	\bibitem{ChG}
	S. S. Chen and L. Greenberg. {\em Hyperbolic Spaces}. Contributions
	to Analysis, Academic Press, New York,  49-87 (1974).
	
	\bibitem{CG}
	J. P. Conze and Y. Guivarc'h, {\it Limit Sets of Groups of Linear Transformations},
	Sankhyā: The Indian Journal of Statistics, Series A (1961-2002)
	Vol. 62, No. 3, Ergodic Theory and Harmonic Analysis (Oct., 2000), pp. 367-385.

\bibitem{DLS} 	
M.  W. Davis, G. Le, K.  Schreve, {\it
Action dimensions of some simple complexes of groups}	
Journal of Topology 12 (2019),  1266-1314.

\bibitem{Dru-Kapo}
C. Dru\c{t}u, M. Kapovich, {\it Geometric group theory}, American Mathematical Soc., Colloquium Publications, Vol. 63, 2018.  

	\bibitem{Go} W. Goldman,
	{\em Complex Hyperbolic Geometry}. Oxford Science Publications,
	1999.
	
	
	\bibitem{GK}
	G. H. Golub; W.  Kahan, {\it Calculating the singular values and pseudo-inverse of a matrix}. Journal of the Society for Industrial and Applied Mathematics, Series B: Numerical Analysis (1965) 2 (2): 205–22.
	
	
	
	
	
	
	\bibitem{kapovich}
	M. Kapovich. {\em Kleinian Groups in Higher Dimensions}. In  ``Geometry and Dynamics of Groups and Spaces", Kapranov M., Manin Y.I., Moree P., Kolyada S., Potyagailo L. (eds). Progress in Mathematics, vol 265. Birkhäuser Basel, 2007. 
	
	\bibitem{Ku} R. S. Kulkarni, {\em 
		Groups with Domains of Discontinuity}. Mathematische Annalen No. 237 (1978).
	 253-272.
	
		\bibitem{LR}
	J. C. Lennox,  D.  J. S. Robinson {\it 
		The Theory of Infinite Soluble Groups},  Oxford Science Publications, 2007.
	
	
	
	
	\bibitem{Maskit}
	B. Maskit, {\em Kleinian groups}, Springer-Verlag Berlin Heidelberg, 1988.
	
	\bibitem{Mc2}    C. T.  McMullen,    {\em Rational maps and
		Kleinian groups}. In ``Proceedings International Congress of
	Mathematicians",   Kyoto, 1990, p. 889-900, Springer Verlag 1991.
	
	\bibitem{Nav2}   J. P. Navarrete,  {\em   The Trace Function and Complex  Kleinian Groups in $\mathbb{P}^2_{\mathbb{C}}$}. Int. J.
	Math. Vol. 19, No. 7 (2008), 865-890.
	
	\bibitem{Pa-Me}
	J. Palis, W. de Melo,  {\em Geometric Theory of Dynamical Systems: An Introduction}. Springer-Verlag, New
	York (1982).
	
	\bibitem{Poincare} H. Poincar\'e, {\em M\'emoire sur les groupes klein\'eens}. Acta Math., Volume 3 (1883), 49-92.
	
	
	\bibitem{rat}
	J. G. Ratcliffe, {\it Foundations of Hyperbolic Manifolds}, Springer GTM, 2019.
	
	
	\bibitem{roy}
	D. Roy, {\it
		Sur une version algébrique de la notion de sous-groupe minimal relatif de}  $\R^n$,
	Bulletin de la Soci\'et\'e Math\'ematique de France,  Volume 118 (1990) no. 2,  171-191
	
	
	
	
	
	
	
	
	\bibitem{SV1} J. Seade, A. Verjovsky, {\em
		Actions of Discrete Groups on Complex Projective Spaces.}
	Contemporary Mathematics No. 269 (2001), 155-178.
	
	
	
	
	\bibitem{Stein}
	R. Steinberg,  {\em On theorems of Lie-Kolchin, Borel, and Lang}. In ``Contributions to Algebra:  A Collection of Papers Dedicated to Ellis Kolchin". Academic Press (1977),  349-354. 
	
	
	
	
	\bibitem{Su7}   D. P. Sullivan,   {\em Conformal Dynamical
		Systems}.   Lecture Notes in Mathematics, Springer Verlag {1007}
	(1983),  725-752.
	
	
	
	\bibitem{Su4} D. P. Sullivan, {\em Quasiconformal homeomorphisms and dynamics I: Solution of the Fatou-Julia problem on wandering domains}. Annals
	of Math. 122 (1985), 401-418.
	
	\bibitem{Su2} D. P. Sullivan, {\em Seminar on conformal and hyperbolic geometry};
	notes by M. Baker and J. Seade. Institut des Hautes \'Etudes
	Scientifiques, Bures-sur-Yvette, France,  1982.
	
	\bibitem{Suwa} 
	T. Suwa. {\em  Compact quotients of $\C^2$ by Affine transformation Groups}. J. Diff. Geometry, vol. 10 (1975), 239-252.
	
	\bibitem{Tits} J. Tits, {\em Free subgroups in linear groups}. Journal of Algebra. 20 (1972), 250-270.
	
	
	\bibitem{mau}
	G. M. Toledo,   {\em Three Generalizations Regarding Limit Sets
		for Complex Kleinian Groups}, Ph. D. Thesis, CIMAT, México,   2019.
	
	\bibitem{VK}	E. R. Van Kampen, {\em Komplexe in euklidischen R\"aumen}, Abh. Math. Sem. Univ. Hamburg 9 (1933), 72-78 and 152-153.
	
	\bibitem{wal}
	M. Waldschmidt,  {\em Densité des points rationnels sur un groupe algébrique},
	Experiment. Math.
	Volume 3, Issue 4 (1994), 329-352.
		\bibitem{wal1}
	M. Waldschmidt,{\it  
	Topologie des Points Rationnels}, Cours de Troisième Cycle 1994/95, Preprint Univ. P. et M. Curie, 175 p.
Non publié. Mise à jour: 25 Février 1998.  https://webusers.imj-prg.fr/~michel.waldschmidt/articles/pdf/TPR.pdf

	\bibitem{Wu}
	 Wu Wen-Tsün, \it{ A theory of imbedding, immersion, and isotopy of polytopes in a euclidean space}, Science Press, Peking 1965. 
	
	
	
	
	
	
	
\end{thebibliography}

\end{document}